\documentclass[12pt]{amsart}
\usepackage{amsmath, amsthm, amssymb,cite, pgfplots}
\usepackage{fullpage}
\usepackage{hyperref}
\usepackage{xcolor,todonotes}

\numberwithin{equation}{section}
\newtheorem{theorem}{Theorem}
\newtheorem{lemma}{Lemma}
\newtheorem{proposition}[lemma]{Proposition}
\newtheorem{corollary}[lemma]{Corollary}

\theoremstyle{definition}
\newtheorem{remark}[lemma]{Remark}
\newtheorem{definition}[lemma]{Definition}
\numberwithin{lemma}{section}

\newcommand{\mc}{\mathcal}
\newcommand{\mb}{\mathbf}
\newcommand{\mbb}{\mathbb}
\newcommand{\mr}{\mathrm}
\newcommand{\mf}{\mathfrak}

\renewcommand{\Re}{\operatorname{Re}}
\renewcommand{\Im}{\operatorname{Im}}

\DeclareMathOperator{\sgn}{sgn}
\newcommand{\bpf}{\begin{proof}}
\newcommand{\epf}{\end{proof}}

\newcommand{\Rl}{\mbb R}
\newcommand{\C}{\mbb C}

\newcommand{\Strip}{S}
\newcommand{\Schwartz}{\mc S}

\newcommand{\mfH}{{\mf H}}
\renewcommand{\H}{{\mathcal H}}
\newcommand{\bmo}{\mr{bmo}}
\newcommand{\BMO}{\mr{BMO}}

\newcommand{\Til}{\mc T}
\newcommand{\A}{\mb A}

\newcommand{\la}{\langle}
\newcommand{\ra}{\rangle}
\renewcommand{\P}{{\mathbf P}}
\renewcommand{\i}{\mathbf{i}}
\renewcommand{\j}{\mathbf{j}}

\newcommand{\Tilh}{\mc T_h}
\newcommand{\Ph}{\mb P_h}
\newcommand{\Hh}{\mc H_h}

\newcommand{\End}{E^{n,(2)}}
\newcommand{\Ent}{E^{n,(3)}}

\newcommand{\Ennf}{E^{n}_{NF}}
\newcommand{\tEnnf}{\tilde E^{n}_{NF}}

\newcommand{\cK}{{\mathcal K}}

\newcommand{\errw}{\text{\bf err}(L^2)}
\newcommand{\errr}{\text{\bf err}( H^\frac12)}
\newcommand{\norm}{{\mathbf N}}
\newcommand\R{\mathbf R}
\newcommand{\fw}{{\mathfrak{w}}}
\newcommand{\fr}{{\mathfrak{r}}}
\newcommand{\mfa}{{\mathfrak{a}}}
\newcommand{\D}{L}

\newcommand{\high}{{high}}
\newcommand{\low}{{low}}

\DeclareMathOperator{\cosech}{cosech}
\DeclareMathOperator{\sech}{sech}

\newcommand{\M}{\mf M}

\newcommand{\cG}{\mc G}

\newcommand{\tA}{\tilde{A}}
\newcommand{\tB}{\tilde{B}}

\newcommand{\tW}{{\tilde W}}
\newcommand{\tQ}{{\tilde Q}}

\newcommand{\W}{{\mathbf W}}

\newcommand{\lin}{\mr{lin}}

\newcommand{\ao}{a_1}

\newcommand{\esym}{\overset{\text{sym}}{=}}

\newcommand{\K}{\mc K}

\newcommand{\opA}{\mf A}
\newcommand{\opB}{\mf B}
\newcommand{\opC}{\mf C}

\newcommand{\Poly}[1]{\Lambda^{#1}}

\title[Finite depth gravity water waves]{Finite depth gravity water waves in holomorphic coordinates}

\author[B.~Harrop-Griffiths]{Benjamin Harrop-Griffiths}
\address{Courant Institute of Mathematical Sciences, New York University}
\email{benjamin.harrop-griffiths@cims.nyu.edu}
\thanks{The first author was supported by a Junior Fellow award from the Simons Foundation.}

\author[M.~Ifrim]{Mihaela Ifrim}
\address{Department of Mathematics, University of California at Berkeley}
\thanks{The second author was supported by the Simons Foundation.}
\email{ifrim@math.berkeley.edu}

\author[D.~Tataru]{Daniel Tataru}
\address{Department of Mathematics, University of California at Berkeley}
\thanks{The third author was partially supported by the NSF grant DMS-1266182
as well as by the Simons Foundation.}
\email{tataru@math.berkeley.edu}

\begin{document}

\begin{abstract}
  In this article we consider irrotational gravity water waves with
  finite bottom. Our goal is two-fold. First, we represent the
  equations in holomorphic coordinates and discuss the local
  well-posedness of the problem in this context. Second, we consider
  the small data problem and establish cubic lifespan bounds for the
  solutions. Our results are uniform in the infinite depth limit, and match our earlier
infinite depth result in \cite{HIT}.
\end{abstract}

\maketitle
\setcounter{tocdepth}{1}
\tableofcontents

\section{Introduction}
This article is
devoted to the study of the two dimensional finite bottom gravity
water wave equations. Precisely, we consider an inviscid,
incompressible, irrotational fluid evolving in the presence of
gravity. The fluid occupies a time dependent domain $\Omega(t) \subset
\Rl^2$ which has flat finite bottom $\{ y = - h \}$ and a free upper
boundary $\Gamma(t)$ which is asymptotically flat to $y \approx 0$.
The  two parameters in the problem, i.e., the gravity $g$
and the depth $h$, are allowed to be arbitrary positive numbers. 
However, our results are only uniform in the range $g \lesssim h$,
which includes the infinite depth limit but not the zero depth limit.

The fluid evolution is modeled by the incompressible Euler equations
in $\Omega(t)$,
\begin{equation}\label{Euler}
  \left\{
    \begin{aligned}
      & u_t + u \cdot \nabla u = \nabla p - g \j
      \\
      & \text{div } u = 0
      \\
      & u(0,x) = u_0(x).
    \end{aligned}
  \right.
\end{equation}
On the bottom we have the boundary conditions for the velocity, namely
\begin{equation}\label{bottom-bc}
  u \cdot \j = 0,  \qquad y = -h .
\end{equation}
On the free boundary $\Gamma(t)$, on the other hand, we have the
dynamic boundary condition
\begin{equation}\label{dynamic-bc}
  p = 0  \ \ \text{ on } \Gamma (t ),
\end{equation}
and the kinematic boundary condition
\begin{equation}\label{kinematic-bc}
  \partial_t+ u \cdot \nabla \text{ is tangent to } \bigcup_t \Gamma (t).
\end{equation}

%%%%%%%%%%%%%%%%%%%%%
%%% START PICTURE %%%
%%%%%%%%%%%%%%%%%%%%%

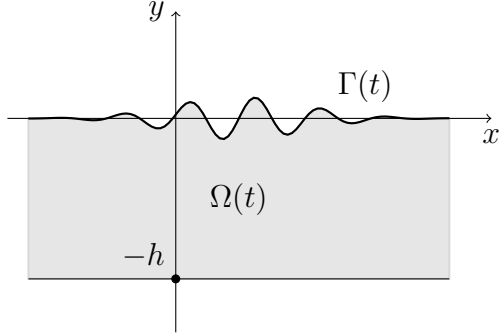
\begin{figure}
  \centering

  \begin{tikzpicture}
    \begin{axis}[ xmin=-12, xmax=12.5, ymin=-2.5, ymax=1.5, axis x
      line = none, axis y line = none, samples=100 ]

      \addplot+[mark=none,domain=-10:10,stack plots=y]
      {0.2*sin(deg(2*x))*exp(-x^2/20)};
      \addplot+[mark=none,fill=gray!20!white,draw=gray!30,thick,domain=-10:10,stack
      plots=y] {-1.5-0.2*sin(deg(2*x))*exp(-x^2/20)} \closedcycle;
      \addplot+[black, thick,mark=none,domain=-10:10,stack plots=y]
      {1.5+0.2*sin(deg(2*x))*exp(-x^2/20)}; \addplot+[black,
      mark=none,domain=-10:10,stack plots=y]
      {-1.5-0.2*sin(deg(2*x))*exp(-x^2/20)};

      \draw[->] (axis cs:-3,-2) -- (axis cs:-3,1) node[left] {\(y\)};
      \draw[->] (axis cs:-11,0) -- (axis cs:12,0) node[below] {\(x\)};
      \filldraw (axis cs:-3,-1.5) circle (1.5pt) node[above left]
      {\(-h\)}; \node at (axis cs:0,-0.75) {\(\Omega(t)\)}; \node at
      (axis cs:6,0.3) {\(\Gamma(t)\)};

    \end{axis}
  \end{tikzpicture}
  \caption{The fluid domain.}
\end{figure}

%%%%%%%%%%%%%%%%%%%
%%% END PICTURE %%%
%%%%%%%%%%%%%%%%%%%

Under the additional assumption that the flow is irrotational, we can
write $u$ in terms of a velocity potential $\phi$ as $u = \nabla
\phi$, where $\phi$ is a harmonic function whose normal derivative is
zero on the bottom.  Thus $\phi$ is determined by its trace \(\psi = \phi|_{\Gamma(t)}\) on the
free boundary $\Gamma (t)$. Denote by $\eta$ the height of the water
surface as a function of the horizontal coordinate. Then the fluid
dynamics can be expressed in terms of a one-dimensional evolution of
the free interface, Precisely, for the pairs of variables
$(\eta,\psi)$ we have
\begin{equation}
  \left\{
    \begin{aligned}
      & \partial_t \eta - G(\eta) \psi = 0 \\
      & \partial_t \psi + g\eta +\frac12 |\nabla \psi|^2 - \frac12
      \frac{(\nabla \eta \cdot\nabla \psi +G(\eta) \psi)^2}{1 +|\nabla
        \eta|^2} = 0 ,
    \end{aligned}
  \right.
\end{equation}
where $G$ represents the Dirichlet to Neumann map on the free boundary
$\Gamma(t)$ associated to the Laplace equation inside the fluid
domain with zero Neumann boundary condition on the bottom.  This is
the Eulerian formulation of the gravity water wave equations.  The
second equation above is known as Bernoulli's law.

While the above Eulerian formulation is easy to write, it is not so convenient to use
due to the presence of the Dirichlet to Neumann map associated to the moving 
domain $\Omega(t)$.  Instead, viewing the choice of the parametrization of the free
boundary as a form of gauge freedom, we employ the 
holomorphic coordinates here. These are obtained using the so-called  conformal method,
where the domain $\Omega(t)$ is viewed as conformally equivalent to a strip.
This will significantly simplify the analysis.

This system has received considerable attention over the years.
The first steps toward understanding the local theory 
were due to Ovsjannikov, see \cite{ov}, who used conformal coordinates
in order to prove local well-posedness in spaces of analytic functions.
Around the same time, in closely related work, Nalimov~\cite{n}  proved the first small 
data result in Sobolev spaces in the infinite depth case. Somewhat later, 
his approach was extended to the finite bottom problem by Yosihara~\cite{y}.

The Eulerian form of the equations, described above, emerged in
\cite{zak,CSS}; however, it was only much later that this led
to a satisfactory local theory.  For a good description of this we
refer the reader to the more recent paper of
Alazard-Burq-Zuily~\cite{abz} as well as to Lannes's book
~\cite{L-book}.

Returning to the conformal method, the evolution equations restricted
to the boundary were independently written by Wu~\cite{wu2} and
Dyachenko-Kuznetsov-Spector-Zakharov~\cite{zakharov}
in the infinite bottom case  in slightly different
forms. Of these, it was  Wu's paper ~\cite{wu2} where this
formulation was fully exploited to prove local well-posedness in the
large data problem. Later, in \cite{CC} Choi and Camassa re-derive the equations for a perfect fluid in the finite depth case when taking both the gravity and capillary force into account. Their method is based on a direct manipulation of the Euler equations, whereas the method of Dyachenko et al. \cite{zakharov} is based on a variational approach. Holomorphic coordinates
have been used subsequently in several other works, for example
\cite{MR2085415,1996PlPhR..22..829D}.

One key feature of this evolution, which led to a very large body of
work, is that it admits soliton solutions, which at low
frequency/small amplitude are close to the KdV solitons. In the
periodic regime these waves are called Stokes waves and
there is a continuous family of such waves up to the maximum height
wave, which has a profile with a 120 degree angle at the top.
As this is only tangentially relevant to the present work,
we simply refer the reader to the recent books of Lannes~\cite{L-book}
for a good description of the KdV approximation, and of Constantin~\cite{C-book}
for the study of solitary waves.

Our goal here is somewhat different, namely  to initiate the study of
the long time dynamics for the small data problem. As mentioned before,
one difficulty in this regard is the presence of the Dirichlet to Neumann map in the
Eulerian formulation of the equations. In order to bypass this
difficulty we consider the equation in holomorphic coordinates, using
a conformal map of the fluid domain into a flat strip. This strategy
was previously implemented by the last two authors is several deep
water scenarios, namely for gravity waves \cite{HIT}, capillary waves
\cite{ITCapillary} and constant vorticity gravity waves
\cite{ITVorticity}.

As this is the first article fully developing the holomorphic
coordinates in the finite bottom scenario, in the first part of the
paper we carefully present the functional setting for our problem, and
then derive the corresponding formulation for the water wave equations in this setting. In this article we only consider the case of the infinite strip.   However, the periodic case is equally interesting, and has received perhaps more  attention in the literature over the years. 

In the holomorphic setting the coordinates are denoted by
$\alpha+i\beta \in \Strip:= \Rl \times (-h,0)$, and the fluid domain 
is parametrized using the conformal map
\[
z: \Strip \to \Omega(t),
\]
which takes the bottom $\Rl - ih$ into the bottom, and the top $\Rl$ 
into the top $\Gamma(t)$. As such, the restriction to the real line
$Z(\alpha) = z(\alpha -  i0)$ can be viewed as a parametrization of 
the free boundary $\Gamma(t)$.

Our variables are the
function $Z(\alpha) = \alpha + W(\alpha)$, which parameterizes the free
surface, and the trace $Q(\alpha)$ of the holomorphic velocity potential on the free surface. Both $(W,Q)$ are what we call here holomorphic
functions, i.e., the trace on the upper boundary $\beta = 0$ of
holomorphic functions in the strip $\Strip$, which are purely real in
the lower boundary $\beta = -1$. The space of holomorphic functions is
a real algebra.

To algebraically describe the space of holomorphic functions we use
the operator $\Til_h$, which is the finite bottom analogue of
the Hilbert transform arising in the description of the Dirichlet to
Neumann map in the canonical domain. Precisely, $\Til_h$ is the
multiplier with symbol $-i \tanh (h\xi)$ and real kernel $-\frac{1}{2h}
\cosech(\frac{\pi}{2h} \alpha)$, interpreted in the principal value
sense.  Then the holomorphic functions are described by the relation
\[
\Im u = - \Til_h \Re u.
\]
The complex conjugates of holomorphic functions will be called
antiholomorphic functions, and are described by the relation $\Im u =
\Til_h \Re u$. Arbitrary functions can be expressed as sums of
holomorphic and antiholomorphic functions,
\[
u = \P_h u + \bar \P_h u.
\]
Here $\Ph$ projects onto the space of holomorphic functions and
its complement ${\bar \P_h = I-\P_h}$ projects onto the space of
antiholomorphic functions. Both can be viewed as orthogonal
projections in the Hilbert space $\mfH_h$ with inner product
\[
\langle u, v\rangle_{\mfH_h} = \int \left( \Til_h \Re u \cdot \Til_h \Re v
  + \Im u \cdot \Im v \right) \, d\alpha.
\]
We note that $\mfH_h$ is not a space of distributions as the $\mfH_h$ norm
does not see real constants. However, it can be viewed as a
quotient space of distributions modulo real constants.

The water wave equations in holomorphic coordinates, derived in
Section~\ref{sect:Derivation}, have the form
\begin{equation}\label{FullSystem-re}
  \begin{cases}
    W_t + F(1+W_\alpha) = 0\vspace{0.1cm}\\
    Q_t + FQ_\alpha - g\Tilh[W] + \P_h
    \left[\dfrac{|Q_\alpha|^2}{J}\right] = 0,
  \end{cases}
\end{equation}
where
\[
J = |1+W_\alpha|^2, \qquad F = \P_h\left[\frac{Q_\alpha-\bar
    Q_\alpha}{J}\right].
\]
We note here that one can freely add real constants to both $W$ and
$Q$; thus these equations are consistent with the low frequency
structure of the space $\mfH_h$.

The above system has a Hamiltonian structure. The Hamiltonian is the
total energy of the system, which is closely related to the above
inner product,
\begin{align*}
  \mc E &= \frac g4\langle W,W\rangle - \frac14\langle
  Q,\Tilh^{-1}[Q_\alpha]\rangle + \frac g2\langle WW_\alpha,W\rangle.
\end{align*}
As written it is not immediately obvious that at low frequency the last term can be controlled by the $\mfH$ norm of $W$. However, a direct computation shows that 
the Hamiltonian can be expressed in the form
\begin{align}
\label{reprezentare}
  \mc E &= \frac g4\langle W,W\rangle - \frac14\langle
  Q,\Tilh^{-1}[Q_\alpha]\rangle +  \frac g2 \int |\Im W|^2 \Re W_\alpha \, d\alpha .
\end{align}
Here one can also see that $\mc E $ remains positive definite for as long as the 
curve $\Gamma (t)$ (i.e., the range of $W+\alpha$) remains non-intersecting. 

For later use here it is will be useful to symmetrize the $Q$ part of the above energy 
by introducing the positive self-adjoint operator 
\[
\D_h = (- \Tilh^{-1} \partial_\alpha)^\frac12,
\]
so that the quadratic part of the energy is given by
\[
E_0(w,r):=g\langle w,w \rangle +  \langle \D_h r, \D_h r \rangle .
\]
It is then natural to look for solutions $(W,Q)$ in the Sobolev space
$\Hh$ with norm
\[
\| (W,Q) \|_{\Hh}^2 := g \|W\|_{\mfH_h}^2 + \|
L_h Q\|_{\mfH_h}^2,
\]
which is similar to $\mfH_h$ for both components at low frequency, and
to $L^2 \times \dot H^\frac12$ at high frequency.

For higher regularity we will use the spaces $\Hh^k = \la D\ra_h^{-k}\Hh$, where \(\la D\ra_h = h^{-1}\la h D\ra\). However, these will
not be applied directly to $(W,Q)$. This is for the same reasons as in
our previous work \cite{HIT}, namely that, after differentiation, the
system for $(W,Q)$ has a degenerate hyperbolic structure, so one needs
to diagonalize it and work with diagonal variables instead. This is a
well known feature of the water wave equation, and we refer the reader
to \cite{abz} and \cite{L-book} for the Eulerian version of this
diagonalization, which is often carried out in a paradifferential
fashion. In our case, as in \cite{HIT}, a convenient choice for the diagonal variables is given by
\[
(\W, R) : = \left(W_\alpha, \frac{Q_{\alpha}}{1+W_\alpha}\right).
\]
These are also physical variables that describe the slope of the free
surface (given by $1+\W$), respectively the fluid velocity of the free
surface.

Indeed, after differentiation one obtains a self-contained diagonal
system in $(\W,R)$:
\begin{equation}\label{DiagonalQuasiSystem-re}
  \begin{cases}
    \W_t + b\W_\alpha + \dfrac{1+\W}{1+\bar \W}R_\alpha = (1+\W)M\vspace{0.1cm}\\
    R_t + b R_\alpha = i\dfrac{g\W - \mfa}{1+\W},
  \end{cases}
\end{equation}
where the double speed (advection velocity) $b$ is given by
\begin{equation}\label{def:b}
b = 2\Re\left[R - \P_h[R\bar Y]\right], \qquad Y = \frac{\W}{1+\W}.
\end{equation}
The other (real) parameters $\mfa$ and $M$ above are given by
\begin{align}
  \mfa =&\ 2\Im\P_h[R\bar R_\alpha] + g(1+ \Til^2_h) \Re \W, \label{def:a}
  \\
  M =&\ 2\Re\P_h[R\bar Y_\alpha - \bar R_\alpha Y]. \label{def:m}
\end{align}

The parameter $\mfa$ also has a physical interpretation, in that $g+\mfa$ is
the normal derivative of the pressure on the free surface. It will be informative to write it in the form
\[
\mfa = a + \ao,
\]
where the quadratic term
\[
a := 2\Im\P_h[R\bar R_\alpha],
\]
remains in the infinite depth limit (see \cite{HIT}) whereas the linear term
\[
\ao := g(1+ \Til^2_h) \Re \W,
\]
is solely a feature of the finite depth case. The positivity of \(g+\mfa\) is also critical as a necessary well-posedness condition
for the above system (the Taylor stability condition):
\begin{equation}\label{TS}
  \left.-\frac{\partial p}{\partial\nu}\right|_{\Gamma(t)} = \frac{g + \mfa}{J} >0.
\end{equation}

The necessity of this condition is not immediately clear from the form of the system
\eqref{DiagonalQuasiSystem-re} above, as this is still a quasilinear
system. However, it will become clear once we consider the linearized
system in Section~\ref{s:linearization}.  In
Section~\ref{sect:TaylorStability} we prove that this positivity
condition remains satisfied as long as the free surface \(\Gamma(t)\)
remains a positive distance above the bottom; this provides 
an alternate,  Fourier-based proof, of the similar result obtained 
in \cite{L-book} in the Eulerian setting using a maximum principle based 
argument. Further, our proof does not depend on the fact that $\Gamma (t)$
is non-intersecting.

In the sequel we will consider solutions $(W,Q)$ for the system
\eqref{FullSystem-re} with the regularity properties
\[
(W,Q) \in \Hh, \qquad (\W,R) \in \Hh^k, \qquad k \geq 1.
\]
To describe the lifespan of these solutions we introduce two control
norms, namely
\begin{equation}
  \label{A-def}
  A := \|\W\|_{L^\infty}+\| Y\|_{L^\infty} + g^{-\frac12} \|\langle D\rangle_h^\frac12 R\|_{L^{\infty} \cap B^{0,\infty}_{2}},
\end{equation}
respectively
\begin{equation}
  \label{B-def}
  B :=g^\frac12 \|\langle D\rangle_h^\frac12 \W\|_{\bmo_h} + \|\langle D\rangle_h R\|_{\bmo_h},
\end{equation}
where, decomposing \(f = f_{<h^{-1}}+ f_{\geq h^{-1}}\) by frequency, the
inhomogeneous space \(\bmo_h\) is given by the norm
\[
\|f\|_{\bmo_h} = \|f_{<h^{-1}}\|_{L^\infty} + \|f_{\geq h^{-1}}\|_{\BMO},
\]
where \(\BMO\) is the usual space of functions of bounded mean
oscillation.

At high frequencies (i.e., larger than $h^{-1}$), these norms coincide with 
the norms in \cite{HIT}.  Here at least for small data $A$ and $B$ are  
controlled by the corresponding Sobolev norms 
of $(W,Q)$ and $(\W,R)$ as follows:
\begin{gather}
\label{sobolev-A}
A \lesssim  g^{-\frac12} \left( \| (\W,R)\|_{\Hh} + h^{-1} \| (W,Q)\|_{\Hh}\right)^\frac12 \left( \| (\W_\alpha,R_\alpha)\|_{\Hh} + h^{-1} \| (\W,R)\|_{\Hh}\right)^\frac12,
\\
\label{sobolev-B}
B \lesssim  \|(\W_\alpha,R_\alpha)\|_{\Hh} +h^{-1} \|(\W,R)\|_{\Hh} + h^{-2} \| (W,Q)\|_{\Hh}.
\end{gather}
For large data some additional care is required due to the need to independently control $Y$ uniformly in $L^\infty$.

Before discussing well-posedness, we remark that as stated the problem
\eqref{FullSystem-re} does not have unique solutions due to the gauge
freedom
\[
(W(t,\alpha),Q(t,\alpha)) \to (W(t,\alpha+\alpha_0(t)) + \alpha_0(t),
Q(t,\alpha+\alpha_0(t)) + q_0(t)),
\]
which corresponds to $F \to F + \alpha'_0(t)$ and a similar choice
involving $q_0'(t)$ for the projector in the second equation.

At the initial time we cannot do more than make an arbitrary choice
(unless we assume more decay at infinity for the initial data).
However, we can fix the choice of $\alpha_0$ and $q_0$ at later times
by requiring that both $F$ and the projector in the second equation
have limit $0$ at $-\infty$. This is allowed because the arguments of
$\Ph$ are not only in $L^2$, but also in $L^1$.

Now we can state our local well-posedness result:

\begin{theorem}\label{t:loc}~

a)   The system \eqref{FullSystem-re} is locally well-posed for all
  initial data $(W_0,Q_0)$ with regularity
  \[
  (W_0,Q_0) \in \Hh, \qquad (\W_0,R_0) \in \Hh^1,\qquad Y_0 \in L^\infty.
  \]
Further, the solutions can be continued as long as our control
parameter $A(t)$ remains finite, and $\displaystyle\int B(t)\, dt$ remains finite.

b) This result is uniform with respect to our choice of parameters $g \lesssim h$
as follows. If for a large parameter $C$ the initial data satisfies 
\[
g^{-1}h^{-1} \| (W_0,Q_0)\|_{\Hh} + g^{-1} \| (\W_0,R_0)\|_{\Hh} + \| (\W_{0,\alpha},R_{0,\alpha})\|_{\Hh} + \| Y_0\|_{L^\infty} \leq C,
\]
then there exists some $T = T(C)$, independent of $g,h$, so that the solution 
exists on $[-T,T]$ with similar bounds.
\end{theorem}

Here well-posedness should be interpreted in the sense of Hadamard as
follows:

\begin{itemize} 
\item Existence of solutions $(W,Q) \in C([-T,T];\Hh)$, $(\W,R)
  \in C([-T,T];\Hh^1)$.

\item Uniqueness of solutions in the same class.

\item Continuous dependence on the initial data in the same topology.

\item Higher regularity: If the initial data has additional regularity
  (e.g. $\Hh^k$) then the solution has additional regularity as well.

\end{itemize}

Our second goal is to establish lifespan bounds for the small data
problem.  Given a generic quasilinear problem with data of size
$\epsilon$ and quadratic interactions, the standard result is to
obtain quadratic lifespan bounds, i.e., $T_{max} \gtrsim
\epsilon^{-1}$. Here we show that for our problem, despite the
presence of quadratic interactions, the lifespan is nevertheless
cubic, i.e., $T_{max} \gtrsim \epsilon^{-2}$.

\begin{theorem} \label{t:long}
  Consider the system \eqref{FullSystem-re} with small initial data
  $(W_0,Q_0)$,
  \[
  g^{-1}h^{-1} \| (W,Q)(0)\|_{\Hh} + g^{-1} \| (\W,R)(0)\|_{\Hh} + \| (\W_\alpha,R_\alpha)(0)\|_{\Hh}  \leq \epsilon .
  \]
  Then the solution $(W,Q)$ exists and satisfies similar bounds on a
  time interval $[-T_\epsilon,T_\epsilon]$ with $T_\epsilon \gtrsim
  \epsilon^{-2}$.  In addition, higher regularity also propagates
  uniformly on the same scale, i.e., for solutions as above we have
  \[
  \| (\W,R) \|_{C([-T_\epsilon,T_\epsilon];\Hh^k)} \lesssim
   \| (\W,R)(0) \|_{\Hh^k} + \epsilon h^{1-k},
  \]
  whenever the right hand side is finite.
\end{theorem}
We emphasize again  that our results are uniform in the range of parameters  $g \lesssim h$.
In particular in the infinite depth limit it agrees with the result in \cite{HIT}.
Furthermore, our setting and our results are also invariant with respect to the scaling 
\begin{equation} \label{scaling}
(W(t,x), Q(t,x)) \to  ( \lambda^{-1} W(t, \lambda x), \lambda^{-1} Q(t, \lambda x)),
\end{equation}
which corresponds to our parameters changing according to the law
\[
(g,h) \to (\lambda g, \lambda h).
\]
Because of this, in the proofs we can freely fix one of the parameters. Precisely, 
after deriving the equations we choose to fix $h = 1$ and work with $g$ in the range $g \lesssim 1$. We will also write \(\Til = \Til_1\) and similarly for other operators and function spaces.

We remark that one can also rescale time for a second degree of freedom in the choice
of the parameters $g$ and $h$. However, our results are not invariant with respect to this second scaling. 

This result formally mirrors  earlier results of the last two authors
 in \cite{HIT} (together with John Hunter) in
the infinite bottom case, as well as \cite{ITCapillary} for infinite bottom 
capillary waves, and \cite{ITVorticity} for constant vorticity gravity waves in deep water.

A common idea in all these papers is the use of the ``quasilinear modified energy method,''
first introduced in \cite{BH}, in order to establish long time bounds.
This can be viewed as  proxy for Shatah's  normal form method \cite{s}, which
cannot be directly implemented in quasilinear problems. Instead
of correcting the quadratic terms in the equation via a normal form transformation,
the basis of our ``quasilinear modified energy method'' is the idea that one can more readily 
modify the energy functional. 

Despite the formal similarity to \cite{HIT}, the analysis here is considerably more 
difficult due to several crucial  differences.
In the infinite bottom case the null condition for resonant quadratic
interactions is satisfied in a stronger form, i.e., the normal form
transformation is nonsingular at zero frequency. Consequently, we are
also able to obtain long time bounds for the linearized equation, and
implicitly for the differences of solutions.  By contrast, only short
time bounds for the linearized equation are obtained in the present
paper.

Another key difference between the two problems has to do with the
existence of solitons, i.e., localized traveling waves. While the
infinite bottom problem admits no solitons, in the finite bottom
problem there are small solitons. This is most readily seen via the
KdV approximation at low frequencies, which is widely discussed in the
literature, see e.g. Lannes's book \cite{L-book}. While these solitons do not play
a significant  role in the present paper, they are expected to be essential
elements of any investigation of the nonlinear shallow water dynamics
on any  longer time scales.

We note that without surface tension these difficulties are essentially unique to the \(2d\) problem and that the additional dispersion in \(3d\) makes the analysis somewhat more straightforward. In the \(3d\) case both enhanced lifespan bounds \cite{MR2372806} and global well-posedness for small, smooth, localized initial data \cite{2015arXiv150806223W,2015arXiv150806227W} have been established previously.

We conclude the introduction with a brief overview of the paper.  We
begin in the next section with a detailed description of the conformal
coordinates, as well the corresponding spaces of ``holomorphic
functions'' where the evolution takes place. The fully nonlinear water
wave system \eqref{FullSystem-re} is derived in
Section~\ref{sect:Derivation}, together with the differentiated quasilinear
system \eqref{DiagonalQuasiSystem-re}.  We also discuss the
Hamiltonian formalism there, as well as the Taylor stability
condition~\ref{TS}.

In Section~\ref{s:model} we study a model linear problem, which captures the quasilinear effects in our problem, but not the quadratic semilinear interactions.
We will subsequently apply the estimates established here to both the linearized and differentiated equations.

The linearized problem is studied in Section~\ref{s:linearization}.
Unlike in our prior work on the infinite bottom problem, here we are 
only able to prove quadratic and not cubic bounds for the linearization.
Thus, the estimates here are only useful for local well-posedness
and not for the cubic lifespan result. 

The study of the long time dynamics begins in the earnest in
Section~\ref{s:nf} with the normal form computation. As one can see there,
the resonant interactions at zero frequency produce a zero frequency
singularity in the the normal form transformation; thus one cannot use
it directly even in the low frequency analysis. In Section~\ref{s:nf-en} we compute the associated normal form energy, where
repeated symmetrizations lead to cancellations of the singular
part. This is the first step in the implementation of our modified
energy method.

In Section~\ref{s:ee} we show 
that the normal form energies admit good quasilinear modifications,
which can be used to prove the long time bounds for the solutions. Finally, our main result is proved in the last section.

Many of the more technical estimates in the paper are relegated to the
Appendix in order to keep the main arguments more clear and
streamlined. This includes a number of Coifman-Meyer type commutator
estimates, as well as their consequences for the various parameters in our water wave 
system.

\section{Holomorphic coordinates}

\subsection{Holomorphic functions in the canonical domain}
We start by considering solutions to the Laplace equation in the strip \(\Strip = \Rl\times(-h,0)\) with mixed boundary conditions,
\begin{equation}\label{MixedBCProblem}
  \begin{cases}
    -\Delta u = 0\qquad \text{in}\ \ \Strip\vspace{0.1cm}\\
    u(\alpha,0) = f\vspace{0.1cm}\\
    \partial_\beta u(\alpha,-h) = 0.
  \end{cases}
\end{equation}

The solution may be written in the form
\[
u(\alpha,\beta) = \frac{1}{\sqrt{2\pi}}\int p(\xi,\beta)\hat
f(\xi)e^{i\alpha\xi}\,d\xi,
\]
where the Fourier multiplier $p$ is given by %
\[
p(\xi,\beta) = \frac{\cosh((\beta+h)\xi)}{\cosh(h\xi)}.
\]
We note that \(p(D,\beta)f\) is well-defined for any
\(f\in\Schwartz'(\Rl)\) and that
\[
\partial_\beta^kp(\xi,\beta) = O(|\xi|^ke^{\frac\beta h\langle h\xi\rangle}).
\]
Given a real-valued solution \(u\) to \eqref{MixedBCProblem} we may
find a harmonic conjugate \(v\) by solving the Cauchy-Riemann
equations,
\[
u_\alpha = v_\beta,\qquad u_\beta = -v_\alpha.
\]
A solution is given by
\[
v(\alpha,\beta) = \frac{1}{\sqrt{2\pi}}\int q(\xi,\beta)\hat
f(\xi)e^{i\alpha\xi}\,d\xi,
\]
where the Fourier multiplier \(q(\xi,\beta)\) is given by
\[
q(\xi,\beta) = \frac{i\sinh((\beta+h)\xi)}{\cosh(h\xi)}.
\]
On the boundary \(\{\beta=0\}\) we have
\[
v(\alpha,0) = -\Tilh f(\alpha),
\]
where the Tilbert transform is
\[
\Tilh f(\alpha) =
-\frac{1}{2h}\lim\limits_{\epsilon\downarrow0}\int_{|\alpha-\alpha'|>\epsilon}
\cosech\left(\frac\pi{2h}(\alpha-\alpha')\right)f(\alpha')\,d\alpha',
\]
is given by the Fourier multiplier \(-i\tanh(h\xi)\). We remark that it
takes real-valued functions to real-valued functions. We denote the
inverse Tilbert transform by \(\Tilh^{-1}\). As discussed above there
is some ambiguity in its definition. For concreteness we define it to
be given by the Fourier multiplier \(i\coth(h\xi+i0)\) such that
\(\Tilh^{-1}f\) vanishes at \(-\infty\) whenever \(f\in L^1\cap L^2\).

We will call functions on the line holomorphic if they
are the restriction to the real line of holomorphic functions in the
strip and satisfy the boundary condition on the bottom. This consists of functions $u$ which satisfy
\[
\Im u = - \Tilh \Re u,
\]
and forms a real algebra as can be seen from a simple
application of the product formula
\begin{equation}\label{SummationFormula}
  u\Tilh[v] + \Tilh[u]v = \Tilh\big[uv - \Tilh[u]\Tilh[v]\big],
\end{equation}
which follows from the corresponding identity for \(\tanh\xi\). The complex conjugates of holomorphic functions are called
antiholomorphic.

\subsection{Sobolev spaces}
On the space of all complex valued functions we define the real inner product
\begin{equation}\label{InnerProduct}
\la u,v \ra = \frac12\Re\int (1-\Tilh^2) u \cdot \bar v - (1+\Tilh^2) u \cdot v \, d \alpha ,
\end{equation}
where we note that $- \Tilh^2 $ is a non-negative operator.
The corresponding Hilbert space is denoted by $\mfH_h$.
Its norm can be rewritten in the form
\[
\| u\|_{\mfH_h}^2 = \int \left(\Tilh\Re u \cdot \Tilh\Re u + \Im u \cdot \Im u \right) \, d\alpha ,
\]
where one can easily see that this is non-negative, and thus a norm.

We denote by $\mfH_h^{(h)}$, respectively $\mfH_h^{(a)}$ the subspaces of $\mfH_h$
consisting of holomorphic, respectively antiholomorphic functions.
The interesting observation, which is in effect the motivation for
our introducing the space $\mfH_h$, is that its holomorphic and antiholomorphic 
subspaces are orthogonal complements of each other. We remark that, restricted to either $\mfH_h^{(h)}$ or $\mfH_h^{(a)}$, the $\mfH_h$ norm can be rewritten as
\[
\| u\|_{\mfH_h}^2 = \int \left(|u|^2 - \frac12u^2 - \frac12\bar u^2 \right)\,d \alpha .
\]

We will also need the associated orthogonal projections, which are denoted
by $\Ph$, respectively $\bar \P_h$. These are operators which are conjugated 
via the standard complex conjugation. We can define these two operators
in two equivalent ways. In a real fashion, we can set
\[
\Ph u = \frac12 \left[(1 - i \Tilh) \Re u + i (1+ i\Tilh^{-1}) \Im u           \right] ,
\]
\[
\bar \P_h u = \frac12 \left[(1 + i \Tilh) \Re u + i (1- i\Tilh^{-1}) \Im u           \right].
\]
In a complex fashion, we can write
\[
\begin{split}
\Ph u = & \  \frac14\left[ ( 2 - i \Tilh + i \Tilh^{-1}) u - i(\Tilh + \Tilh^{-1}) \bar u \right]  
\\ = & \ \frac14\left[ ( 1 - i \Tilh)(1 + i \Tilh^{-1}) u + (1-i\Tilh)( 1 - i\Tilh^{-1}) \bar u \right] ,
\end{split}
\]
respectively
\[
\begin{split}
\bar \P_h u = & \ \frac14\left[ ( 2 + i \Tilh - i \Tilh^{-1}) u + i(\Tilh + \Tilh^{-1}) \bar u \right] 
\\ =  & \ \frac14\left[ ( 1 + i \Tilh)(1 - i \Tilh^{-1}) u - (1-i\Tilh)( 1 - i\Tilh^{-1}) \bar u \right] .
\end{split}
\]

\subsection{Conformal mappings}\label{sect:ConformalMap}
Given a fluid domain $\Omega =\Omega(t)$ with upper boundary $\Gamma
=\Gamma(t)$ with a prescribed Sobolev regularity, and lower boundary $y = -h$, our goal here is to obtain a conformal map
\[
z: \Strip \to \Omega
\]
with similar regularity. Here we do not assume that $\Gamma$ is a
graph, only that it is the upper boundary of a simply connected domain
$\Omega$ which admits a parametrization with a suitable Sobolev
regularity. Precisely, we represent the boundary $\Gamma(t)$ as a
parametrized curve
\[
s \to z(s)
\]
with the following properties:

\begin{itemize}
\item[(i)] Sobolev regularity: $z(s) - s \in H^k_h := \la D\ra_h^{-k}L^2$.

\item[(ii)] Nondegenerate and non-intersecting: The map $s \to z(s)$
  is surjective and nondegenerate, $z'(s) \neq 0$.

\item[(iii)] Does not touch the bottom: $\Im z > - h$.
\end{itemize}
Then we have:

\medskip
\begin{theorem}~

  a) Let $\Omega$ be a simply connected domain whose lower boundary
  consists of the line $\Im z = -h$ and whose upper boundary is a
  curve $\Gamma$ as above, with $k > \frac32$. Then there exists a
  conformal map
  \[
  z: \Strip \to \Omega
  \]
  taking the line $\beta = - h$ into itself and the line $\beta = 0$
  into $\Gamma$.  Further, the restriction of $z$ to the upper
  boundary $\beta = 0$ has the regularity $z-\alpha \in \mfH^k_h$ and is
  unique up to horizontal translations.

  b) If in addition $\Gamma$ admits a parametrization which satisfies the smallness
 condition
\[
h^{-\frac32}( \| z\|_{L^2} + h^k \| z\|_{H^k_h}) \ll 1 ,
\] 
then it is a graph $y = y(x)$ satisfying similar $H^k_h$ bounds, and the following
  norms are comparable:
  \[
  h^{-j}\|y\|_{L^2}+ \| y\|_{H^j_h} \approx  h^{-j}\| z-\alpha \|_{L^2}+ \| z-\alpha\|_{\mfH^j_h}, \qquad 0 \leq j \leq k.
  \]
\end{theorem}
\medskip

\begin{remark}
  We remark here on a minor downside to the use of holomorphic
  coordinates in the strip, namely that there is no canonical way to
  remove the horizontal translation symmetry (unless $z(s)-s$ has some
  $L^1$ integrability perhaps). We address this issue dynamically in
  our study of the water wave equations. Precisely, we make an
  arbitrary choice at the initial time, but we define a unique way to
  propagate this choice to later times.
\end{remark}
\medskip

\begin{proof}
  By rescaling it suffices to assume \(h=1\). To clarify the geometric context, we note that the $L^2$
  integrability condition on the parametrization guarantees that
  outside a compact set, the boundary $\Gamma$ is the graph of a small
  $H^s$ function.

  It is easier to construct the inverse map
  \[
  \Omega \ni z \to \zeta \in S.
  \]
  For this we begin with the function $\beta(z)$, which is defined as
  the unique bounded solution to the elliptic boundary value problem
  \[
  \left\{
    \begin{aligned}
      &\Delta_{x,y} \beta = 0 \qquad \text{in} \ \ \Omega
      \\
      & \beta(x,-1) = -1
      \\
      & \beta (x,y) = 0 \qquad \text{on} \ \ \Gamma .
    \end{aligned}
  \right.
  \]
  Maximum principle type arguments show that $\beta$ is of class $C^1$
  in $\Omega$, and also that it has no critical points. Since $\Gamma$
  is asymptotically flat, it also easily follows that
  \[
  \lim_{x \to \pm \infty} \nabla \beta(x,y) = (0,1).
  \]
  Once we have the function $\beta$, its harmonic conjugate $\alpha$
  is determined via the Cauchy-Riemann equations, and satisfies
  \[
  \lim_{x \to \pm \infty} \frac{\alpha(x,y)}{x} = 0.
  \]
  It is clear that $\alpha$ is uniquely determined up to constants.

  The generated map $ z \to \alpha + i \beta$ will then be a
  diffeomorphism from $\Omega$ to $S$. It remains to establish the
  regularity properties of this map restricted to $\Gamma$, and then
  of its inverse.

  Our goal here is to show that the map
  \[
  s \mapsto \frac{d \alpha}{ds}
  \]
  (which so far is bounded, continuous and nonzero) has the regularity
  \begin{equation} \label{da-ds} \frac{d \alpha}{ds} - 1 \in H^{k-1}.
  \end{equation}
  As $k -1 > \frac12$, inverting we also have
  \[
  \frac{ds}{d\alpha} -1 \in H^{k-1} .
  \]
  Hence by the chain rule we get
  \[
  \frac{dz}{d\alpha} -1 \in H^{k-1}, \qquad \Im z(\alpha) \in H^k,
  \]
  as desired.

  To prove \eqref{da-ds} we use the Cauchy-Riemann equations to
  rewrite this in terms of the normal derivative of $\beta$, namely
  \[
  \frac{d \alpha}{ds} = \frac{dz}{ds} \cdot \frac{d\beta}{d \nu}.
  \]
  Hence we still need to show that
  \[
  \frac{d\beta}{d \nu} - 1 \in H^{k-1}.
  \]
  The function $\beta - y$ solves the Laplace equation in $\Omega$
  with $H^k$ Dirichlet data on the top $\Gamma$ and zero Dirichlet
  data on the top. In addition, $\Gamma$ also has $H^k$ regularity
  (which implies also $C^1$ as $k \geq \frac32$.  Then we want its
  normal derivative on $\Gamma$ to be in $H^{k-1}$. But this follows
  from standard elliptic theory; for an exposition of this which
  exactly fits the strip type of domains here we refer the reader to
  Chapter 3 of Lannes's book \cite{L-book}.

\end{proof}

%%%%%%%%%%%%%%%%%%%%%%%%%%%%%%%%%%%
%%% DERIVATION OF THE EQUATIONS %%%
%%%%%%%%%%%%%%%%%%%%%%%%%%%%%%%%%%%

\section{Derivation of the equations}\label{sect:Derivation}

\subsection{Derivation of the fully nonlinear system}
In this section we derive the fully nonlinear system
\eqref{FullSystem-re} from the Euler equations \eqref{Euler}, and the
boundary conditions \eqref{bottom-bc}, \eqref{dynamic-bc} and
\eqref{kinematic-bc}.

We start by defining the holomorphic function \(w\) by
\[
w(t,\alpha,\beta) = z(t,\alpha,\beta) - (\alpha+i\beta),
\]
where \(z = x + iy\colon \Strip\rightarrow\Omega(t)\) is the conformal
map constructed in Section~\ref{sect:ConformalMap}
As \(z\) is holomorphic we have the Cauchy-Riemann equations
\[
x_\alpha = y_\beta,\qquad x_\beta = -y_\alpha.
\]

Let \(\phi(t,x,y)\) be the velocity potential in Euclidean coordinates
and take the potential in holomorphic coordinates to be
\[
\psi(t,\alpha,\beta) = \phi(t,x(t,\alpha,\beta),y(t,\alpha,\beta)).
\]
We take \(\theta\) to be the harmonic conjugate of \(\psi\) and
define \(q = \psi+i\theta\). Applying the chain rule, the velocity \(u
= \nabla\phi\) is given by
\begin{equation}\label{uvExp}
  u = \frac1j(x_\alpha\psi_\alpha + x_\beta\psi_\beta,y_\alpha\psi_\alpha+y_\beta\psi_\beta),
\end{equation}
where the Jacobian $j$ has the form
\[
j = x_\alpha y_\beta - x_\beta y_\alpha = x_\alpha^2 + y_\alpha^2.
\]

In this section we will use capital letters to denote the trace of functions on the boundary \(\{\beta =
0\}\). In particular, by a slight abuse of notation, we will write \(Y(t,\alpha) = y(t,\alpha,0)\). We then have that \(W(t,\alpha) = w(t,\alpha,0)\) and
\(Q(t,\alpha) = q(t,\alpha,0)\) are holomorphic and hence
\begin{equation}\label{Rel}
  Y = -\Tilh[X - \alpha],\qquad Y_\alpha = -\Tilh[X_\alpha],\qquad \Theta = -\Tilh\Psi.
\end{equation}

We observe that \(1-Z_\alpha^{-1} =
\dfrac{W_\alpha}{1+W_\alpha}\) is holomorphic, so by comparing real and
imaginary parts we obtain
\begin{equation}\label{FracRel}
  \frac{Y_\alpha}{J} = \Tilh\left[\frac{X_\alpha}{J} - 1\right] = \Tilh\left[\frac{X_\alpha}{J}\right],
\end{equation}
where \(J(t,\alpha) = j(t,\alpha,0) = |1+W_\alpha|^2\).

Using the normal \((-Y_\alpha,X_\alpha)\) to the free boundary we
write the kinematic boundary condition \eqref{kinematic-bc} in the
form
\[
(X_t,Y_t)\cdot(-Y_\alpha,X_\alpha) = U\cdot(-Y_\alpha,X_\alpha),
\]
where \(U(t,\alpha) = u(t,\alpha,0)\) is the restriction of the
velocity to the free boundary. Using the expression \eqref{uvExp} for
the velocity in holomorphic coordinates and the Cauchy-Riemann
equations we simplify the right hand side to obtain
\begin{equation}\label{Kinematic-}
  X_\alpha Y_t - Y_\alpha X_t = -\Theta_\alpha.
\end{equation}
Using \eqref{Rel} and \eqref{FracRel} we write this in the form
\[
\frac{X_\alpha}{J}\Tilh[X_t] + \Tilh\left[\frac{X_\alpha}{J}\right]X_t =
\frac{\Theta_\alpha}{J}.
\]
Applying the product formula \eqref{SummationFormula} to the left-hand side we
obtain
\begin{equation}\label{Kinematic+}
  \Tilh\left[\frac{X_\alpha X_t+Y_\alpha Y_t}{J}\right] = \frac{\Theta_\alpha}{J}.
\end{equation}

Combining \eqref{Kinematic-} and \eqref{Kinematic+} we solve for
\(X_t,Y_t\) to obtain
\[
\begin{cases}
  X_t = \dfrac{\Theta_\alpha}{J}Y_\alpha + \Tilh^{-1}\left[\dfrac{\Theta_\alpha}{J}\right]X_\alpha\vspace{0.1cm}\\
  Y_t = -\dfrac{\Theta_\alpha}{J}X_\alpha +
  \Tilh^{-1}\left[\dfrac{\Theta_\alpha}{J}\right]Y_\alpha.
\end{cases}
\]
In terms of the holomorphic function \(W = (X - \alpha) + iY\) we have
\[
W_t = X_t + iY_t =
-i(1+i\Tilh^{-1})\left[\frac{\Theta_\alpha}{J}\right](1+W_\alpha).
\]
If we define
\[
F = \Ph\left[\frac{Q_\alpha-\bar Q_\alpha}{J}\right],
\]
then we may write this in the form
\begin{equation}\label{FullSystem-part1}
  W_t + F(1+W_\alpha) = 0.
\end{equation}

Next we use \eqref{Euler} to obtain the Bernoulli equation with
dimensionless gravitational constant \(g>0\),
\begin{equation}\label{Bernoulli}
  \phi_t + \frac12|\nabla\phi|^2 + gy + p = 0.
\end{equation}
From the dynamic boundary condition \eqref{dynamic-bc} we have
\(p=0\).

Applying the chain rule and Cauchy-Riemann equations we obtain
\begin{gather*}
  \phi_t|_{\{\beta=0\}} = \Psi_t - \frac1J(X_\alpha X_t+Y_\alpha Y_t)\Psi_\alpha - \frac1J(Y_\alpha X_t-X_\alpha Y_t)\Theta_\alpha,\\
  \frac12|\nabla\phi|^2|_{\{\beta=0\}} =
  \frac{1}{2J}(\Psi_\alpha^2+\Theta_\alpha^2).
\end{gather*}
Using the relations \eqref{Kinematic-} and \eqref{Kinematic+}, we
simplify the first of these to obtain
\[
\phi_t|_{\{\beta=0\}} = \Psi_t -
\Tilh^{-1}\left[\frac{\Theta_\alpha}{J}\right]\Psi_\alpha
-\frac{1}{J}\Theta_\alpha^2.
\]
This leads to the equation
\[
\Psi_t - \Tilh^{-1}\left[\frac{\Theta_\alpha}{J}\right]\Psi_\alpha
+\frac{1}{2J}(\Psi_\alpha^2 - \Theta_\alpha^2) + gY = 0.
\]
We write this in terms of \(Q = \Psi+i\Theta\) by applying \( \Ph
\) to obtain
\[
Q_t - \Ph
\left[\Tilh^{-1}\left[\frac{\Theta_\alpha}{J}\right]\Psi_\alpha
  +\frac{\Theta_\alpha^2}{J}\right]
+\Ph\left[\frac{1}{2J}(\Psi_\alpha^2 + \Theta_\alpha^2)\right] -
g\Tilh[W] = 0.
\]
An application of the product formula \eqref{SummationFormula} gives
us
\[
\Tilh\left[\Tilh^{-1}\left[\frac{\Theta_\alpha}{J}\right]\Psi_\alpha +
  \frac{\Theta_\alpha^2}{J}\right] =
\frac{\Theta_\alpha}{J}\Psi_\alpha -
\Tilh^{-1}\left[\frac{\Theta_\alpha}{J}\right]\Theta_\alpha,
\]
which leads to the equation
\begin{equation}\label{FullSystem-part2}
  Q_t + FQ_\alpha  - g\Tilh[W] + \Ph\left[\frac{|Q_\alpha|^2}{J}\right] = 0.
\end{equation}
Combining \eqref{FullSystem-part1} and \eqref{FullSystem-part2} we
obtain at the fully nonlinear system \eqref{FullSystem-re}.

\medskip
%

%%%%%%%%%%%%%%%%%%
%%% SYMMETRIES %%%
%%%%%%%%%%%%%%%%%%

\subsection{Symmetries}
Besides the gauge freedom, the
system \eqref{FullSystem-re} has a number of symmetries:

\smallskip

\begin{itemize}
\item[(i)] \emph{Translation.} The equations are invariant under time
  and space translations, for \((t_0,\alpha_0)\in\Rl^2\)
  \[
  (W(t,\alpha),Q(t,\alpha))\mapsto (W(t +
  t_0,\alpha+\alpha_0),Q(t+t_0,\alpha+\alpha_0)).
  \]
\item[(ii)] \emph{Reflection.} We have a horizontal reflection
  symmetry given by
  \[
  (W(t,\alpha),Q(t,\alpha))\mapsto (-\bar W(t,-\alpha),\bar
  Q(t,-\alpha)).
  \]
\item[(iii)] \emph{Time reversal.} We have a time reversal symmetry
  given by
  \[
  (W(t,\alpha),Q(t,\alpha))\mapsto (W(-t,\alpha),-Q(-t,\alpha)).
  \]
\item[(iv)] \emph{Galilean invariance.} The system has a Galilean
  invariance, for \(c\in\Rl\)
  \[
  (W(t,\alpha),Q(t,\alpha))\mapsto (W(t,\alpha - ct),Q(t,\alpha - ct)
  - c\left((\alpha - ct) + W(t,\alpha - ct)\right) + \frac12 c^2t).
  \]
  However, as our choice of spaces require \(R\) to vanish at
  \(\pm\infty\) we break the Galilean symmetry as in terms of \((\W,R)\)
  the Galilean shift corresponds to the map
  \[
  (\W(t,\alpha),R(t,\alpha)) \mapsto(\W(t,\alpha-ct),R(t,\alpha-ct) -
  c).
  \]
\end{itemize}

%%%%%%%%%%%%%%%%%%%%%%%%%%%%%
%%% HAMILTONIAN STRUCTURE %%%
%%%%%%%%%%%%%%%%%%%%%%%%%%%%%

\subsection{Hamiltonian structure and conserved quantities}
If the free surface is given by \(y = \eta(x)\), then the energy of the
system in Euclidean coordinates is given by
\[
\mc E(\eta,\phi) = \frac g2\int_\Rl |\eta|^2\,dx + \frac12 \int_\Rl\int_{-h}^{\eta(x)}|\nabla
\phi|^2 \,dydx.
\]
We may write this in terms of the holomorphic variables \((W,Q)\) as
\[
\mc E(W,Q) = \frac g4 \langle W,W\rangle - \frac 14\langle
Q,\Tilh^{-1}[Q_\alpha]\rangle + \frac g2\langle WW_\alpha,W\rangle.
\]
We note that the additional factor of \(\frac12\) appears here due to the use of the complex-valued functions.

It was first observed by Zakharov \cite{zak} that the water wave system is a
Hamiltonian equation with Hamiltonian \(\mc E\). To see this we
consider the space of holomorphic functions \((W,Q)\in\H_h\)
equipped with the inner product
\begin{equation}\label{Squirrel}
\left\la \begin{bmatrix}W_1\\Q_1\end{bmatrix},\begin{bmatrix}W_2\\Q_2\end{bmatrix}\right\ra := \frac g2\langle W_1,W_2\rangle
+ \frac12\langle L_hQ_1,L_hQ_2\rangle.
\end{equation}
With respect to this inner product we have
\[
d\mc E(W,Q) = \begin{bmatrix}W + WW_\alpha - \Tilh^{-1}\Ph\left[\bar
    W\Tilh[W_\alpha]\right]\\ Q \end{bmatrix}.
\]

We claim that the system \eqref{FullSystem-re} may then be written in
the form
\begin{equation}\label{eq:Hamster}
  \begin{bmatrix}W_t\\Q_t\end{bmatrix} = \begin{bmatrix}0 & \opA \\ \opC & \opB\end{bmatrix}d\mc E(W,Q),
\end{equation}
where the operators $\opA$, $\opB$ and \(\opC\) are given by
\begin{align*}
\opA[w] &:= - (1+W_\alpha)\Ph\left[\frac{w_\alpha - \bar w_\alpha}{J}\right],\\
\opB[q] &:= - Q_\alpha\Ph\left[\frac{q_\alpha - \bar q_\alpha}{J}\right] -
\Ph\left[\frac {\Ph[\bar Q_\alpha q_\alpha] + \bar\P_h[Q_\alpha\bar q_\alpha]}{J}\right],
\\
\opC[w] &:= g \Ph\left[\frac{\Ph\left[(1 + \bar W_\alpha)\Tilh[w]\right] + \bar\P_h\left[(1 + W_\alpha)\Tilh[\bar w]\right]}{J}\right].
\end{align*}
Taking \(\opA^*\) to be the adjoint of \(\opA\) with respect to the inner product on the space of holomorphic functions in \(\mfH_h\), we apply Lemma~\ref{lem:Adjoint} to obtain
\[
L_h^2\opC[w] = -g\opA^*[w],\qquad L_h^2 \opB[q] = - \opB^*[L_h^2 q],
\]
and hence the matrix operator
\[
\begin{bmatrix}0 & \opA \\ \opC & \opB\end{bmatrix},
\]
is skew-adjoint with respect to the inner product \eqref{Squirrel}. This skew-adjoint matrix is the representation in our setting of the symplectic form for the finite bottom system.

We now prove \eqref{eq:Hamster}. 
We first note that
\[
\opA[Q] = - F(1+W_\alpha),\qquad \opB[Q] =
- FQ_\alpha - \Ph\left[\frac{|Q_\alpha|^2}{J}\right].
\]
It remains to consider the term involving \(\opC\),
which we may write in the form
\[
g \Ph\left[\frac{\Ph\left[(1 + \bar W_\alpha)\Tilh[w] -
      W_\alpha\Tilh[\bar w]\right]}{J}\right], \qquad 
      w = W + WW_\alpha - \Tilh^{-1}\Ph\left[\bar
    W\Tilh[W_\alpha]\right].
\]

Given holomorphic functions \(u,v\) we may write them in terms of their real parts and apply the product
formula \eqref{SummationFormula} to obtain the identity
\begin{equation}\label{Gerbil}
\Ph\big[\Tilh[uv] - \bar u\Tilh[v] - \Tilh[\bar
  u]v\big] = \Tilh[u]v.
\end{equation}
Taking \(u = W\) and \(v = W_\alpha\) we may apply this identity to the quadratic part of the numerator to obtain
\[
\Ph\left[\Tilh[WW_\alpha] - \bar W\Tilh[W_\alpha] - W_\alpha\Tilh[\bar W] + \bar W_\alpha \Tilh[W]\right] = W_\alpha\Ph[\Tilh[W]] + \Ph\left[\bar W_\alpha\Tilh[W]\right].
\]
Next we consider the cubic part of the numerator. Here we apply both the identity \eqref{Gerbil} and its complex conjugate with \(u = W\) and \(v = W_\alpha\) to obtain
\[
\Ph\left[\bar W_\alpha\Tilh[WW_\alpha] - \bar W_\alpha\bar
    W\Tilh[W_\alpha] - W_\alpha\Tilh[\bar W\bar W_\alpha] +
    W_\alpha\bar\P_h[ W\Tilh[\bar W_\alpha]]\right]
  = W_\alpha \Ph\left[\bar W_\alpha\Tilh[W]\right].
\]
Combining these with the linear part \(\Ph[\Tilh[W]]\) we obtain
\[
\opC\left[W + WW_\alpha - \Tilh^{-1}\Ph[\bar W\Tilh[W_\alpha]]\right] =
g\Ph\left[\frac{(1+W_\alpha)\Ph[(1+\bar W_\alpha)\Tilh[W]]}{J}\right] =
g\Tilh[W],
\]
where the second equality follows from the fact that \(\dfrac{1+W_\alpha}{J} = \dfrac{1}{1+\bar W_\alpha}\) is antiholomorphic and hence we may discard the inner projection operator. This completes the proof of \eqref{eq:Hamster}.

As the system is invariant under translation,
\(\alpha\mapsto\alpha+c\), via Noether's principle there will be a
corresponding conserved quantity. This is the horizontal momentum,
\[
\mc I(W,Q) = - \frac12\langle L_h W,L_h Q\rangle =  \frac12\langle W,\Til^{-1} Q_\alpha\rangle .
\]
With respect to the above inner product on \(\H_h\) we
have
\[
d\mc I(W,Q)
= - \begin{bmatrix}g^{-1}L_h^2Q\\W\end{bmatrix}.
\]
A further calculation gives us that
\[
\begin{bmatrix}W_\alpha\\Q_\alpha\end{bmatrix}
= \begin{bmatrix}0&\opA\\\opC&\opB\end{bmatrix}d\mc I(W,Q).
\]

%%%%%%%%%%%%%%%%%%%%%%%%
%%% TAYLOR STABILITY %%%
%%%%%%%%%%%%%%%%%%%%%%%%

\subsection{Positivity of the normal derivative of the
  pressure}\label{sect:TaylorStability} As discussed above, a
necessary condition for the well-posedness of \eqref{FullSystem-re} is
the Taylor stability condition \eqref{TS}. In this section we first
derive the expression for the normal derivative of the pressure in
holomorphic coordinates, and then show that it remains positive for as
long as the free surface remains a positive distance away from the
bottom. We remark that an alternate proof of this property, using the 
maximum principle, can be found in Lannes~\cite{L-book}. Our proof here,
based on a sum of squares representation, provides a different 
insight into this problem.

From the Bernoulli equation \eqref{Bernoulli} we may write the normal
derivative of the pressure as
\[
-\left.\frac{\partial p}{\partial \nu}\right|_{\Gamma(t)} = \frac
1J\partial_\beta\left.\left(\phi_t + \frac12|\nabla\phi|^2 +
    gy\right)\right|_{\{\beta=0\}}.
\]
Using the Cauchy-Riemann equations we obtain
\[
\partial_\beta\phi_t|_{\{\beta=0\}} =
-\partial_\alpha\left(\Tilh\left[\frac{\Psi_\alpha^2 +
      \Theta_\alpha^2}{2J}\right] + g\Tilh[Y]\right).
\]
A further application of the Cauchy-Riemann equations yields
\(g\partial_\beta y|_{\{\beta=0\}} = gX_\alpha\), so
\[
-J\left.\frac{\partial p}{\partial \nu}\right|_{\Gamma(t)} = g +
\frac12(\partial_\beta -
\Tilh\partial_\alpha)(|\nabla\phi|^2)|_{\{\beta=0\}} + g((X_\alpha - 1) -
\Tilh[Y_\alpha]).
\]

Next we define the holomorphic velocity,
\begin{equation}\label{HolomorphicVelocity}
  r := \phi_x - i\phi_y = \frac{q_\alpha}{1+w_\alpha}.
\end{equation}
From \eqref{uvExp} we see that
\( |\nabla\phi|^2 = |r|^2 \),
and as \(r\) is holomorphic we obtain
\[
\frac12\left.(\partial_\beta -
  \Tilh\partial_\alpha)(|r|^2)\right|_{\{\beta=0\}} = 2\Im\Ph[R\bar R_\alpha] = a.
\]
Further,
\(g\left((X_\alpha - 1) - \Tilh[Y_\alpha]\right) = g(1 + \Tilh^2)\Re\W = \ao.  \)
As a consequence,
\begin{equation}
-J\left.\frac{\partial p}{\partial \nu}\right|_{\Gamma(t)} = g + \mfa.
\end{equation}

We now show that the Taylor stability condition \eqref{TS} is
satisfied whenever the free surface \(\Gamma(t)\) is a positive
distance away from the bottom \(\{y = -1\}\):

\medskip
\begin{lemma}\label{L:Taylor}
  Assume that $(W,Q) \in \Hh$ are holomorphic, with $\Im W \geq c >
  -h$.  Then we have the pointwise bound
  \begin{equation}
    g+ \mfa \geq g(c+h).
  \end{equation}
\end{lemma}

\begin{proof}
Using the spatial scaling discussed in the introduction, it suffices to assume that \(h = 1\). We recall the expression of $\mfa$,
  \[
  \mfa = 2 \Im \P[R \bar R_\alpha] + g(1+\Til^2) \Re \W .
  \]
  We will consider the two terms separately, and prove that
  \begin{equation} \label{pn+}
  (1+\Til^2) \Re \W \geq c, \qquad \Im \P[R \bar R_\alpha] \geq 0.
  \end{equation}

  For the first of these, we write it in terms of $\Im W$ as follows:
  \[
  (1+\Til^2) \Re \W = - (1+\Til^2) \partial_\alpha \Til^{-1} (\Im W).
  \]
  The multiplier on the right has symbol
  \[
  m(\xi) = 2\xi\cosech2\xi.
  \]
  As a consequence we may write
  \[
  (1+\Til^2) \Re \W = \int K(\alpha-\alpha')\Im W(\alpha')\,d\alpha',
  \]
  where
  \[
  K(\alpha) = \frac{1}{\sqrt{2\pi}}\check m(\alpha) =
  \frac{\pi}{8}\sech^2(\frac{\pi}{4}\alpha)
  \]
  is non-negative, Schwartz and has integral \(1\). Then the first part
  of \eqref{pn+} follows.

  For the second part we begin by writing
  \[
  2 \Im \P [ R \bar R_\alpha] = -\frac{i}2 \left[(1-i\Til) (R \bar
    R_\alpha) - (1+i \Til) (\bar R R_\alpha) \right].
  \]
  Hence in Fourier space we have the representation
  \[
  \widehat{2 \Im \P [ R \bar R_\alpha]}(\zeta) = \int_{\xi-\eta =
    \zeta} \hat R(\xi) \bar{ \hat R}(\eta) K(\xi,\eta) \, d\eta,
  \]
  where the kernel $K$ is given by
  \[
  K(\xi,\eta) = -\frac12\left((\xi+\eta) +
    (\xi-\eta)\tanh(\xi-\eta)\right).
  \]

  As in \cite{HIT}, a natural idea might be to obtain a sum of squares
  representation of the above integral. Naively, we could seek a
  decomposition of the kernel as below
  \[
  K(\xi,\eta) = \int f_N(\xi) f_N(\eta) \, dN.
  \]
  However, here we have the additional information that $R$ is
  holomorphic, which naively allows us to estimate integrals mostly
  concentrated where $\xi,\eta > 0$ by symmetric integrals
  concentrated where $\xi,\eta < 0$.  To eliminate this constraint we
  write everything in terms of the real part of $R$, which is an
  arbitrary function:
  \[
  \hat R(\xi) = (1 - \tanh \xi) \widehat{\Re R}(\xi).
  \]
  Then the kernel $K$ is replaced by
  \[
  K_1(\xi,\eta) = (1 - \tanh \xi)(1 - \tanh \eta)K(\xi,\eta).
  \]
  Further, for real functions we have the symmetry
  \[
  \hat f(-\xi) = \bar{\hat{f}}(\xi),
  \]
  so the above kernel can be further replaced by
  \[
  K_2(\xi,\eta) = \frac12 (K_1(\xi.\eta) + K_1(-\xi,-\eta)).
  \]
  We compute
  \[
  \begin{split}
    K_2(\xi,\eta)
    =&\ \frac12(\tanh \xi + \tanh \eta)(\xi+\eta) - \frac12(1 + \tanh \xi \tanh
    \eta)(\xi-\eta)\tanh(\xi-\eta)
    \\
%     =&\ \frac12(\tanh \xi + \tanh \eta)(\xi+\eta) - \frac12(\tanh\xi-\tanh\eta)(\xi-\eta)\\
%     &\ -(\xi - \eta)\tanh(\xi-\eta)\tanh\xi\tanh\eta
%     \\
    =&\ \tanh \xi \tanh \eta \left (\frac{\xi}{\tanh \xi} +
      \frac{\eta}{ \tanh \eta} - (\xi-\eta) \tanh(\xi-\eta)\right ).
  \end{split}
  \]

  On the other hand the following expression gives the symbol of a
  pointwise non-negative form
  \[
  \begin{split}
    I (\xi,\eta) := & \ \int (1+\tanh N) ( 1+
    \tanh(\xi-N))(1+\tanh(\eta-N)) \, dN
    \\
    = & \ \xi\left(1+\frac{1}{\tanh \xi}\right)\left( 1 +
      \frac{1}{\tanh(\eta-\xi)}\right) + \eta\left(1+\frac{1}{\tanh
        \eta}\right)\left( 1 + \frac{1}{\tanh(\xi-\eta)}\right),
  \end{split}
  \]
  and after symmetrization
  \[
  \frac12\left(I(\xi,\eta) + I (-\xi,-\eta)\right) = \frac{\xi}{\tanh
    \xi} + \frac{\eta}{\tanh \eta} - \frac{\xi - \eta}{\tanh (\xi
    -\eta)}.
  \]
  Then we can write
  \[
  K_2(\xi,\eta) = \frac12 \tanh \xi \tanh \eta \left(I(\xi,\eta) + I
    (-\xi,-\eta)\right) + \tanh \xi \tanh \eta \, K_3(\xi,\eta),
  \]
  where
  \[
  K_3(\xi,\eta) = K_3(\xi-\eta):= 2(\xi - \eta)\cosech(2(\xi - \eta)).
  \]

  The quadratic form determined by the first term in \(K_2\) is  non-negative. On the other hand for the second term we take advantage of its translation invariance to write
  \[
  K_3(\xi-\eta)= \int g(\xi-N) g(\eta - N) \,dN,
  \]
  with even, real-valued $g$. Indeed, taking the Fourier transform we get
  \[
  \hat g^2 = \frac{1}{\sqrt{2\pi}}\hat K_3,
  \]
  or equivalently
  \[
  \hat g(\alpha)^2 = \frac\pi8 \sech^2(\frac\pi4 \alpha).
  \]
  The right hand side is non-negative and its square root is a Schwartz
  function. This suffices for our purposes, and yields the desired
  representation for $K_3$.

\end{proof}

%%%%%%%%%%%%%%%%%%%%%%%%%%
%%% QUASILINEAR SYSTEM %%%
%%%%%%%%%%%%%%%%%%%%%%%%%%

\subsection{Derivation of the quasilinear system} In this section we derive the quasilinear system \eqref{DiagonalQuasiSystem-re} for the holomorphic variables  
\[
(\W,R) = \left( W_\alpha, \frac{Q_\alpha}{1+W_\alpha} \right),
\]
where we recall that \(R \) is the restriction of the holomorphic velocity \eqref{HolomorphicVelocity} to the free boundary.

As we expect mixed holomorphic-antiholomorphic terms to be lower order than purely holomorphic terms, we first introduce the (real-valued) advection velocity
\[
b = 2\Re\Ph\left[\frac{Q_\alpha}{J}\right],
\]
so that \(F = b-\dfrac{\bar Q_\alpha}{J}\). We note that our earlier gauge fixing procedure corresponds to fixing the real constant in $b$ so that  
\[
\lim_{\alpha \to -\infty} b(t,\alpha) = 0.
\]
Evidently the similar condition at positive infinity does not need to hold.

Differentiating \eqref{FullSystem-re} we obtain a self-contained quasilinear system in \((W_\alpha,Q_\alpha)\),
\[
\begin{cases}
W_{\alpha t} + bW_{\alpha\alpha} + b_\alpha(1+W_\alpha) =\left[\dfrac{\bar Q_\alpha}{1+\bar W_\alpha}\right]_\alpha\vspace{0.1cm}\\
Q_{\alpha t} + bQ_{\alpha\alpha} + b_\alpha Q_\alpha - g\Tilh[W_\alpha] = \bar\P_h\left[\dfrac{|Q_\alpha|^2}{J}\right]_\alpha.
\end{cases}
\]
As \(b=\bar b\) we have
\[
b_\alpha = \bar F_\alpha + \frac1J\left(Q_{\alpha\alpha} - \frac{Q_\alpha W_{\alpha\alpha}}{1+W_\alpha} - \frac{Q_\alpha\bar W_{\alpha\alpha}}{1+\bar W_\alpha}\right).
\]
Grouping the highest order terms on the left hand side we obtain
\[
\begin{cases}
W_{\alpha t} + bW_{\alpha\alpha} + \dfrac{1}{1+\bar W_\alpha}\left(Q_{\alpha\alpha} - \dfrac{Q_\alpha W_{\alpha\alpha}}{1+W_\alpha} \right) = -\bar F_\alpha(1 + W_\alpha) + \dfrac{Q_\alpha\bar W_{\alpha\alpha}}{(1+\bar W_\alpha)^2} + \left[\dfrac{\bar Q_\alpha}{1+\bar W_\alpha}\right]_\alpha\vspace{0.1cm}\\
Q_{\alpha t} + bQ_{\alpha\alpha} + \dfrac{Q_\alpha}{J}\left(Q_{\alpha\alpha} - \dfrac{Q_\alpha W_{\alpha\alpha}}{1+W_\alpha} \right)  - g\Tilh[W_\alpha] = -\bar F_\alpha Q_\alpha + \dfrac{Q_\alpha^2\bar W_{\alpha\alpha}}{J(1+\bar W_\alpha)} + \bar\P_h\left[\dfrac{|Q_\alpha|^2}{J}\right]_\alpha .\!\!\!
\end{cases}
\]

As in the infinite depth case, in order to obtain favorable estimates at high frequency we must diagonalize this system. To do this we define the operator
\begin{equation}\label{Diagonalizer}
\A(w,q) = (w,q - Rw),
\end{equation}
 Taking \(\W = W_\alpha\) we use the diagonal variables
\[
(\W,R) = \A(W_\alpha,Q_\alpha).
\]
We calculate
\[
R_\alpha = \frac{1}{1+W_\alpha}\left(Q_{\alpha\alpha} - \dfrac{Q_\alpha W_{\alpha\alpha}}{1+W_\alpha}\right),
\]
and obtain the equation
\begin{equation}\label{DQ-part1}
\W_t + b\W_\alpha + \frac{1+\W}{1+\bar \W}R_\alpha = \left(\dfrac{R_\alpha}{1+\bar\W} - b_\alpha\right)(1 + \W) + \bar R_\alpha.
\end{equation}
Defining
\[
M = \frac{R_\alpha}{1+\bar\W} + \frac{\bar R_\alpha}{1+\W} - b_\alpha,
\]
we obtain the first part of \eqref{DiagonalQuasiSystem-re}.

For the second part of \eqref{DiagonalQuasiSystem-re}, we first write
\begin{equation}\label{DQ-part2}
Q_{\alpha t} + bQ_{\alpha\alpha} + \dfrac{1+\W}{1+\bar \W}RR_\alpha  - g\Tilh[\W] = \left(\dfrac{R_\alpha}{1+\bar\W} - b_\alpha\right)(1 + \W)R + \bar\P_h\left[|R|^2\right]_\alpha,
\end{equation}
and then calculate,
\[
R_t = \frac{Q_{\alpha t}}{1+W_\alpha} - \dfrac{Q_\alpha W_{\alpha t}}{(1+W_\alpha)^2}.
\]
Thus, using \eqref{DQ-part1} and \eqref{DQ-part2}, we obtain the second part of \eqref{DiagonalQuasiSystem-re}.

%%%%%%%%%%%%%%%%%%%%%

\section{Local well-posedness for a model equation}
\label{s:model}  

In this section we will study the local well-posedness for a model equation, which will play a key role, both in the study of the linearized problem in the next section, and in the study of the differentiated equations later on. Here, and for the rest of the paper, we will assume 
that $h=1$, which we can do by scaling, and require uniformity with respect to $g$ in the range $g \lesssim 1$.

Our model system has the form
\begin{equation}\label{ModelLinSys}
\begin{cases}
w_t + \M_bw_\alpha + \P\left[\dfrac{r_\alpha}{1+\bar\W}\right] - \P\left[\dfrac{R_\alpha \Til^2 w}{1+\bar\W}\right] = G\vspace{0.1cm}\\
r_t + \M_br_\alpha - \P\left[\dfrac{(g+\mfa)\Til[w]}{1+\W}\right] = K,
\end{cases}
\end{equation}
where $\M_b$ is the holomorphic multiplication operator 
\[
\M_b f= \P(bf).
\]
Here both the unknowns $(w,r)$ and the inhomogeneous terms \((G,K)\in\H\) are holomorphic.  The functions $(\W,R)$ are solutions to the differentiated system \eqref{DiagonalQuasiSystem-re} and $b$ and $\mfa$ are the associated advection velocity, respectively the frequency shift, which are given by the formulas \eqref{def:b}, \eqref{def:a} in terms of $\W$ and $R$.  For convenience we recall the expressions of $b$, $\mfa$ and $M$ below
\[
b = 2\Re\left[R - \P[R\bar Y]\right],
\]
and
\[
\mfa = g(1 + \Til^2)\Re \W + 2\Im\P[R\bar R_\alpha], \quad M = 2\Re\P[R\bar
Y_\alpha - \bar R_\alpha Y].
\]

Notably, in our analysis we will not use at all the Sobolev regularity of $(\W,R)$.
Instead we will only use the bounds for $(\W,R)$ which are available in terms of the uniform control norms $A$ and $B$. Similarly, for $b$ and $\mfa$ we use only the corresponding
uniform bounds also in terms of $A$ and $B$, see Lemmas~\ref{l:a},~\ref{l:b} in the Appendix.

A natural energy for this system is given by the quadratic part of the Hamiltonian,
\begin{equation}\label{LinearEnergy}
 E_0(w,r) = g\langle w,w\rangle - \langle r,\Til^{-1}[r_\alpha]\rangle \approx_g \|(w,r)\|_{\H}^2.
\end{equation}
However, as the equations above  have variable coefficients, we instead work with an  adapted energy functional
\begin{equation}\label{QuadLinNRG}
E_{lin}^{(2)}(w,r) = \langle w,w\rangle_{g+\mfa} - \langle r,\Til^{-1}[r_\alpha]\rangle
= \langle w,w\rangle_{g+\mfa} + \langle \D r, \D r\rangle,
\end{equation}
where for a real valued weight $\omega$ we define the weighted inner product
\[
\langle u,v\rangle_\omega = \int \left(\Til\Re u\cdot\Til\Re v + \Im u\cdot\Im v\right)\, \omega\,d\alpha.
\]
We note that this inner product retains the orthogonality between holomorphic and antiholomorphic  functions, and thus the projectors $\P$ and $\bar\P$ continue to play the same role. From the 
Taylor stability condition \eqref{TS} in Lemma \ref{L:Taylor} ,
and the upper bound for $\mfa$ in Lemma~\ref{l:a}, we have
\[
E_{\lin}^{(2)}(w,r)\approx_{A}  E_0(w,r)
\]
for as long the fluid stays away from the bottom.

We also need  a weighted form of the above energy functional.
For a real valued weight $\omega$  we define
\begin{equation}
\begin{aligned}
E^{(2)}_{\omega, lin} (w,r):=&\left\langle w,w\right\rangle_{(g+\mfa)\omega} +
\left\langle \D r, \D r \right\rangle_{\omega} .
\end{aligned}
\end{equation}
Our main estimate for the model system is as follows:

\begin{proposition} \label{p:local} 
Let $I$ be a time interval where $A$ is bounded
and $B \in L^1$. Then in $I$ the following properties hold:

a)  The system of equations \eqref{ModelLinSys} is well posed in $ \H$, and
satisfies the estimate
  \begin{equation}
    \begin{aligned}
      \label{est1}
      \frac{d}{dt}E^{(2)}_{lin}(w,r)=2\left\langle G, w
      \right\rangle_{g+\mfa} + 2\left\langle \D K, \D r
      \right\rangle + O_{A}(B)  E^{(2)}_{lin}(w,r).
    \end{aligned}
  \end{equation}

b) Assume in addition that  $\omega$ is a weight satisfying
\begin{equation} \label{en-omega}
\| \omega \|_{L^\infty} \leq A, \qquad \| \omega \|_{\bmo^\frac12} \leq B,
\qquad \| (\partial_t + b \partial_\alpha) \omega \|_{L^\infty} \leq B.
\end{equation}
Then we also have 
  \begin{equation}
  \begin{aligned}
     \label{est2}
  \frac{d}{dt}E^{(2)}_{\omega,lin}(w,r)=2\left\langle G, w
      \right\rangle_{(g+\mfa)\omega} +2\langle  \D K, \D r \rangle_\omega + O_{A}(B) E^{(2)}_{lin}(w,r).
    \end{aligned}
  \end{equation}
\end{proposition}

\begin{proof}
We note that the bound \eqref{est1} can be viewed as a special case of \eqref{est2}, so we will only prove the latter. We start by calculating
\[
\begin{aligned}
\frac{d}{dt}\left\langle w,w \right\rangle_{(g+\mfa)\omega}=& \left\langle w,w \right\rangle_{[(g+\mfa)\omega ]_t}  -2\left\langle bw_{\alpha},w \right\rangle_{(g+\mfa)\omega}
+2\left\langle d \Til^2 w,w \right\rangle_{(g+\mfa)\omega}+2\left\langle G,w \right\rangle_{(g+\mfa)\omega}\\
&-2\left\langle (1-\bar Y) r_{\alpha}, w \right\rangle_{(g+\mfa)\omega}, 
\end{aligned}
\]
where we define \(d := \dfrac{R_\alpha}{1+\bar\W}\).
We complete the weight of the first term appearing in the expression above to
\[
\left\langle w,w \right\rangle_{(\partial_t +b\partial_{\alpha})[(g+\mfa)\omega ]}.
\]
Using also the relation $b_{\alpha}=M+2\Re d$ we separate the above time derivative into
\[
\frac{d}{dt}\left\langle w,w \right\rangle_{(g+\mfa)\omega} =  2\left\langle G,w \right\rangle_{(g+\mfa)\omega}
+ D_w^1 + D_w^2 + D_w^3 + D_w^4,
\]
where
\[
\begin{split}
D_w^1 = &  \     \left\langle w,w \right\rangle_{(\partial_t +b\partial_{\alpha})[(g+\mfa)\omega ]}
- \langle M \Til^2 w, w\rangle_{(g+\mfa)\omega},
\\
D_w^2 = & \    \left\langle 2i\Im d \Til^2 w,w \right\rangle_{(g+\mfa)\omega},
\\
D_w^3 = & \  -2\left\langle bw_{\alpha},w \right\rangle_{(g+\mfa)\omega}+ \left\langle b_\alpha \Til^2 w,w \right\rangle_{(g+\mfa)\omega}
-\left\langle w,w \right\rangle_{b\partial_{\alpha}[(g+\mfa)\omega ]},
\\
D_w^4 = & \  -2\left\langle (1-\bar{Y})r_{\alpha}, w \right\rangle_{(g+\mfa)\omega}.
\end{split}
\]

In a similar manner we expand the time derivative of the $r$ term as 
\[
\frac{d}{dt}\left\langle \D r, \D r \right\rangle_{\omega}= 2\left\langle \D K,\D r\right\rangle_{\omega} 
+ D_r^1 + D_r^2 + D_r^3,
\]
where 
\[
\begin{split}
D_r^1 = & \ \left\langle \D r, \D r \right\rangle_{(\partial_t +b\partial_{\alpha})\omega},
\\
D_r^2 = & \ -2 \left\langle \D (br_{\alpha}) , \D r \right\rangle_{\omega} + \left\langle \D r, \D r \right\rangle_{-b\partial_{\alpha}\omega},
\\
D_r^3 = & \ 2 \left\langle \D \left( (1-Y)(g+\mfa)\Til [w]\right) , \D r \right\rangle_{\omega}.
\end{split}
\]

We now successively consider all the terms above:
\medskip

{\em 1. The terms $D_w^1$ and $D_r^1$ } are trivially estimated using the pointwise bounds for $\mfa$ and its derivatives (see Lemma~\ref{l:a}) and $\omega$, as well as the pointwise bound for $M$ (see Lemma~\ref{l:M}).

\medskip
{\em 2. The term $D_w^2$} is expanded using the definition of our inner product as 
\[
D_w^2 = 2\int - (g+\mfa)\omega  \Til ( \Im d \, \Til^2 \Im w) \Til \Re w +  (g+\mfa)\omega \Im d \, \Til^2 \Re w \Im w \ d \alpha.
\] 
We use the relation $\Im w = -\Til \Re w$ to eliminate $\Re w$ and obtain
\[
\begin{split}
D_w^2 = & \ - 2\int  - (g+\mfa)\omega \Til ( \Im d \, \Til^2 \Im w)   \Im w +   (g+\mfa)\omega \Im d\,
\Til \Im  w \Im w \ d \alpha
\\
= & \ - 2\int  (g+\mfa)\omega  \Im w \left( - [\Til, \Im d]\Til^2  \Im  w + \Im d \, (1+\Til^2) \Til\Im w \right) d\alpha.
\end{split}
\]
Now we use a commutator bound
\[
\| [\Til, \Im d]\|_{L^2 \to L^2} \lesssim \|d\|_{\bmo}
\]
for the first term, see \eqref{est:TilComm},
and a Coifman-Meyer bound for the remaining product
\[
\| \Im d \, (1+\Til^2) \Til\Im w\|_{L^2} \lesssim \|\Im d\|_{\bmo} \|\Im w\|_{L^2},
\]
using the fact that the multiplier $1+\Til^2$ has a rapidly decaying kernel,
via (the dual of) \eqref{est:StrongLoHi}.

\medskip
{\em 3. The term $D_w^3$} is similarly expanded as 
% \[
% \begin{aligned}
%   D_w^3 = & \int_{\mathbb{R}}(g+\mfa)\omega \Im w \left\lbrace 2\Til(b
%     \Til^{-1} (\Im w_{\alpha})) - \Til (b_{\alpha}\Til \Im w) -2b\Im
%     w_{\alpha}\\
%&\hspace*{6cm}+ b_{\alpha}\Til^2\Im w\right\rbrace \, d\alpha\\
%   & =\int_{\mathbb{R}}(g+\mfa)\omega \Im w \left\lbrace 2[\Til,
%     b]\Til^{-1} \Im w_{\alpha}+[\Til, b_{\alpha}]\Til \Im w
%   \right\rbrace \, d\alpha .
% \end{aligned}
% \]
\[
\begin{aligned}
  D_w^3 = &  \int_{\mathbb{R}}(g+\mfa)\omega \Im w \left\lbrace - 2\Til(b
    \Til^{-1}  \Im w_{\alpha}) + \Til (b_{\alpha}\Til \Im w) + 2b\Im
    w_{\alpha}  \right.
    \\
    &\hspace*{4cm} \left. + b_{\alpha}\Til^2
    \Im w + 2b_\alpha\Im w\right\rbrace \, d\alpha
    \\
 =  &  \int_{\mathbb{R}}(g+\mfa)\omega \Im w \left\lbrace - 2[\Til,
    b]\Til^{-1} \Im w_{\alpha}+[\Til, b_{\alpha}]\Til \Im w + 2b_\alpha(1+\Til^2)\Im w
  \right\rbrace \, d\alpha .
\end{aligned}
\]

To bound the integral above we use the $L^{\infty}$ bounds for $\omega$ and $a$, together with  H\"older's inequality. The desired bounds for this integral are a consequence of the commutator bounds
\[
\Vert [\Til, b]\Vert_{H^{-1} \rightarrow L^2} \lesssim \Vert b_{\alpha}\Vert_{\bmo}, \quad \Vert [\Til, b_{\alpha}]\Vert_{L^2 \rightarrow L^2} \lesssim \Vert b_{\alpha}\Vert_{\bmo},
\]
which can be found in \eqref{est:TilComm}.

\medskip
{\em 4. The term $D_r^2$.} For simplicity we introduce the holomorphic variable $s:=\D r$. Further, we expand
\[
\begin{aligned}
- \left\langle \D (br_{\alpha}), \D r \right\rangle_{\omega}=& \left\langle \D (b\D^2 \Til (r)), \D r \right\rangle_{\omega} = \left\langle \D (b\D \Til (s)),s \right\rangle_{\omega}.
\end{aligned}
\]
Then 
\[
\begin{aligned}
D_r^2 = &\ 2\int_{\mathbb{R}} \omega \Im s \left(  \Til \D b\D+\D b\D \Til\right)\Im s -b\omega _{\alpha}(\Im s)^2\, d\alpha
\\ = &\  2\int_{\mathbb{R}}\omega \Im s \left(  \Til \D b\D+ \D b\D \Til
+ \partial_\alpha b + b \partial_\alpha
\right)\Im s\,  d\alpha
\\ = & \ 2 \int_{\mathbb{R}}\omega \Im s \left(  [\Til \D,b]\D+ \D [b,\D \Til]
\right)\Im s \, d\alpha.
\end{aligned}
\]
Thus we need an $L^2$ bound for the double commutator
\[
 \|[[\Til \D, b],\D]\|_{L^2 \to L^2} \lesssim \| b_\alpha\|_{\bmo},
 \]
which is established in the Appendix, see \eqref{double}.

\medskip
{\em 5. The term $D_w^4+D^3_r$.} This has the form
\[
\begin{split}
D_w^4+D^3_r = &  -2\left\langle (1-\bar{Y})r_{\alpha}, w \right\rangle_{(g+\mfa)\omega} 
+ 2 \left\langle \D \left( (1-Y)(g+\mfa)\Til [w]\right) , \D r \right\rangle_{\omega} 
\\ = & \ 2\left\langle (1-\bar{Y})\Til \D s, w \right\rangle_{(g+\mfa)\omega} 
+ 2 \left\langle \D \left( (1-Y)(g+\mfa)\Til [w]\right) , s \right\rangle_{\omega}.
\end{split}
\]
This has some commutator structure, so we expect to get the bound
\[
| D_w^4+D^3_r| \lesssim (\| Y\|_{\bmo^\frac12} + \|\mfa\|_{\bmo^\frac12} +\|\omega\|_{\bmo^\frac12})   \|w\|_{\mfH} \|s\|_{\mfH}, 
\]
with the implicit constant depending on the $L^\infty$ norm of the same parameters
$Y,\mfa$ and $\omega$. To see this we  divide the analysis into several steps. 

First we commute $L$ across $\omega$, and estimate the 
difference 
\[
|\langle L z, s \rangle_{\omega} - \langle z, Ls \rangle_{\omega}| \lesssim \|\omega\|_{\bmo^\frac12}\|z \|_{\mfH} \|s\|_{\mfH}.
\]
Expanding as above,  this reduces to the commutator bound (see Lemma~\ref{l:CMComm})
\[
\| [L,\omega]\|_{L^2 \to L^2} \lesssim \|\omega\|_{\bmo^\frac12}.
\]
We apply this to $z = (1-Y)(g+\mfa)\Til [w]$ and $s = \D r$. This reduces our problem to estimating 
the difference
\[
-2\left\langle (1-\bar Y)r_{\alpha}, w \right\rangle_{(g+\mfa)\omega} 
-2 \left\langle  (1-Y)(g+\mfa)\Til [w] , \Til^{-1} r_\alpha \right\rangle_{\omega}. 
\] 

Next we insert $g+\mfa$ inside via the estimate
\[
\langle (g+\mfa) z,w \rangle_\omega - \langle z,w \rangle_{(g+\mfa)\omega} \lesssim \|g+\mfa\|_{\bmo^\frac12}  \|z\|_{H^{-\frac12}} \|w\|_{\mfH}, 
\]
which reduces to the commutator bound (see \eqref{est:TilComm})
\begin{equation}\label{c12}
\| [g+\mfa, \Til]\|_{H^{-\frac12} \to L^2} \lesssim \|g+\mfa\|_{\bmo^\frac12}.
\end{equation}
We apply this with $z = (1-\bar{Y})r_{\alpha}$ to reduce our problem to estimating the 
difference 
\[
-2\left\langle (1-\bar Y)(g+\mfa) r_{\alpha}, w \right\rangle_{\omega} 
-2 \left\langle  (1-Y)(g+\mfa)\Til [w] , \Til^{-1} r_\alpha \right\rangle_{\omega}. 
\] 

Finally, with $e = (1-\bar Y)(g+\mfa) \in \bmo^\frac12$ and $z = \Til^{-1} r_\alpha$, it remains to estimate
 the difference
\[
|\langle e \Til z, w \rangle_\omega + \langle z, \bar e \Til w \rangle_{\omega}| \lesssim (\|e\|_{L^\infty} \| \omega\|_{\bmo^\frac12}
+ \|e\|_{\bmo^\frac12} \|\omega\|_{L^\infty}) \| w\|_{\mfH} \| z\|_{\mfH^{-\frac12}}.
\]
This vanishes if $\omega$ is constant. Else, writing $e = f+ig$, it reduces to the commutator bounds
\[
\| [ \omega, \Til f + f \Til]\|_{H^{-\frac12} \to L^2} \lesssim (\|f\|_{L^\infty} \| \omega\|_{\bmo^\frac12}
+ \|f\|_{\bmo^\frac12} \|\omega\|_{L^\infty}), 
\]
respectively 
\[
\| [ \omega, \Til g \Til]\| \lesssim (\|f\|_{L^\infty} \| \omega\|_{\bmo^\frac12}
+ \|f\|_{\bmo^\frac12} \|\omega\|_{L^\infty}), 
\]
which follow by repeated application of  bounds of the form \eqref{c12}.

\end{proof}

\section{The linearized equation}
\label{s:linearization}
In this section we first calculate the linearization of \eqref{FullSystem-re} and then prove that the corresponding linearized system is well-posed in \(\H\).

 We take the linearized variables at \((W,Q)\) to be \((w,q) = (\delta W,\delta
 Q)\) and compute
 \[
 \delta R = \frac{q_\alpha - Rw_\alpha}{1+\W},\qquad \delta F =
 \P\left[m - \bar m\right],\qquad \bar R\delta R = n,
 \]
 where we define
 \[
 m = \frac{q_\alpha - Rw_\alpha}{J} + \frac{\bar R
   w_\alpha}{(1+\W)^2},\qquad n = \frac{\bar R(q_\alpha -
   Rw_\alpha)}{1+\W}.
 \]
 We then obtain the linearized equations
 \begin{equation}\label{LinSys1}
   \begin{cases}
     w_t + Fw_\alpha + \P[m - \bar m](1+\W) = 0\vspace{0.1cm}\\
     q_t + Fq_\alpha + \P[m - \bar m]Q_\alpha - g\Til[w] + \P[n + \bar
     n] = 0.
   \end{cases}
 \end{equation}
 As \(F = b - \dfrac{\bar R}{1+\W}\) and \(\P = 1 - \bar\P\) we may
 write this system in the form
 \begin{equation}\label{LinSys2+}
   \begin{cases}
     w_t + bw_\alpha + \dfrac{q_\alpha - Rw_\alpha}{1+\bar \W} = 2(1+\W)\Re\bar\P[m]\vspace{0.1cm}\\
     q_t + bq_\alpha - g\Til[w] + \dfrac{R(q_\alpha - Rw_\alpha)}{1 + \bar\W} = 2i\Im\bar\P[n] +
     2Q_\alpha\Re\bar\P[m].
   \end{cases}
 \end{equation}

 This is a degenerate hyperbolic system with a double speed $b$, so in
 order to produce good energy estimates at high frequency we
 introduce diagonal variables. Following \cite{HIT}, a natural choice would be to take
 \((w,r) = \A(w,q) = (w,q - Rw)\). This would work at high frequencies, but 
not at low frequencies as we cannot make sense of
 the product \(Rw\) for \(w\in\mfH\). So instead we work with
 \[
 (w,r) = (w,q + R\Til^2w).
 \]
 We observe that \((w,r)\approx \A(w,q)\) when \(w\) is at frequencies
 \(\gg1\) whereas \((w,r)\approx(w,q)\) when \(w\) is at frequencies
 \(\ll1\).

 In terms of the diagonalized variables \((w,r)\) we have
 \begin{align*}
   m &= -\frac{\bar \W(r_\alpha - R_\alpha\Til^2[w] - R(1+\Til^2)[w_\alpha])}{J} +  \dfrac{\bar R w_\alpha}{(1+\W)^2},\\
   n &= \frac{\bar R(r_\alpha - R_\alpha\Til^2[w] -
     R(1+\Til^2)[w_\alpha])}{1+\W},
 \end{align*}
 where we have harmlessly removed the leading order holomorphic
 component of \(m\) that vanishes after projection to the space of
 antiholomorphic functions in \eqref{LinSys2+}. We then obtain the
 diagonalized system,
 \begin{equation}\label{DLinSys2}
   \begin{cases}
     w_t + bw_\alpha + \dfrac{r_\alpha }{1+\bar\W} - \dfrac{R_\alpha\Til^2 [w]}{1+\bar\W} = \cG\vspace{0.1cm}\\
     r_t + br_\alpha - \dfrac{(g+\mfa)\Til[w]}{1+\W} = \cK,
   \end{cases}
 \end{equation}
 where
 \begin{align*}
   \cG &= 2(1+\W)\Re\bar\P[m] + \frac{R(1+\Til^2)[w_\alpha]}{1+\bar\W}, \\
   \cK &= 2i\Im\bar\P[n] - R[1 + \Til^2,b]w_\alpha + R(1 + \Til^2)[w_t + bw_\alpha]
   + \frac{g\W - \mfa}{1+\W}(1+i\Til)\Til[w].
 \end{align*}
Here, for brevity in the notation, we have kept the $w_t + bw_\alpha$ term as a part of $\cK$, 
rather then substituting it from the first equation. This is harmless since $1+\Til^2$ 
has a Schwartz symbol so this term will only play a perturbative role.

 While \((w,r)\) are holomorphic, it is not immediately clear that
 \eqref{DLinSys2} preserves the space of holomorphic functions so we
 apply the projection \(\P\) to obtain
 \begin{equation}\label{PLinSys}
   \begin{cases}
     w_t + \M_bw_\alpha + \P\left[\dfrac{r_\alpha}{1+\bar\W}\right]  -\P\left[\dfrac{R_\alpha\Til^2[w]}{1+\bar\W}\right] = \P\cG\vspace{0.1cm}\\
     r_t + \M_br_\alpha - \P\left[\dfrac{(g+\mfa)\Til[w]}{1+\W}\right] =
     \P\cK,
   \end{cases}
 \end{equation}
which now has the form of the model equation \eqref{ModelLinSys}.
 Our main result for the linearized system \eqref{PLinSys} is the
 following Theorem:
 \medskip
 \begin{theorem}\label{thrm:LinQuadWP}
   Suppose that there exists a solution \((W,Q)\) to
   \eqref{FullSystem-re} on a time interval \([-T,T]\) such that
   \((W,Q)\in C([-T,T];\H)\) and \((\W,R)\in C([-T,T];\H^1)\). Then the
   linearized equation \eqref{PLinSys} is locally
   well-posed in \(\H\) on the interval \([-T,T]\), and the corresponding solution
   \((w,r)\in C([-T,T];\H)\) satisfies the estimate
   \begin{equation}
     \|(w,r)(t)\|_{\H}\lesssim \exp\left(C\int_0^t\|(g^{\frac12}\W,R)(s)\|_{H^1\times H^{\frac32}}\,ds\right) \|(w,r)(0)\|_{\H},
   \end{equation}
   where the implicit  constant depends only on
   \(A\) and \(\sup_{t\in[-T,T]}g^{-\frac12}\|(g^{\frac12}\W,R)(t)\|_{L^2\times H^{\frac12}}\).
 \end{theorem}
 \medskip

We remark that  \((w,q) = (W_\alpha,Q_\alpha)\) is a solution to \eqref{LinSys1},
for which we will prove  cubic lifespan bounds. Following \cite{HIT},
 one might hope to also establish cubic lifespan bounds for small initial
 data for the linearized system \eqref{PLinSys}.  Unfortunately this is not 
the case and we  expect that cubic lifespan bounds for the linearized system will fail
 on account of a breaking of symmetry when \((w,q) \neq (W_\alpha,Q_\alpha)\).
One can view this as a reflection of the fact that the quadratic low frequency interactions are
stronger here than in the infinite depth case.

In order to
prove Theorem~\ref{thrm:LinQuadWP} it will suffice to
obtain a priori estimates for \(\|(\P\cG,\P\cK)\|_{\H}\) and apply Proposition~\ref{p:local}. However, in
stark contrast to the infinite depth case \cite{HIT} we will be unable
to control \(\|(\P\cG,\P\cK)\|_{\H}\) only in terms of the pointwise
norms \(A,B\) and the energy \(E_{\lin}^{(2)}(w,r)\). The difficulty
arises due to the presence of nonlocal terms in the expression
\(\Re\bar\P[m]\) appearing in both \(\cG\) and \(\cK\). Here we will make use of
the fact that \(\P S_0\colon L^1\rightarrow L^\infty\), which leads to
bounds in terms of the energy norms of \((\W,R)\).

As a consequence, we have the following Proposition:

\medskip
\begin{proposition}\label{LinQuadWP}
We have the estimate
\begin{equation}\label{LinQuadEnough}
\|(\P\cG,\P\cK)\|_{\H}\lesssim_{A,g^{-\frac12}\|(g^{\frac12}\W,R)\|_{L^2\times H^{\frac12}}} \left( B + \|(g^{\frac12}\W,R)\|_{H^{\frac12}\times H^1} \right)\|(w,r)\|_{\H}.
\end{equation}
\end{proposition}
\bpf

We decompose
\[
\cG = G_1 + G_2,\qquad \cK = K_1 + K_2 + K_3 + K_4,
\]
where,
\begin{alignat*}{3}
G_1 &= 2(1+\W)\Re\bar\P[m],\quad &G_2 &= \frac{R(1+\Til^2)[w_\alpha]}{1+\bar\W}, \\
K_1 &= 2i\Im\bar P[n],\qquad &K_2 &= - R[1+\Til^2,b]w_\alpha,\\
K_3 &= R(1+\Til^2)(w_t+bw_\alpha),\qquad &K_4 &= \frac {g\W - \mfa}{1+\W}(1+i\Til)\Til[w],
\end{alignat*}
and estimate each term separately.
\medskip

{\em 1. Bounds for \(G_1\).}
We may estimate
\[
\|\P G_1\|_{\mfH}\lesssim \|\bar\P[m]\|_{\mfH} + \|\W\Re\bar\P[m]\|_{L^2}.
\]

We first prove that
\begin{equation}\label{e:Pbarm}
\|\bar\P[m]\|_{\mfH}\lesssim_A g^{-\frac12}B\|(w,r)\|_{\H}.
\end{equation}
As \(\bar\P\) vanishes when applied to holomorphic terms, we write
\begin{align*}
\bar\P[m] &= - [\bar\P,\bar Y](1 - Y)r_\alpha + [\bar\P,d\bar\W](1 - Y)\Til^2[w]\\
&\quad + \bar\P[\bar Y(1 - Y)R(1+\Til^2)w_\alpha] + [\bar\P,\bar R](1-Y)^2w_\alpha.
\end{align*}
For the first, second and fourth terms we apply the commutator estimate \eqref{PComm} and the product estimate \eqref{e:HolomProductBound} with the estimate \eqref{est:Ybmo} for \(Y\) and the estimate \eqref{est:d} for \(d\). For the third term we simply use that \(1+\Til^2\) has Schwartz symbol and that \(\|R\|_{L^\infty}\lesssim B\).

For the second term in \(G_1\) we first decompose according to the frequency of \(\Re\bar\P[m]\),
\[
\W\Re\bar\P[m] = \W P_{\geq1}\Re\bar \P[m] + \W S_0\Re\bar\P[m].
\]
For the high frequency component we use the estimate \eqref{e:Pbarm} for \(\bar \P[m]\) to obtain
\[
\|\W P_{\geq1}\Re\bar\P[m]\|_{L^2}\lesssim \|\W\|_{L^\infty}\|\bar \P[m]\|_{\mfH}\lesssim_A g^{-\frac12}AB\|(w,r)\|_{\H}.
\]
For the low frequency component we are unable to estimate \(S_0\Re\bar\P[m]\) in \(L^2\), so instead we estimate
\[
\|\W S_0\Re\bar\P[m]\|_{L^2}\lesssim \|\W\|_{L^2}\|S_0\Re\bar\P[m]\|_{L^\infty}.
\]
It then remains to show that
\begin{equation}\label{G1-there}
\|S_0\Re\bar\P[m]\|_{L^\infty}\lesssim_A g^{-\frac12}\|(g^{\frac12}\W,R)\|_{H^{\frac12}\times H^1}\|(w,r)\|_{\H}.
\end{equation}

For the first term in \(m\) we use that we use that \(S_0\P\colon L^1\rightarrow L^\infty\) to obtain
\[
\|S_0\bar\P[\bar Y(1-Y)r_\alpha]\|_{L^\infty}\lesssim \|S_0(\bar Y(1-Y)r_\alpha)\|_{L^1}.
\]
Considering this to be the product of \(\bar Y(1-Y)\) and \(r_\alpha\) we may only have high-high frequency interactions and hence
\begin{align*}
\|S_0(\bar Y(1-Y)r_\alpha)\|_{L^1} &\lesssim \sum\limits_{k\approx k'}\|\P_k[\bar Y(1-Y)]\|_{L^2}\|\P_{k'}[r_\alpha]\|_{L^2}\\
&\lesssim \|\bar Y(1-Y)\|_{H^{\frac12}}\|L r\|_{\mfH}\\
&\lesssim_A \|\W\|_{H^{\frac12}}\|L r\|_{\mfH},
\end{align*}
where the final line follows from the Moser estimate \eqref{MoserL2}. For the second and third terms in \(m\) we may straightforwardly estimate
\[
\left\|\frac{R_\alpha\bar Y}{1 + \W}\Til^2[w]\right\|_{L^1} +\left\|\frac{R\bar Y}{1 + \W}(1+\Til^2)w_\alpha\right\|_{L^1} \lesssim_A \|R\|_{H^1}\|w\|_{\mfH}.
\]
For the final term in \(m\) we consider it to be a product of \(\bar R\) and \((1-Y)^2w_\alpha\) to obtain
\[
\|S_0(\bar R(1-Y)^2w_\alpha)\|_{L^1}\lesssim \sum\limits_{k\approx k'}\|R_k\|_{L^2}\|\P_k[(1-Y)^2w_\alpha]\|_{L^2}\lesssim \|R\|_{H^1}\|(1-Y)^2w_\alpha\|_{H^{-1}}.
\]
The estimate \eqref{G1-there} then follows from the product estimate \eqref{e:HolomProductBound}.
\medskip

{\em 2. Bounds for \(G_2\).} Here we simply use that \(1+\Til^2\) has Schwartz symbol to obtain
\[
\left\|\P\left[\frac{R(1+\Til^2)w_\alpha}{1+\bar\W}\right]\right\|_{\mfH}\lesssim_AB\|w\|_{\mfH}.
\]
\medskip

{\em 3. Bounds for \(K_1\).}
As \(K_1\) is purely imaginary we have
\[
\|L\P[i\Im\bar\P[n]]\|_{\mfH}\lesssim \|L\bar\P[n]\|_{\mfH}.
\]
We then write
\[
\bar\P[n] = [\bar\P,\bar R](1-Y)r_\alpha - [\bar\P,\bar R](1-Y)R_\alpha\Til^2[w] - \bar\P\left[\frac{|R|^2}{1+\W}(1+\Til^2)w_\alpha\right],
\]
and may estimate each term similarly to the proof of \eqref{e:Pbarm} to obtain
\[
\|L\bar\P[n]\|_{\mfH}\lesssim_AB\|(w,r)\|_{\H}.
\]
\medskip

{\em 4. Bounds for \(K_2\).}
We start by dividing \(K_2\) up according to frequency balance using the paraproduct operator \(T_R\) as
\[
R[1+\Til^2,b]w_\alpha = T_R[1+\Til^2,b]w_\alpha + (R - T_R)[1+\Til^2,b]w_\alpha.
\]

When \(R\) is at low frequency we may estimate
\[
\|T_R[1+\Til^2,b]w_\alpha\|_{H^{\frac12}}\lesssim \|R\|_{L^\infty}\|[1+\Til^2,b]w_\alpha\|_{H^{\frac12}}
\]
and for the remaining terms we apply the paraproduct estimate \eqref{Para1} to obtain
\[
\|(R - T_R)[1+\Til^2,b]w_\alpha]\|_{H^{\frac12}}\lesssim \|R\|_{\bmo^{\frac12}}\|[1+\Til^2,b]w_\alpha\|_{L^2}.
\]
As a consequence,
\[
\|L\P K_2\|_{\mfH}\lesssim g^{\frac12}A\|[1+\Til^2,b]w_\alpha\|_{H^{\frac12}}
\]

We then decompose using paraproducts,
\begin{align*}
[1+\Til^2,b]w_\alpha &= [1+\Til^2,T_b]w_\alpha + [1+\Til^2,b_0]w_{\leq4} + (1+\Til^2)T_{w_\alpha}b - T_{(1+\Til^2)w_\alpha}b\\
&\hspace*{1cm} + (1+\Til^2)\Pi[b_{\geq1},w_\alpha] - \Pi[b_{\geq1},(1+\Til^2)w_\alpha],
\end{align*}
and estimate each of these terms as follows: for the first two terms we apply the commutator estimate \eqref{est:WeakLoHi}, for the third and fourth terms the estimate \eqref{est:StrongLoHi} and for the final two terms we apply the paraproduct estimate \eqref{Para1}. The estimate for \(K_2\) then follows from the estimate \eqref{est:bBMO} for \(b\).
\medskip

{\em 5. Bounds for \(K_3\).}
We may estimate similarly to \(K_2\) to obtain
\begin{align*}
\|L\P K_3\|_{\mfH} &\lesssim \|R\|_{\bmo^{\frac12}}\|(1+\Til^2)[w_t + bw_\alpha - 2(1 + \W)S_0\Re\bar P[m]]\|_{H^{\frac12}}\\
&\quad + \|R\|_{H^{\frac12}}(1 + \|\W\|_{L^\infty})\|S_0\Re\bar P[m]]\|_{L^\infty}.
\end{align*}
For the first term we estimate as for \(G_1\) using that \(1 + \Til^2\) has Schwartz symbol to obtain
\[
\|(1+\Til^2)[w_t + bw_\alpha - 2(1 + \W)S_0\Re\bar P[m]]\|_{H^{\frac12}}\lesssim_A g^{-\frac12}B\|(w,r)\|_{\H},
\]
and for the second term we may simply apply the estimate \eqref{G1-there}.

{\em 6. Bounds for \(K_4\).}
As \(\Til[w]\) is holomorphic we have
\[
(1 + i\Til)\Til[w] = (1 + \Til^2)\Re\Til[w].
\]
We may then apply the paraproduct estimates \eqref{Para1} and \eqref{est:StrongLoHi} with the estimates \eqref{est:a} for \(\mfa\) and \eqref{est:Ybmo} for \(Y\) to obtain
\[
\left\|\frac{g\W - \mfa}{1+\W}(1+i\Til)\Til[w]\right\|_{H^{\frac12}}\lesssim\left(g\|Y\|_{\bmo^{\frac12}} + \|\mfa\|_{\bmo^{\frac12}}\right)\|w\|_{L^2}.
\]

This completes the proof of \eqref{LinQuadEnough}.

\epf
%

%%%%%%%%%%%%%%%%%%%%
%%% NORMAL FORMS %%%
%%%%%%%%%%%%%%%%%%%%

\section{Normal forms}\label{s:nf}

The goal of this section is to algebraically compute a normal
form correction for the system \eqref{FullSystem-re} for $(W,Q)$ as a
translation invariant bilinear form. We recall that the aim of the
normal form transformation is to eliminate the quadratic terms in the
equation. Precisely, at least formally the normal form variables
$(\tW,\tQ)$ will solve a nonlinear equation where all the nonlinear
terms are cubic and higher order. In this article we will not use such
an equation directly for three reasons:
\begin{itemize}
\item[(i)] The equation for the normal form variables $(\tW,\tQ)$ is not
self-contained, instead it still uses the original variables $(W,Q)$
in the nonlinearity.
\item[(ii)] The system \eqref{FullSystem-re} is fully nonlinear and 
the normal form transformation does not mix well with the nonlinear
structure.
\item[(iii)] The symbols for the normal form transformation are singular
precisely when the output has frequency zero.
\end{itemize}
Instead, in the next section we use the normal form
transformation in order to produce a cubic normal form energy that has the property that its time derivative along the flow is
of quartic and higher order. Interestingly (and very usefully) the 
normal form symbol singularities do not carry over to the 
normal form energy; this is due to cancellations arising 
after repeated symmetrizations.

 Incidentally, we remark that when considering the linearized equation some 
of these symmetrizations are lost, which is why we cannot prove 
cubic energy estimates for the linearized flow.

\subsection{The resonance analysis}
If we take \((W,Q) = 0\) in the linearized system \eqref{LinSys1} we obtain the system
\begin{equation}\label{LinearizedAtZero}
\begin{cases}
w_t + q_\alpha = 0\vspace{0.1cm}\\
q_t - g\Til[w] = 0,
\end{cases}
\end{equation}
which has dispersion relation
\[
\tau^2 = g\xi\tanh\xi.
\]
As a consequence we see that solutions split into right-moving and
left-moving components with dispersion relations \(\tau = \pm g^{\frac12}\omega(\xi)\), respectively, where
\[
\omega(\xi) = -\sgn\xi\sqrt{\xi\tanh\xi}.
\]

To understand bilinear resonant interactions we define the function
\[
\Delta(\xi,\eta,\zeta) = \omega(\xi) + \omega(\eta) + \omega(\zeta).
\]
Then resonant two wave interactions correspond to  solutions to the system
\[
\begin{cases}
\Delta(\pm\xi,\pm\eta,\pm\zeta) = 0\vspace{0.1cm}\\
\xi + \eta + \zeta = 0.
\end{cases}
\]
As \(\omega\) is sublinear, the only solutions occur when at least one of
 \(\xi,\eta,\zeta\) vanishes.

Symmetrizing the $\Delta$ function, we define the resonance function $\Omega$ by
\begin{align*}
\Omega(\xi,\eta,\zeta) &= \Delta(\xi,\eta,\zeta)\Delta(\xi,-\eta,-\zeta)\Delta(\xi,-\eta,\zeta)\Delta(\xi,\eta,-\zeta)\\
&= J(\xi)^2 + J(\eta)^2 + J(\zeta)^2 - 2J(\xi)J(\eta) - 2J(\eta)J(\zeta) - 2J(\zeta)J(\xi),
\end{align*}
on the set $\mc P = \{\xi + \eta + \zeta = 0\}$, where \(J(\xi) = \omega(\xi)^2 =
\xi\tanh\xi\). This vanishes quadratically on each of the lines $\xi =
0$, $\eta = 0$, respectively $\zeta = 0$. The function $\Omega$ will
play a key role in all computations which follow.

\subsection{Expansion to cubic order}
In order to construct normal forms for \((W,Q)\) we first expand \(F\) to cubic order as
\begin{align*}
\Poly{\leq 3}F
&= Q_\alpha - Q_\alpha W_\alpha - \P\left[Q_\alpha \bar W_\alpha - \bar Q_\alpha W_\alpha\right] + Q_\alpha W_\alpha^2 + \P\left[(Q_\alpha \bar W_\alpha - \bar Q_\alpha W_\alpha)(W_\alpha + \bar W_\alpha)\right],
\end{align*}
where, for a sufficiently smooth function \(f\colon\C^2\rightarrow\C\) we define \(\Poly{\leq k}f\) to select the terms of polynomial order \(\leq k\) in the Taylor expansion of \(f\) at zero.

We may then rewrite \eqref{FullSystem-re} as
\begin{equation}\label{LinearQuadratic}
\begin{cases}
W_t + Q_\alpha = G^{[2]} + G^{[3]} + G^{[4+]}\vspace{0.1cm}\\
Q_t - g\Til[W] = K^{[2]} + K^{[3]} + K^{[4+]},
\end{cases}
\end{equation}
where the quadratic terms are given by
\begin{gather*}
G^{[2]} = \P[Q_\alpha \bar W_\alpha - \bar Q_\alpha W_\alpha],\quad K^{[2]} = - Q_\alpha^2 -\P\left[Q_\alpha\bar Q_\alpha\right],
\end{gather*}
the cubic terms are given by
\[
\begin{aligned}
&G^{[3]} = W_\alpha\P\left[Q_\alpha\bar W_\alpha - \bar Q_\alpha W_\alpha\right] - \P\left[(Q_\alpha \bar W_\alpha - \bar Q_\alpha W_\alpha)(W_\alpha + \bar W_\alpha)\right],\\
&K^{[3]} = Q_\alpha^2 W_\alpha + Q_\alpha\P\left[Q_\alpha\bar W_\alpha - \bar Q_\alpha W_\alpha\right] +\P\left[Q_\alpha\bar Q_\alpha(W_\alpha+\bar W_\alpha)\right],
\end{aligned}
\]
and \(G^{[4+]},K^{[4+]}\) contain only quartic and higher order terms.

\subsection{Normal forms}
By considering parity, we seek holomorphic normal form corrections of the form
\[
\begin{cases}
\tW = W + B^h[W,W] + \dfrac1gC^h[Q,Q] + B^a[W,\bar W] + \dfrac 1gC^a[Q,\bar Q]\vspace{0.1cm}\\
\tQ = Q + A^h[W,Q] + A^a[W,\bar Q] + D^a[Q,\bar W],
\end{cases}
\]
so that the normal form variables $(\tW,\tQ)$  satisfy
\begin{equation}\label{NF-eq}
\begin{cases}
\Poly{\leq2}[\tW_t + \tQ_\alpha] =  0 \vspace{0.1cm}\\
\Poly{\leq 2}[\tQ_t - g\Til\tW ] =  0.
\end{cases}
\end{equation}
Here the operators $B^h$, $C^h$, $B^a$, $C^a$, $A^h$, $A^a$, $D^a$ are 
translation invariant bilinear forms, which can be described via their symbols, as below:
\begin{gather*}
B^h[W,W] = \frac{1}{2\pi}\int B^h(\xi,\eta)\hat W(\xi)\hat W(\eta) 
e^{i(\xi+\eta)\alpha}\,d\xi d\eta\\
B^a[W,\bar W] = \frac{1}{2\pi}\int B^a(\xi,\eta)\hat W(\xi)\bar{\hat W}(\eta) 
e^{i(\xi-\eta)\alpha}\,d\xi d\eta.
\end{gather*}
To determine these symbols uniquely we  assume that \(B^h,C^h\) are symmetric.

For the subsequent construction of the normal form energies we will interpret all symbols as functions on the plane $\mathcal P =
\{\xi + \eta + \zeta = 0\}$. For notational convenience we will adopt this convention in the following computations. In the context of bilinear operators we may interpret \(\zeta = \zeta(\xi,\eta):= - (\xi + \eta)\). For this reason we will compute holomorphic symbols at \((\xi,\eta)\) and mixed holomorphic-antiholomorphic symbols at \((\xi,-\eta)\).

\subsubsection{Holomorphic products} 
The holomorphic terms $A^h$, $B^h$ and $C^h$ are generated by the
holomorphic part of the quadratic nonlinearity, i.e., the first term
in $K^{[2]}$.  Comparing holomorphic terms at the quadratic level we
obtain a linear system for the symbols
\[
\begin{bmatrix}\xi+\eta & -2\eta & -2\tanh\xi\\
-\xi &0 &\tanh(\xi+\eta)\\
-\tanh\eta&\tanh(\xi+\eta)&0\end{bmatrix}\begin{bmatrix}A^h(\xi,\eta)\\B^h(\xi,\eta)\\C^h(\xi,\eta)\end{bmatrix} = \begin{bmatrix}0\\i\xi\eta\\0\end{bmatrix}.
\]
From the first row we have
\[
A^h = \frac{2\eta B^h}{\xi+\eta} + \frac{2\tanh\xi C^h}{\xi+\eta}.
\]
We then calculate the symmetrizations
\begin{gather*}
(\xi A^h)_{\mr{sym}} = \frac{2\xi\eta B^h}{\xi+\eta} + \frac{(\xi\tanh\xi + \eta\tanh\eta) C^h}{\xi+\eta},\\
(\tanh\eta A^h)_{\mr{sym}} = \frac{(\xi\tanh\xi + \eta\tanh\eta)B^h}{\xi+\eta} + \frac{2\tanh\xi\tanh\eta C^h}{\xi+\eta}.
\end{gather*}
Plugging this into the second row we obtain,
\[
((\xi+\eta)\tanh(\xi+\eta) - \xi\tanh\xi - \eta\tanh\eta)C^h = i\xi\eta(\xi+\eta) + 2\xi\eta B^h,
\]
and into the third row,
\[
((\xi+\eta)\tanh(\xi+\eta) - \xi\tanh\xi - \eta\tanh\eta) B^h = 2\tanh\xi\tanh\eta C^h.
\]
As a consequence we obtain the solutions
\begin{align*}
A^h(\xi,\eta) &= \frac{2i\eta J(\xi)\left(J(\zeta) - J(\xi) + J(\eta)\right)}{\Omega},\\
B^h(\xi,\eta) &= -\frac{2i\zeta J(\xi)J(\eta)}{\Omega},\\
C^h(\xi,\eta) &= -\frac{i\xi\eta\zeta\left(J(\zeta) - J(\xi) - J(\eta)\right)}{\Omega}.
\end{align*}

\subsubsection{Mixed terms} The mixed terms $A^a$, $B^a$, $C^a$ and
$D^a$ are generated by the mixed holomorphic-antiholomorphic part of the quadratic
nonlinearity.  As above, we write the mixed
holomorphic-antiholomorphic terms as a linear system
\[
\begin{bmatrix}\xi+\eta & -\eta & -\tanh\xi & 0\\
0 & -\xi & - \tanh\eta & \xi+\eta\\
-\xi & 0 & \tanh(\xi+\eta) & -\eta\\
-\tanh\eta & \tanh(\xi+\eta) & 0 & -\tanh\xi\end{bmatrix}\begin{bmatrix}A^a(\xi,-\eta)\\B^a(\xi,-\eta)\\C^a(\xi,-\eta)\\D^a(\xi,-\eta)\end{bmatrix} = \begin{bmatrix}\frac12i(1 - \coth(\xi+\eta))\xi\eta\\ -\frac12i(1-\coth(\xi+\eta))\xi\eta\\\frac12i(1 - \tanh(\xi+\eta))\xi\eta\\0\end{bmatrix}
\]
We solve this system to obtain
\begin{align*}
A^a(\xi,-\eta) &= - \frac{e^{2\zeta}}{e^{2\zeta} + 1}\left\{\left( J(\eta) + \eta \right)\frac{B^h(\xi,\eta)}{\zeta\tanh\eta} + \left(J(\xi) - \xi\right)\frac{C^h(\xi,\eta)}{\xi\zeta}\right\},
\\
B^a(\xi,-\eta) &= \frac{e^{2\zeta}}{e^{2\zeta} - 1}\left\{\left(J(\zeta) - (\xi-\eta)\right)\frac{B^h(\xi,\eta)}{\zeta} + \left( \eta J(\xi)-\xi J(\eta) \right)\frac{C^h(\xi,\eta)}{\xi\eta\zeta}\right\},
\\
C^a(\xi,-\eta) &= \frac{e^{2\zeta}}{e^{2\zeta}-1}\left\{\left( \eta J(\xi)-\xi J(\eta) \right)\frac{B^h(\xi,\eta)}{\zeta\tanh\xi\tanh\eta} + \left(J(\zeta) - (\xi-\eta)\right)\frac{C^h(\xi,\eta)}{\zeta}\right\},
\\
D^a(\xi,-\eta) &= -\frac{e^{2\zeta}}{e^{2\zeta}+1}\left\{ \left( J(\xi)-\xi \right)\frac{B^h(\xi,\eta)}{\zeta\tanh\xi}  + \left( J(\eta) + \eta \right)\frac{C^h(\xi,\eta)}{\eta\zeta}\right\}.
\end{align*}

\subsection{Symbol classes and asymptotics for the normal form}
Here we consider the symbols arising in the normal form,
and describe their size and regularity. These are needed 
in order to have good $L^2$ and $L^p$ multilinear bounds.

From the perspective of high frequency bounds, we are interested
in the interactions between one high negative frequency and 
one low frequency. Here we expect only the symbols $B^h$, $B^a$, $A^h$
and $D^a$ to play a role; the remaining symbols $C^h$, $C^a$ and $A^a$
(which do not appear at all in the infinite bottom case) will decay 
rapidly in the above regime. For the former symbols, on the other
hand, we will need to compute second order expansions around $\xi = 0$
(for $B^h$ and $A^h$) respectively around $\eta = 0$ (for $B^a$, $A^h$ and 
$D^a$). However, due to the linear component of the normal derivative of the pressure, we will also require an expansion for \(C^h\) near 
\(\eta = 0\).

From the perspective of low frequency analysis, we do not have any 
low frequency pointwise control on $\Re W$ and $\Re Q$. Hence 
we will need to show that $\Re W$ and $\Re Q$ do not appear undifferentiated
in our cubic energy functional. This requires certain cancellations 
to happen (akin to a null condition). For this we will need 
to exactly compute almost all of the above symbols at $\xi = 0$ 
and at $\eta = 0$.

We will interpret all symbols as functions on the plane $\mathcal P =
\{\xi+\eta+\zeta = 0\}$. In this plane we consider
three distinguished lines $\xi = 0$, $\eta = 0$, $\zeta = 0$.  The
symbol regularity will depend on the distance $d$ to these lines and
on the radius $\rho$,
\[
d = 1+ \min\{|\xi|,|\eta|,|\zeta| \}, \qquad \rho =  1+ \max\{|\xi|,|\eta|,|\zeta| \}.
\]
For a weight $\sigma$ which is slowly varying with respect to these scales we denote by $S(\sigma)$ the class  of symbols $s$ on $\mathcal P$ which satisfy 
\[
| (d \partial)^\alpha (\rho \partial_\rho)^\beta s | \lesssim c_{\alpha \beta} \sigma .
\]

We begin our discussion with the  expression $\Omega$, for which we have:

\medskip
\begin{lemma}~

a) The symbol $\Omega$ restricted to $\mathcal P$ 
is non-positive and belongs to $S(d \rho)$.

b) The symbol $\Omega$ vanishes quadratically on the three lines and is 
elliptic elsewhere, 
\[
\Til^2(\xi) \Til^2(\eta) \Til^2(\zeta) \Omega^{-1} \in S(d^{-1} \rho^{-1}).
\]

c) We have the following expansion in the region \(|\eta|\ll-\xi\):
\[
\Omega(\xi,\eta,\zeta) = 4 J(\eta) \xi + (\eta + J(\eta))^2 + S(e^{\xi})
=  - 4 J(\eta) \zeta + (\eta - J(\eta))^2 + S(e^{-\zeta}).
\]

d) On the line $\eta = 0$ we have the limit
\[
\lim_{|\eta| \to 0} \eta^{-2} \Omega(\xi,\eta,\zeta) =   J'(\xi)^2 - 4 J(\xi) =: \Lambda(\xi) < 0. 
\]

\end{lemma}
\medskip

The proof is a fairly straightforward algebraic computation and is
omitted.  We remark that part (b) is consistent with the fact that in
our problem two wave resonances appear only when either an input
frequency or the output frequency is zero. Part (c) is relevant in our
high frequency analysis, while part (d) is needed for the low
frequency cancellation.

Now we successively consider the symbols in our normal form analysis:

\bigskip
{\bf The symbol $B^h$.}  Here by inspection we see that all the zeros 
of $\Omega$ are canceled by the numerator, except for a simple zero
at $\zeta = 0$. Then the natural regularity statement is obtained 
after multiplication with $\Til(\zeta)$. Precisely, we have
\begin{equation}\label{reg-bh}
\Til(\zeta) B^h(\xi,\eta) \in S(\rho).
\end{equation}
For the high frequency asymptotics in the region \(|\eta|\ll- \xi\) we have the expansion
\begin{equation}
B^h(\xi,\eta) = -\frac{i}2 \left( \xi - \frac{(\eta-J(\eta))^2}{4J(\eta)}\right) + S(d^2 \rho^{-1}).
\end{equation}

On the other hand, at $\eta = 0$ we have
\[
B^h(\xi,0) = \frac{2i \xi J(\xi)}{\Lambda(\xi)}.
\]

\bigskip
{\bf The symbol $C^h$.} Again all the zeros 
of $\Omega$ are canceled by the numerator, except for a simple zero
at $\zeta = 0$. Further, the difference $J(\xi) + J(\eta) - J(\zeta)$
decays exponentially if $\xi$ and $\eta$ have the same sign,
\[
J(\xi) + J(\eta) - J(\zeta) = O( e^{-|\xi|} + e^{-|\eta|}),\quad \xi\eta>0.
\]
We then obtain the size of $C^h$ as 
\begin{equation}\label{reg-ch}
\Til(\zeta) C^h(\xi,\eta) \in \left\{ \begin{array}{ll} S(\rho\ \min\{|\xi|,|\eta|\}) & \xi \eta < 0 
\cr S( d^{-N} \rho ) & \xi \eta > 0. \end{array} \right.
\end{equation}
The asymptotics in the region \(|\eta|\ll-\xi\) are
\[
C^h(\xi,\eta) = - \frac{i\eta(\eta+J(\eta))}{4J(\eta)} \left( \xi - \frac{(\eta - J(\eta))^2}{4J(\eta)}\right) + S(d^3 \rho^{-1}).
\]

Finally, we also need 
\[
C^h(\xi,0) = \frac{i \xi^2 J'(\xi)}{ \Lambda(\xi)}.
\]

\bigskip

{\bf The symbol $A^h$.} As before, we remove the zero at $\zeta = 0$ to obtain the regularity
\begin{equation}\label{reg-ah}
\Til(\zeta) A^h(\xi,\eta) \in  \left\{ \begin{array}{ll} S(|\eta|) & \eta\zeta < 0 
\cr S( d^{-N} |\eta| ) & \eta\zeta > 0. \end{array} \right. 
\end{equation}
Since $A^h$ is not symmetric, we need asymptotics both near $\xi =0 $
and $\eta = 0$. First we consider the region $|\eta|\ll - \zeta$.
Here we have
\begin{equation}\label{zero-ah-eta}
A^h(\xi,\eta) = - \frac{i\eta(\eta + J(\eta))}{2 J(\eta)}\left(1 + \frac{(\eta - J(\eta))^2}{4\xi J(\eta)}\right) + S(d^3\rho^{-1}),
\end{equation}
where the leading order term vanishes (which is consistent with the infinite bottom problem). Next, we consider the region $|\xi|\ll-\eta$:
\begin{equation}\label{zero-ah-eta+}
A^h(\xi,\eta) = - i \left( \eta +\frac{J(\xi)^2 -\xi^2}{4J(\xi)} \right) + S(d^2\rho^{-1}).
\end{equation}

Finally, we compute 
\[
A^h(\xi,0) = \frac{2 i J(\xi) J'(\xi)}{  \Lambda(\xi)},\qquad A^h(0,\eta) = \frac{4i \eta J(\eta)}{  \Lambda(\eta)}.
\]

\bigskip

Next we consider the symbols for the mixed terms, namely
$A^a$, $B^a$, $C^a$ and $D^a$. Here we will continue to consider $(\xi,\eta,\zeta)\in\mc P$ and compute the symbols at \((\xi,-\eta)\).

\bigskip

{\bf The symbol $A^a$.}
The symbol \(A^a(\xi,-\eta)\) decays exponentially in all directions except near the half-lines
$\{\xi = 0,\ \eta < 0\}$ and $\{\zeta=0,\ \xi<0\}$. Precisely, we have
\begin{equation}
\Til(\zeta) A^a(\xi,-\eta) \in  \left\{ \begin{array}{ll} S(d^{-N}|\xi|) & |\xi| \ll  -\eta\ \ \text{or}\ \  |\zeta|\ll-\xi
\cr S( \rho^{-N} ) & \text{elsewhere}. \end{array} \right. .
\end{equation}
Finally, we have
\[
A^a(\xi,0) = \frac{i}{\Lambda(\xi)(e^{2\xi} + 1)}\left[ 2  J(\xi) - \xi J'(\xi) + J(\xi) J'(\xi) \right].
\]

\bigskip

{\bf The symbol $B^a$.}
This is similar to $B^h$, in that 
\begin{equation}
\Til(\zeta) B^a(\xi,-\eta) \in S(\rho).
\end{equation}
In the region $|\eta| \ll -\xi$ we have the asymptotics
\begin{equation}
B^a(\xi,-\eta) = -i \xi + S(d^2\rho^{-1}).
\end{equation}
Finally,  we do not need the exact expressions for 
$B^a(\xi,0)$ and $B^a(0,-\eta)$, only the fact that they are purely 
imaginary.

\bigskip

{\bf The symbol $C^a$.}
The symbol \(C^a(\xi,-\eta)\) decays exponentially away from the region $\{0 < \xi \ll -\eta\}$
and the half-line $\{\zeta = 0,\ \eta < 0\}$. Precisely,
\begin{equation}
\Til(\zeta) C^a(\xi,-\eta) \in \left\{ \begin{array}{ll} S(|\xi| \rho) & 0 < \xi \ll -\eta \ \ \text{or} \ \ |\zeta|\ll-\eta
\cr S( d^{-N} |\xi| \rho ) & 0<-\xi \ll-\eta \ \ \text{or} \ \ |\zeta|\ll-\xi
\cr S( \rho^{-N} ) & \text{elsewhere}. \end{array} \right..
\end{equation}
Finally, on the 
two lines we have
\begin{align*}
C^a(\xi,0) &= \frac{i\xi}{\Lambda(\xi)(e^{2\xi} - 1)} \left [2J(\xi) - \xi J'(\xi) +  J(\xi) J'(\xi) \right], \\
C^a(0,-\eta) &= -\frac{i\eta}{\Lambda(\eta)(e^{2\eta} - 1)}\left[ 2J(\eta) - \eta J'(\eta) - J(\eta)J'(\eta) \right].
\end{align*}

\bigskip

{\bf The symbol $D^a$.} The symbol \(D^a(\xi,-\eta)\) decays exponentially away from the region $\{0 < -\xi \ll -\eta\}$
and the half-line $\{\eta = 0,\ \xi < 0\}$. Precisely,
\begin{equation}
\Til(\zeta) D^a(\xi,-\eta) \in \left\{ \begin{array}{ll} S(|\xi|)
& 0 < -\xi \ll -\eta \ \ \text{or} \ \ |\eta|\ll-\xi
\cr S(d^{-N}|\xi|) & 0<\xi\ll-\eta\ \ \text{or} \ \ |\zeta|\ll-\eta
\cr S( \rho^{-N} ) & \text{elsewhere}. \end{array} \right. .
\end{equation}
We will only require its high frequency asymptotics in the region $|\eta| \ll -\xi$:
\begin{equation}
D^a(\xi,-\eta) = - i \xi+ S(d^2 \rho^{-1}),
\end{equation}
which are similar to those for $B^a$.

Finally, we also need
\[
D^a(0,-\eta) = -\frac{i}{\Lambda(\eta)(e^{2\eta} + 1)}\left[ 2J(\eta) - \eta J'(\eta) - J(\eta)J'(\eta) \right].
\]

\section{The normal form energy.} \label{s:nf-en}

The aim of this section is to use the normal form computation in the previous section to produce a normal form energy, i.e., an energy functional which is accurate to quartic order. We summarize our result as follows:

\begin{proposition} \label{p:ennnf}
For each $n \geq 1$ there exists a normal form
  energy $E^n_{NF}\!=\! E^{n}_{NF}(\W, R)$ with the
  following properties:
  \medskip
  
a) Algebraic properties. $E^{n}_{NF}(\W, R)$ has only quadratic and cubic terms,
\[
\Poly{\geq 4}  E^{n}_{NF}(\W, R) = 0,
\]
and its quadratic part is given by the linear energy
\[
\Poly{\leq 2} E^{n}_{NF}(\W, R) = E_0(\partial^{n-1} \W, \partial^{n-1} R).
\]
Further, $E^{n}_{NF}(\W, R)$ is accurate to quartic order, i.e.,
\begin{equation}\label{dt-quartic}
\Poly{\leq 3}  \frac{d}{dt} E^{n}_{NF}(\W, R) = 0
\end{equation}
along the flow of \eqref{FullSystem-re}.
\medskip

b) Qualitative description. $E^n_{NF}$ has the form
\[
E^{n}_{NF}(\W, R) = E_0(\partial^{n-1} \W, \partial^{n-1} R) 
+ g B(\W,\W,\W) + A(\W,R,R),
\]
where $A$ and  $B$ are translation invariant trilinear forms.
Further, there is a decomposition
\[
  E^{n}_{NF}=E^n_{NF, high}+E^n_{NF, low},
\]
with 
\[
\begin{split}
E^n_{NF, high}= & \ E_0(\partial^{n-1} \W, \partial^{n-1} R) 
+ g B_{high}(\W,\W,\W) + A_{high}(\W,R,R), 
\\
E^n_{NF, low} = & \  g B_{low}(\W,\W,\W) + A_{low}(\W,R,R),
\end{split}
\]
where the forms $ B_{high}$ and $A_{high}$, respectively 
 $ B_{low}$ and $ A_{low}$ are characterized as follows:

\medskip 
(i) Case \(n\geq2\). Then the forms \(B_{high},A_{high}\) are given by
  \begin{equation}
    \label{m1}
    \begin{aligned}
  &B_{high}(\W,\W,\W)     := \ \langle \partial^{n-1} \W,\partial^{n-1} \W\rangle_{- 4n\Re \W
        + \frac12 (1+\Til ^2)\Re \W},\smallskip
\\
&A_{high}(\W,R,R) :=  - \langle \partial^{n-1} R, \Til^{-1} \partial^{n-1} R_\alpha \rangle_{- 4n\Re \W
        - \frac12 (1+\Til^2)\Re \W}\\
        &\hspace*{3.3cm}  - 2 \langle \W  \partial^{(n-1)}R, \Til^{-1} \partial^{(n-1)}R_\alpha \rangle + 2 \la \partial^{(n-2)}\W R_\alpha,\Til^{-1}\partial^{(n-1)}R_\alpha\ra,
    \end{aligned}
  \end{equation}
  whereas the forms $B_{low}$ and $A_{low}$ have symbols 
 $B_{low}(\xi,\eta,\zeta)$, $ A_{low}(\xi,\eta,\zeta)$
in the class 
\[
B_{low} \in S(d \rho^{2n-3}), \qquad A_{low} \in S(d d_1 \rho^{2n-3}) + S(\rho^{2n-2}),
\]
where \(d,\rho\) are defined as before and \(d_1 = \min\{ |\eta|,|\zeta|\}\) is the smaller of the two \(R\) frequencies.

(ii) Case \(n=1\). Then the forms \(B_{high}\), \(A_{high}\) are given by
  \begin{equation}
    \label{m11}
    \begin{aligned}
 &B_{high}(\W,\W,\W)     := \langle  \W, \W\rangle_{-4\Re \W
        +\frac12 (1+\Til ^2)\Re \W},\
\\
&A_{high}(\W,R,R) :=  - \langle  R, \Til^{-1} R_\alpha \rangle_{-4\Re\W
        - \frac12 (1+\Til^2)\Re \W} - 2 \langle R\W,\Til^{-1}R_{\alpha}\ra,
    \end{aligned}
  \end{equation}
  and the forms $ B_{low}$ and $ A_{low}$ have symbols 
 $B_{low}(\xi,\eta,\zeta)$, $ A_{low}(\xi,\eta,\zeta)$
in the class 
\[
B_{low} \in S(\rho^{-1}), \qquad A_{low} \in S(1).
\]
\end{proposition}

The remainder of this section is devoted to the proof of the above
proposition.  To start with we give a brief description of the
types of trilinear forms $B$ and $A$ that we will work with. These
trilinear forms are translation invariant so they can be described in
terms of their symbols. Precisely, one can 
represent any such trilinear form
$B(W,W,W)$ and $A(W,Q,Q)$ as
\[
\begin{aligned}
&B(W,W,W) =  \frac2{\sqrt{2\pi}}\Re\int_{\xi+\eta+\zeta = 0} B(\xi,\eta,\zeta)  \hat W(\xi) \hat W(\eta) \hat W(\zeta)\,
d\xi d\eta,\\
&A (W,Q,Q)=  \frac2{\sqrt{2\pi}}\Re\int_{\xi+\eta+\zeta = 0} A(\zeta,\xi,\eta) \hat W(\zeta) \hat Q(\xi) \hat Q(\eta) \,d\xi d\eta.
\end{aligned}
\]
At the same time, we also need trilinear forms which involve complex conjugates. However,  the functions $W$ and $Q$ are holomorphic, and thus
their Fourier transforms satisfy the relations
\begin{equation}\label{flip}
\bar{ \hat W}(-\xi) = e^{2\xi} \hat W(\xi), \qquad \bar{ \hat Q}(-\xi) = e^{2\xi} \hat Q(\xi).
\end{equation}
These relations allow us to uniquely represent all the cubic terms in
the normal form energy functional in the above form without any conjugates.
The price to pay is that we need to allow such exponentials in our symbol classes.
However, this happens in a very limited way. To account for this we introduce the 
following notation

\smallskip
\begin{definition}
Given any class of symbols $S(\sigma)$ on the plane $\mc P = \{\xi+\eta+\zeta = 0\}$, we denote by
$ES(\sigma)$ the linear span of symbols in $\{S(\sigma), e^{\pm 2\xi} S(\sigma), e^{\pm 2\eta} S(\sigma), e^{\pm 2\zeta} S(\sigma) \}$.
\end{definition}
\smallskip

We note than any trilinear form with symbols in the class $ES$,
acting on holomorphic functions, can be written 
as a sum of trilinear forms with symbols in $S$, but where complex 
conjugation is also allowed. 

Also we note that for any such trilinear form, its symbol is uniquely determined 
up to symmetries, i.e., for the symmetric part
of the above symbols.  Indeed, symmetrizations will play a crucial role in our computations 
because they will allow us  to gain some critical cancellations.

The aim of this section is to determine the symbols $A$, $B$ above,
and to study their properties. 

\subsection{From normal forms to normal form energies}
As a first step in the proof of the proposition, here we obtain 
a preliminary normal form energy $ \tEnnf(W,Q)$ of the form
\[
\tEnnf(W,Q) = E_0 (\partial^n \tilde W, \partial^n \tilde Q),
+ g \tB(W,W,W) + \tA(W,Q,Q)
\]
so that the key property \eqref{dt-quartic} holds.
The natural expression for the normal form energy is provided by the normal form 
transformation computed in the previous section. Precisely, we will take
\[
\begin{split}
\tEnnf(W,Q) = & \  \Poly{\leq 3} E_0 (\partial^n \tilde W, \partial^n \tilde Q)
\\ = & \  E_0(\partial^n W, \partial^n Q) + 2g \la \partial^n W,  \partial^n W^{[2]} \ra - 2
\la  \Til^{-1} \partial^{n+1} Q,  \partial^n Q^{[2]} \ra.
\end{split}
\]
In view of the equations \eqref{NF-eq} the property \eqref{dt-quartic}  is automatically satisfied.
It remains to express the trilinear forms above involving the normal 
form corrections $W^{[2]}$, $Q^{[2]}$ as trilinear forms  $\tB(W,W,W)$ and   $\tA(W,Q,Q)$.

Given the expressions  for  $W^{[2]}$ and $Q^{[2]}$, the trilinear form  $\tB$  is as follows:
\[
\begin{split}
\tB(W,W,W) = &  \frac2{\sqrt{2\pi}} \Re \int_{\xi+\eta+\zeta = 0}  \zeta^{2n} ( \bar{\hat W}(-\zeta)) - \hat W(\zeta)) 
B^h(\xi,\eta) \hat W(\xi) \hat W(\eta)\, d\xi d\eta 
\\
& \ +    \frac2{\sqrt{2\pi}}\Re \int_{\xi+\eta+\zeta = 0} \zeta^{2n} ( \bar{\hat W}(-\zeta)) - \hat W(\zeta)) 
B^a(\xi,-\eta) \hat W(\xi) \bar {\hat  W}(-\eta)\, d\xi d\zeta.
\end{split}
 \]

%%%

We can put these two integrals together
using the relation \eqref{flip} to obtain 
\[
\tB(\xi,\eta,\zeta) =   (e^{2\zeta}-1)  \zeta^{2n} ( B^h(\xi,\eta) +  e^{2\eta} B^a(\xi,-\eta)).
\]
Further, we can symmetrize $\tB$ with respect to the three variables, as well as  with respect to the reflection symmetry\footnote{Here we use the fact that the trilinear forms $\tA$, $\tB$ are real valued.} 
\[
\tB(\xi,\eta,\zeta) \to \bar \tB(-\xi,-\eta,-\zeta).
\]
We denote the symmetrization of  $\tB$ by $\tB^{sym}$, which can be used instead of $\tB$. As mentioned before, this symmetrization is very important, not only in order to uniquely describe the trilinear 
form, but also because it allows us to eliminate small denominators
in the symbol for $\tilde B$ (even though such singularities do appear
in the normal form).

We can perform a similar computation for $\tA$:
\[
\begin{split}
\tA(W,Q,Q) = & \  \frac2{\sqrt{2\pi}}\Re \int_{\xi+\eta+\zeta = 0}  \zeta^{2n} ( \bar{\hat W}(-\zeta)) - \hat W(\zeta)) 
C^h(\xi,\eta) \hat Q(\xi) \hat Q(\eta)\,d\xi d\eta 
\\
& \ +   \frac2{\sqrt{2\pi}}\Re \int_{\xi+\eta+\zeta = 0} \zeta^{2n} ( \bar{\hat W}(-\zeta)) - \hat W(\zeta)) 
C^a(\xi,-\eta) \hat Q(\xi) \bar {\hat  Q}(-\eta) \,d\xi d\zeta
\\ 
& \ +  \frac2{\sqrt{2\pi}}\Re \int_{\xi+\eta+\zeta = 0}  \coth\zeta\ \zeta^{2n+1} ( \bar{\hat Q}(-\zeta)) - \hat Q(\zeta)) 
A^h(\xi,\eta) \hat W(\xi) \hat Q(\eta) \,d\xi d\eta
\\
& \ +   \frac2{\sqrt{2\pi}}\Re \int_{\xi+\eta+\zeta = 0}\coth\zeta\ \zeta^{2n+1} ( \bar{\hat Q}(-\zeta)) - \hat Q(\zeta)) 
A^a(\xi,-\eta) \hat W(\xi) \bar {\hat  Q}(-\eta) \,d\xi d\zeta
\\
& \ +   \frac2{\sqrt{2\pi}}\Re \int_{\xi+\eta+\zeta = 0}\coth\zeta\ \zeta^{2n+1} ( \bar{\hat Q}(-\zeta)) - \hat Q(\zeta)) 
D^a(\xi,-\eta) \hat Q(\xi) \bar {\hat  W}(-\eta) \,d\xi d\zeta.
\end{split}
\]
This yields the symbol for $\tA$, namely
\[
\begin{split}
\tA(\zeta,\xi,\eta)
= & \ \zeta^{2n} (e^{2\zeta} -1) \left(  C^h(\xi,\eta) + e^{2\eta} C^a(\xi,-\eta)\right) \\ & \ +  \xi^{2n+1} (e^{2\xi} + 1) 
\left( A^h(\zeta,\eta) + e^{2\eta} A^a( \zeta,-\eta) +  e^{2\zeta} D^a(\eta,-\zeta) \right)  .
\end{split}
\]
Again, this can be further symmetrized with respect to $\xi$ and
$\eta$, as well as with respect to  the reflection symmetry to obtain the symbol $\tA^{sym}$.

\subsection{ The properties of the symbols 
$\tA^{sym}$ and $\tB^{sym}$
}

A crucial step in our analysis is to understand the properties of the
symbols $\tA^{sym}$ and $\tB^{sym}$.  In this we have two goals. In terms
of low frequencies, we want to show that we can extract factors of
$\xi \eta \zeta$, so that $\tEnnf$ depends only on the differentiated
variables $W_\alpha$ and $Q_\alpha$.  In terms of high frequencies we
seek to find the leading terms in the expansion of the symbols for $\tA$
and $\tB$ near the axis $\xi = 0$, $\eta = 0$ and $\zeta = 0$.
These are as follows:
\bigskip

\begin{lemma}~

a) The symbols $\tA^{sym}$ and $\tB^{sym}$ can be expressed in the form
\[
\tA^{sym} \in \xi \eta \zeta ES(\rho^{2n-1}), \qquad \tB^{sym} \in  \xi \eta \zeta ES(\rho^{2n-2}).
\]

b) The leading order terms in  $\tB^{sym}$  in the region $|\eta| \ll \xi$ have the form
\begin{equation}
\begin{split}
\tB^{sym} =  & \ -\frac{i}{48} e^{2\xi} \xi^{2n}\left( 8n  \eta -  \frac{(\eta+J(\eta))^2}{ J(\eta)}\right) 
- \frac{i}{48} e^{-2\zeta} \zeta^{2n}\left( 8n  \eta +  \frac{(\eta-J(\eta))^2}{ J(\eta)}\right) 
 \\&\ +  \eta  ES(d \rho^{2n-1}) .
\end{split}\label{tB-Asymptotics}
\end{equation}

c) The leading order terms in $ \tA^{sym}$ in the region $|\zeta| \ll \xi$
are as follows:
\begin{equation}
\begin{split}
\tA^{sym} = &\  \frac{i}{16}  e^{2\xi} \xi^{2n} \eta\left( 8n \zeta + \frac{J(\zeta)^2 - \zeta^2}{J(\zeta)}\right)  - \frac{i}{16}  e^{-2\eta} \eta^{2n} \xi\left( 8n \zeta - \frac{J(\zeta)^2 - \zeta^2}{J(\zeta)} \right)
\\&\ + \zeta ES(d \rho^{2n}). 
\end{split}
\end{equation}

d) The leading order terms in $ \tA^{sym}$ in the region $|\eta| \ll \zeta$
are as follows:
\begin{equation}\label{tA-Asymptotics-small-eta}
\begin{split}
\tA^{sym} =&\ \frac14 i e^{2\eta + 2\zeta} \xi^{2n+1} \eta +  \eta ES(  d^2 \rho^{2n-1}) + \eta ES(\rho^{2n}).
\end{split}
\end{equation}

\end{lemma}

\begin{proof}
  We successively establish the desired properties for $\tA^{sym}$ and
  $\tB^{sym}$.  To simplify the bookkeeping we introduce the notation
  $\esym$ to describe the relation between two symbols which have the
  same symmetrization.
\bigskip

{\em  1. The symbol  $\tB^{sym}$.} 
 We recall that
\[
\tB(\xi,\eta,\zeta) =   (e^{2\zeta}-1)  \zeta^{2n} \left( B^h(\xi,\eta) +  e^{2\eta} B^a(\xi,-\eta)\right).
\]
Using symmetries, the $B^h$ contribution to  $\tB^{sym}$ is given by 
obtained  by symmetrizing the  expression
\[
\tB^{h,sym} \esym  \zeta^{2n}  (e^{2\zeta}-1) B^h(\xi,\eta) \esym  
 2 \zeta^{2n}  \sinh^2 \zeta  B^h(\xi,\eta) \esym   
- i \zeta^{2n}  (e^{2\zeta}- e^{-2\zeta})   \frac{J(\xi) J(\eta) J(\zeta)}{\Omega(\xi,\eta,\zeta)}.
\]
This is a smooth symbol. Further, since the exponential factor is odd and all other factors are even, its symmetrization vanishes on all three diagonals.

For the $B^a$ part we simplify using the reflection symmetry,
\[
\begin{split}
\tB^{a,sym}\esym & \  \zeta^{2n-1} e^{-2 \xi} \{ 
(J(\xi+\eta) - (\xi-\eta)) B^h(\xi,\eta) + (\tanh \xi -\tanh \eta)C^h(\xi,\eta) \}
\\ \esym & \  \ i \zeta^{2n}(e^{2\xi}- e^{-2 \xi}) \frac{J(\xi) J(\eta) J(\zeta)}{\Omega(\xi,\eta,\zeta)}
 + i \frac{\zeta^{2n}( e^{2\xi}+ e^{-2 \xi})}{2\Omega(\xi,\eta,\zeta)}  K(\xi,\eta),
 \end{split}
\]
where
\[
K(\xi,\eta) =    2(\xi-\eta) J(\xi) J(\eta) - (\eta J(\xi) - \xi J(\eta)) (J(\zeta) - J(\xi) - J(\eta)). 
\]
The symmetrization of the first term vanishes on the diagonals as the
first two factors are even, respectively odd, and the fraction is fully
symmetric.  The same applies for the last term, where all we need to
use for $K$ is that it is odd and antisymmetric.

Next we consider the high frequency asymptotics. Simply by considering
separately the size of each component above, we obtain $\tB^{sym} \in
ES(\rho^{2n+1})$, which suffices outside a small conical neighborhood
of the diagonals. We need to improve this near the diagonals so we
consider the case $|\eta| \ll |\xi|,|\zeta|$.  Here we need to compute
the principal part of $\tB^{sym}$ modulo lower order terms,
i.e., symbols in $ES(d^2 \rho^{2n-1})$.

The terms containing $e^{\pm 2 \eta}$ are exponentially small compared
to $e^{\pm 2\xi}$ and $e^{\pm 2 \zeta}$ so we can neglect them. We can
also neglect terms with the $\eta^{2n}$ factor.  Further, there can be
no polynomial cancellation arising from the exponentials so we might
as well consider them separately.  Hence we consider the leading order
coefficient $L_\xi$ of $e^{2\xi}$ in the region where $\xi > 0$ (and
thus $\zeta < 0$).  Neglecting lower order terms we compute
\[
\begin{split}
- i \Omega L_\xi = & \  \frac13 (- \xi^{2n} + \frac12 \zeta^{2n})  J(\xi) J(\eta) J(\zeta)  + \frac1{12} \zeta^{2n} K(\xi,\eta)
\\
= & \  - \frac13 (- \xi^{2n} + \frac12 \zeta^{2n})  \xi \zeta J(\eta)   
+ \frac1{12} \zeta^{2n} ( 2 \xi(\xi-\eta) J(\eta) - \xi( \eta - J(\eta))^2)
\\
= & \ \frac13 (\xi^{2n} - \zeta^{2n}) \xi \zeta J(\eta) - \frac1{12} \zeta^{2n} \xi(\eta+J(\eta))^2
\\
= & - \frac13 2n \xi^{2n}\zeta \eta J(\eta) -    \frac1{12} \zeta^{2n} \xi(\eta+J(\eta))^2 .
\end{split}
\]
Thus, dividing by $\Omega$ we obtain 
\[
-i  L_\xi  =   - \frac1{12} 2n \xi^{2n} \eta + \frac1{48} \zeta^{2n} \frac{(\eta+J(\eta))^2}{J(\eta)} .
\]
There is a second relevant term in the same region, namely the one with the $e^{-2\zeta}$ factor,
which is obtained by the reflection symmetry and yields the complex conjugate of the previous contribution. 
Thus we get the statement in the proposition.

 \bigskip

{\em 2. The symbol  $\tA^{sym}$.}
We recall the expression for $\tA$:
\[
\begin{split}
\tA(\zeta,\xi,\eta) 
= & \ (e^{2\zeta} -1) \zeta^{2n} C^h(\xi,\eta) 
+ ( e^{-2\xi} - e^{2 \eta}) \zeta^{2n}  C^a(\xi,-\eta) \\ & \ +  \xi^{2n+1} (e^{2\xi} + 1) 
(A^h(\zeta,\eta) + e^{2\eta} A^a( \zeta,-\eta) + e^{2\zeta} D^a(\eta,-\zeta) ).  
\end{split}
\]
Using the reflection symmetry for the first  term we have
\[
\begin{split}
\tA(\zeta,\xi,\eta) \esym & \ 2 \zeta^{2n} \sinh^2 \zeta \   C^h(\xi,\eta) 
+ ( e^{-2\xi} - e^{2 \eta}) \zeta^{2n}  C^a(\xi,-\eta) \\ & \ +  \xi^{2n+1} (e^{2\xi} + 1) 
(A^h(\zeta,\eta) + e^{2\eta} A^a( \zeta,-\eta) + e^{2\zeta} D^a(\eta,-\zeta)).  
\end{split}
\]

We first verify that the symbol $\tA^{sym}$ vanishes on the edges. The edge $\zeta = 0$ requires that 
$\tA^{sym}(0,\xi,-\xi) = 0$. This needs no computation, instead it is a consequence of the 
fact that $\tA$ above is smooth and purely imaginary. Indeed, the symmetry in $(\xi,\eta)$ 
corresponds to $\tA(0,\xi,-\xi)  \to \tA(0,-\xi,\xi)$, whereas the reflection symmetry corresponds to the transformation 
$\tA(0,\xi,-\xi) \to - \tA(0,-\xi,\xi)$. 

It remains to compute the edge $\xi = 0$, i.e., $\tA^{sym}(-\eta,0,\eta)$. In view of the symmetries
and the fact that $\tA$ is purely imaginary, we have 
\[
4 \tA^{sym}(-\eta,0,\eta) = \tA(-\eta,0,\eta) + \tA(-\eta,\eta,0) - \tA(\eta,0,-\eta) - \tA(\eta,-\eta,0).
\]
So we proceed to compute
\[
\begin{split}
\Lambda(\eta) \tA(-\eta,0,\eta) = & \ (e^{-2\eta} - 1)\eta^{2n}\Lambda(\eta)C^h(0,\eta) + (1 - e^{2\eta})\eta^{2n}\Lambda(\eta)C^a(0,-\eta)\\
=& \ i\eta^{2n + 1}\left(- e^{-2\eta}J(\eta)J'(\eta) + 2J(\eta) - \eta J'(\eta)\right) - 2i\eta^{2n+1}J(\eta)J'(\eta).
\end{split}
\]
A similar computation yields
\[
\begin{split}
\Lambda(\eta) \tA(-\eta,\eta,0) = & \  (e^{-2\eta} - 1)\eta^{2n}\Lambda(\eta)\left(C^h(\eta,0) + C^a(\eta,0)\right)\\
& \ + (e^{2\eta} + 1)\eta^{2n+1}\left(A^h(-\eta,0) + A^a(-\eta,0) + e^{-2\eta}D^a(0,\eta)\right)\\
=& - i\eta^{2n+1}\left(-e^{-2\eta}J(\eta)J'(\eta) + 2J(\eta) - \eta J'(\eta)\right)\\
&\ - i(4 + 3e^{2\eta}+ 3e^{-2\eta})\eta^{2n+1}J(\eta)J'(\eta) i( e^{2\eta} - e^{-2\eta})\eta^{2n+1}(2J(\eta) - \eta J'(\eta)).
\end{split}
\]
Combining these two we  get $\tA^{sym}(-\eta,0,\eta) = 0$.

Finally we compute the high frequency asymptotics for $\tA^{sym}$. 
To be precise, we have $\tA^{sym} \in ES(\rho^{2n+2})$ and we compute 
its symbol modulo lower order terms in $ES(d^2\rho^{2n})$ near the edge 
$\zeta = 0$, respectively $ES(d^2d_1\rho^{2n-1}) + ES(d_1\rho^{2n})$ near the edges
$\xi = 0$ and $\eta = 0$.
Here $A^a$ and $C^a$ do not contribute to the principal part so we drop them. 

First we consider the case when $\zeta$ is small and $\xi$ and $\eta$
are large. Neglecting terms with a $\zeta^2$ factor we are left with
\[
  \xi^{2n+1} (e^{2\xi} + 1)  (A^h(\zeta,\eta)  + e^{2\zeta} D^a(\eta,-\zeta) )  .
\]
We only need to retain the factors with $e^{\pm 2\xi}$ 
and $e^{\pm 2 \eta}$, which leaves us with 
\[
  \xi^{2n+1} ( e^{2\xi} A^h(\zeta,\eta) +  e^{-2\eta} D^a(\eta,-\zeta) )  .
\]
 In view of the symmetries it suffices to compute the coefficient $L_\xi$
of $ e^{ 2\xi}$  when $\xi > 0$. This is given by,  after symmetrization,
\[
\begin{split}
L_\xi = & \ \frac14  (\xi^{2n+1}  A^h(\zeta,\eta)  +  \eta^{2n+1}  D^a(-\xi,\zeta)) 
\\ = & \ \frac14 \left[
- i \xi^{2n+1}   \left( \eta +\frac{J(\zeta)^2 -\zeta^2}{4J(\zeta)} \right)   + i \eta^{2n+1} 
 \xi \right]
\\ = & \  i  \eta \xi^{2n} \left( \frac{n}2 \zeta   + \frac{J(\zeta)^2 -\zeta^2}{16 J(\zeta)} \right),
\end{split}
\]
as required in the proposition.

Lastly, we consider the case when $\eta$ is small, neglecting $A^a$, $C^a$ and
all the $\eta^3$ terms. Here $D^a$ also does not contribute. Thus, as both $A^h$ and $C^h$ 
are odd and purely imaginary,  
applying the symmetries we need to consider the expression
\[
\begin{split}
  \frac12\zeta^{2n} (e^{2\zeta} + e^{-2\zeta})  C^h(\xi,\eta) 
  +  \frac14  \xi^{2n+1} (e^{2\xi}-e^{-2\xi}) A^h(\zeta,\eta) .    
\end{split}
\]
By symmetry it suffices to consider the case that \(\zeta>0\) and \(\xi<0\). Thus, the leading order terms in the region \(|\eta|\ll\zeta\) are given by
\[
e^{2\zeta}\left(\xi^{2n}e^{2\eta} - \zeta^{2n} \right) \frac{i \eta(J(\eta) + \eta)}{8J(\eta)}\left(\xi - \frac{(J(\eta) - \eta)^2}{4J(\eta)}\right).
\]
Using that
\[
(e^{2\eta} - 1)\frac{i\eta(J(\eta) + \eta)}{8J(\eta)} = \frac14i\eta e^{2\eta},
\]
and ignoring lower order terms we are left with
\[
\frac14i\xi^{2n+1}\eta e^{2\eta + 2\zeta} - \frac n4i\xi^{2n}\eta^2  e^{2\eta + 2\zeta}\frac{J(\eta) - \eta}{J(\eta)} - \frac1{16}i\xi^{2n}\eta e^{2\eta + 2\zeta}\frac{(J(\eta) - \eta)^2}{J(\eta)}.
\]
The second and third symbols yield contributions in the class \(\eta ES(\rho^{2n})\) and hence they can be neglected. Thus, we are left with only the leading term,
\[
\frac14i\xi^{2n+1}\eta e^{2\eta + 2\zeta},
\]
and the final claim of the lemma follows.
% %

\end{proof}
\medskip

\subsection{ The high-low decomposition of $\tA$ and $\tB$}
The normal form energy is conserved to quartic order but, as our
problem is quasilinear, we expect that its time derivative will
contain more derivatives of $(W,Q)$ than we want. The idea is then to
remedy this issue by adding quartic (and higher order) quasilinear
corrections to the normal form energy. Fortunately, in this problem it
suffices to correct only the leading order terms in the normal form
energy.  Because of this, it is convenient to split the normal form
energy into a leading part plus a lower order part,
\[
\tEnnf = \tilde E^n_{NF,high} + \tilde E^n_{NF,low},
\]
which corresponds to the decomposition of the trilinear forms $\tA$ and $\tB$
as 
\begin{equation}\label{first-hl}
\tA = \tA_{high}+\tA_{low}, \qquad \tB = \tB_{high}+\tB_{low}.
\end{equation}
The decomposition of the symbols is already given in the previous lemma,
here we just compute the terms in the leading order part.
To understand this decomposition it is useful to separate the generic case 
$n \geq 2$ from $n = 1$. For larger $n$ we have:

\begin{lemma}\label{l:nf-high}
Let $n \geq 2$. Then the trilinear forms $\tA$, $\tB$ admit a decomposition as in 
\eqref{first-hl} where the symbols of $\tA_{low}$, $\tB_{low}$ satisfy 
\begin{equation}\label{low}
\tB_{low} \in \xi \eta \zeta ES(d \rho^{2n-3}), \qquad \tA_{low} \in \xi \eta \zeta ES(dd_1 \rho^{2n-3}) + \xi\eta\zeta ES(\rho^{2n-2}),
\end{equation}
and the forms $\tA_{high}$, $\tB_{high}$ are given by
\begin{equation} \label{eq-high}
\begin{aligned}
 &\tB_{high}(W,W,W)  =  \ \langle W^{(n)}, W^{(n)} \rangle_{- 4 n \Re W_\alpha + \frac12(1 +\Til^2) \Re W_\alpha},
 \\
&\tA_{high}(W,Q,Q)=
  \ - \langle Q^{(n)}, \Til^{-1} Q^{(n+1)}\rangle
_{ - 4 n \Re W_\alpha - \frac12(1 +\Til^2) \Re W_\alpha}
\\ &\hspace{3.3cm} + 2 \langle Q_\alpha W^{(n)}, \Til^{-1}Q^{(n+1)} \rangle  + 2n \langle Q_{\alpha\alpha} W^{(n-1)}, \Til^{-1} Q^{(n+1)} \rangle.
\end{aligned}
\end{equation}

\end{lemma}

On the other hand for $n = 1$ we have the more accurate result
\begin{lemma}\label{l:nf-high1}
Let $n =1 $. Then the trilinear forms $\tA$, $\tB$ admit a decomposition as in 
\eqref{first-hl} where $\tA_{low}$, $\tB_{low}$ satisfy 
\begin{equation}\label{low1}
\tB_{low} \in \xi \eta \zeta ES( \rho^{-1}), \qquad \tA_{low} \in \xi \eta \zeta ES(1),
\end{equation}
and $\tA_{high}$, $\tB_{high}$ are given by
\begin{equation}\label{eq-high1}
\begin{aligned}
 &\tB_{high}(W,W,W)  = 
 \ \langle W_\alpha, W_\alpha \rangle_{- 4\Re W_\alpha + \frac12(1 +\Til^2) \Re W_\alpha},
\\
&\tA_{high}(W,Q,Q)=
  \ - \langle Q_\alpha, \Til^{-1} Q_{\alpha \alpha} \rangle_{- 4\Re W_\alpha - \frac12(1 +\Til^2) \Re W_\alpha}.
\end{aligned}
\end{equation}

\end{lemma}

We remark that the difference in sign  in the coefficient of $\dfrac12(1 +\Til^2) \Re W_\alpha$
above  accounts exactly for the linear part of the
normal derivative of the pressure, namely \(\ao\).  The second line in $\tA_{high}$ in \eqref{eq-high}
is also natural and is due to the fact that  $(W,Q)$ is not a good set of variables for the 
differentiated equations. Instead, in the next subsection we switch from $Q_\alpha$
to the diagonal variable $R$ and the bulk of these terms will disappear.

\begin{proof}[Proof of Lemma~\ref{l:nf-high}]
We successively consider all the contributions in the leading part of $\tA^{sym}$
and $\tB^{sym}$.

\medskip

{\em 1. The contribution of $\tB_{high}$}. 
This is given by the symbol 
\[
\tB_{high} = - \frac{i}{48} e^{2\xi} \xi^{2n}\left( 8n  \eta - \frac{(\eta+J(\eta))^2}{
J(\eta)}\right) + \text{symmetries}.
\]
There are twelve symmetries, and after applying them all we obtain
\begin{equation}\label{high1}
- \int  |W^{(n)}|^2 \left(4 n \Re W_\alpha - \frac12(1 +\Til^2) \Re W_\alpha\right)  \,d\alpha.
\end{equation}
Modulo lower order terms which can be included in $\tB_{low}$ this
agrees with the expression for $\tB_{high}$ in the lemma.

\medskip

{\em 2. The contribution of $\tA_{high}$ with high frequencies on
  $Q$}. This is given by the symbol
\[
\tA_{high} = e^{2\xi} i\xi^{2n}\eta \left( \frac n2 \zeta   + \frac{J(\zeta)^2 -\zeta^2}{12J(\zeta)} \right) + \text{symmetries}.
\]
There are four symmetries, and after applying them all we obtain
\begin{equation}\label{high2}
\int  \Re (i \bar Q^{(n+1)} Q^{(n)})  \left(4 n \Re W_\alpha + \frac12 (1 +\Til^2) \Re W_\alpha\right)\,  d\alpha.
\end{equation}
Modulo lower order terms which can be included in $\tA_{low}$ 
this agrees with the first term in the expression for $\tA_{high}$ in the lemma.
\medskip

{\em 3. The contribution of $\tA_{high}$ with high frequency on $W$}.
This is given by the symbol
\[
\begin{split}
\frac14 e^{2\eta+2\zeta} (i\xi)^{n+1}(i\zeta)^n \eta + \frac{n}{4}ie^{2\eta+2\zeta}(i\xi)^{n+1}(i\zeta)^{n-1}\eta^2 + \text{symmetries}.
\end{split}
\]
There are four symmetries so we get the expression
\begin{equation}\label{high3}
\begin{split}
& 2 \Re \int i\bar Q^{(n+1)} W^{(n)} Q_\alpha\,d\alpha + 2n \Re\int i \bar Q^{(n+1)} W^{(n-1)} Q_{\alpha\alpha}\, d\alpha ,
\end{split}
\end{equation}
which up to lower order terms is equivalent to the second line in $\tA_{high}$.

\end{proof}

\begin{proof}[Proof of Lemma~\ref{l:nf-high1}]
This follows the same steps as in the previous proof, with the only difference that some terms which were previously distinct are now combining.

{\em 1. Contribution of \(\tB_{high}\).} We may write
\[
B^{sym} = L_\xi e^{2\xi} + \text{symmetries},
\]
where the full symbol is given by
\begin{align*}
- 12i\Omega L_\xi &= - (6\zeta\eta + \xi^2)J(\xi)J(\eta)J(\zeta) - \xi \zeta J(\xi)^2J(\eta) - \xi\eta J(\xi)^2J(\zeta)\\
&\quad - \zeta\eta J(\xi)^3 - \zeta^2 J(\xi)J(\eta)^2 - \eta^2 J(\xi)J(\zeta)^2.
\end{align*}
Again it suffices to consider the region \(\xi\gg|\eta|\). A similar computation to before gives us that
\[
L_\xi e^{2\xi} = \frac16i\xi\eta\zeta e^{2\xi} - \frac1{24}i\xi\zeta (\eta + J(\eta))^2J(\eta)^{-1} e^{2\xi} + \xi\eta\zeta ES(\rho^{-1}).
\]
Applying the symmetries and observing that the leading order term is already symmetric in \(\eta,\zeta\) we obtain \(\tB_{high}\).

{\em 2. Contribution of \(\tA_{high}\).}
Again we may write
\[
\tA^{sym} = L_\xi e^{2\xi} + L_\zeta e^{2\zeta} + \textrm{symmetries},
\]
where the full symbols
\begin{align*}
-4i\Omega L_\xi & = \xi\eta(\eta^2 + \zeta^2 - 2\xi^2) J(\zeta)(J(\zeta) - J(\xi) - J(\eta)) \\
&\quad + \xi^2\eta(\zeta - \eta) J(\xi)(J(\xi) - J(\eta) - J(\zeta))\\
&\quad + \xi\eta^2(\zeta - \eta) J(\xi) (J(\xi) - J(\eta) + J(\zeta)),
\end{align*}
\begin{align*}
-4i\Omega L_\zeta & = \xi\eta(\xi^2 + \eta^2 - 2\zeta^2) J(\zeta)(J(\zeta) - J(\xi) - J(\eta)) \\
&\quad - \xi^2\eta(\xi - \eta) J(\zeta)(J(\xi) - J(\eta) - J(\zeta))\\
&\quad - \xi\eta^2(\xi - \eta) J(\zeta)(J(\xi) - J(\eta) + J(\zeta)).
\end{align*}
In the region where \(\xi\gg|\zeta|\) a similar computation to before gives us that
\begin{align*}
L_\xi e^{2\xi} &= \frac14i\xi\eta\zeta(\xi - \eta) e^{2\xi} + \frac1{16}i\xi^2\eta(J(\zeta)^2 - \zeta^2)J(\zeta)^{-1} e^{2\xi} + \xi\eta\zeta ES(1).
\end{align*}
The first term gives us part of the term
\[
- \la Q_\alpha,\Til^{-1}Q_{\alpha\alpha}\ra_{-4\Re W_\alpha},
\]
and the second term gives us
\[
- \la Q_\alpha,\Til^{-1}Q_{\alpha\alpha}\ra_{-\frac12(1 + \Til^2)\Re W_\alpha}.
\]
In the region where \(\zeta\gg|\eta|\) the coefficients \(L_\zeta\) and \(L_{-\xi}\) of \(e^{2\zeta}\) and \(e^{-2\xi}\) respectively combine to give
\[
L_\zeta e^{2\zeta} + L_{-\xi}e^{-2\xi} = -\frac14 i\xi\eta\zeta(\xi - \eta) e^{2\zeta + 2\eta} + \xi\eta\zeta ES(1),
\]
which combines with the the first part of \(\tA_{high}\) to give the rest of the term
\[
- \la Q_\alpha,\Til^{-1}Q_{\alpha\alpha}\ra_{-4\Re W_\alpha}.
\]

\end{proof}

\subsection{The normal form energy in the diagonal 
variables $(\W,R)$}

The normal form energy $\tEnnf$ constructed so far is expressed in terms of the 
variables $(W,Q)$. Since we can smoothly extract factors of $\xi \eta \zeta$ 
from the symbols $\tA$, $\tB$, it is clear that one can view both $\tA$ and $\tB$ 
as trilinear forms in $(W_\alpha,Q_\alpha)$,
\[
\tB (W,W,W) = \tB_1(\W,\W,\W), \qquad \tA (W,Q,Q) = \tA_1(\W,Q_\alpha,Q_\alpha),
\]
where their symbols satisfy
\[
\tB_1 \in ES(\rho^{2n-2}), \qquad \tA_1 \in ES( \rho^{2n-1}).
\]
The same procedure applied separately to the high frequency parts
$\tA_{high}$, $\tB_{high}$ respectively the lower order terms
$\tA_{low}$, $\tB_{low}$ yields the forms $\tA_{1, high}$, $\tB_{1,
  high}$, respectively $\tA_{1,low}$, $\tB_{1,low}$, where the former are given for \(n\geq 2\) by
(see Lemma~\ref{l:nf-high})
\begin{equation}\label{nf-t1}
\begin{split}
 \tB_{1,high}(\W,\W,\W)  = & \  \langle \W^{(n-1)}, \W^{(n-1)} \rangle_{- 4 n \Re \W + \frac12(1 +\Til^2) \Re \W}
 \\
\tA_{high}(\W,Q_\alpha,Q_\alpha)=
 & \ -  \langle \Til^{-1} Q^{(n)}_\alpha, Q^{(n-1)}_\alpha\rangle
_{ - 4 n \Re \W - \frac12(1 +\Til^2) \Re \W}
\\ &+ 2 \langle \Til^{-1}Q^{(n)}_\alpha, Q_\alpha \W^{(n-1)} \rangle  + 2n \langle \Til^{-1}Q^{(n-1)}_\alpha, Q_{\alpha\alpha} \W^{(n-1)} \rangle ,
\end{split}
\end{equation}
and the symbols
for the latter have regularity
\[
\tB_{1,low} \in ES(d \rho^{2n-3}), \qquad \tA_{1,low} \in ES(dd_1 \rho^{2n-3}) + ES(\rho^{2n-2}).
\]
In order to conclude the proof of Proposition~\ref{p:ennnf} we need one last step,
namely to further switch from $(\W,Q_\alpha)$ to the diagonal variables $(\W,R)$.
This is still a purely algebraic computation, where we only need to insure that 
the original normal form energy $\tEnnf(W,Q)$ and the new one 
$\Ennf(\W,R)$ agree to cubic order,
\[
\Poly{\leq 3} \left( \tEnnf(W,Q) - \Ennf(\W,R) \right) = 0.
\]
We caution the reader that at this point, by a slight abuse of notation,
we switch the meaning of the $\Lambda$ operators. Whereas previously these were taken with respect to the expansion in the $(W,Q)$ variables, from here on we use instead the expansion in the $(\W, R)$ variables.
It is a simple observation that the above relation has identical meaning 
in both frames of reference, and this allows for a smooth transition 
between one setting and the other.

To fulfill the above requirement each of the terms in $\tEnnf$ is treated as follows,
based on the relation $Q_\alpha = R(1+\W)$:

\begin{itemize}
\item The term $ \tB_1(\W,\W,\W)$ is left unchanged.
\item The term $  \tA_1(\W,Q_\alpha,Q_\alpha)$ is replaced by  $\tA_1(\W,R,R)$.
\item The term $ \langle W^{(n)}, W^{(n)} \rangle = \langle \W^{(n-1)}, \W^{(n-1)} \rangle$
is left unchanged.
\item The term $ \langle Q^{(n)}, \Til^{-1} Q^{(n)}_\alpha \rangle$ is replaced by 
the expression $$\Poly{\leq 3} \langle [R(1+\W)]^{(n-1)}, \Til^{-1} [R(1+\W)]^{(n-1)}_\alpha \rangle.$$
\end{itemize}
Rewriting the last expression as 
\[
\Poly{\leq 3} \langle [R(1+\W)]^{(n-1)}, \Til^{-1} [R(1+\W)]^{(n-1)}_\alpha \rangle = 
  \langle R^{(n-1)}, \Til^{-1} R^{(n-1)}_\alpha \rangle + 2  
\langle [R\W]^{(n-1)}, \Til^{-1} R^{(n-1)}_\alpha \rangle,
\]
we can write our final normal form energy as 
\[
\Ennf(\W,\! R) = E_0(\W^{(n-1)},\! R^{(n-1)}) - 2  
\langle [R\W]^{(n-1)}, \Til^{-1} R^{(n-1)}_\alpha \rangle +  \tB_1(\W,\!\W,\!\W)+ \tA_1(\W,\! R,\! R),
\]
which is as required in Proposition~\ref{p:ennnf},
with
\[
B(\W,\!\W,\!\W)= \tB_1(\W,\!\W,\!\W), \quad A(\W,R,R) = \tA_1(\W,R,R) - 2  
\langle [R\W]^{(n-1)}, \Til^{-1} R^{(n-1)}_\alpha \rangle.
\]

By construction this has all the properties in part (a) of Proposition~\ref{p:ennnf}, as well as the required symbol regularity properties for the trilinear part. It remains to 
compute the high frequency  parts $B_{high}$ and $A_{high}$. For $B_{high}$ 
there is nothing to compute, as we can take
\[
B_{high} = \tB_{1,high},
\]
with $\tB_{1,high}$ as in \ref{l:nf-high}.

For $A_{high}$ on the other hand we expand
\begin{align*}
\langle [R\W]^{(n-1)}, \Til^{-1} R^{(n-1)}_\alpha \rangle &=
\langle R^{(n-1)} \W, \Til^{-1} R^{(n-1)}_\alpha \rangle + \langle \W^{(n-1)} R, \Til^{-1} R^{(n-1)}_\alpha \rangle\\
&\quad + (n-1)\langle \W^{(n-2)} R_\alpha, \Til^{-1} R^{(n-1)}_\alpha \rangle + \text{l.o.t.}
\end{align*}
The last two terms cancel with the last two terms in $\tA_{1,high}(\W,R,R)$ so we obtain
\begin{align*}
A_{high} &= -  \langle \Til^{-1} R^{(n)}, R^{(n-1)}\rangle
_{ - 4 n \Re \W - \frac12(1 +\Til^2) \Re \W}\\
&\quad - 2 \langle (R^{(n-1)} \W, \Til^{-1} R^{(n-1)}_\alpha \rangle + 2\la R_\alpha\W^{(n-2)},\Til^{-1}R^{(n-1)}_\alpha\ra
\end{align*}
as needed in part (b(i)) of  Proposition~\ref{p:ennnf}. To complete the proof of Proposition~\ref{p:ennnf} we apply an identical computation for the case \(n=1\).

%
%%%%%%%%%%%%%%%%%%%%%%%%%%%%%%%%%%%%%%%%%%%%%%%%%%%%%%%%%%%%%%%%%%%%%%%%%%%%%%%%%
\section{Higher order energy
  estimates}%%%%%%%%%%%%%%%%%%%%%%%%%%%%%%%%%%%%%%%%%%
\label{s:ee}%%%%%%%%%%%%%%%%%%%%%%%%%%%%%%%%%%%%%%%%%%%%%%%%%%%%%%%%%%%%%%%%%%%%%
%%%%%%%%%%%%%%%%%%%%%%%%%%%%%%%%%%%%%%%%%%%%%%%%%%%%%%%%%%%%%%%%%%%%%%%%%%%%%%%%%

The main goal of this section is to establish two energy bounds for
$(\W,R)$, and their higher derivatives. Precisely, we will seek to obtain first short 
and then long time bounds for the time dependent quantities
\[
\norm_n:=\Vert (g^\frac12 \W, R)\Vert_{H^{n-1} \times H^{n - \frac{1}{2}}}, \qquad n \geq 1.
\]
For $n = 0$ we will instead set 
\[
\norm_0:=\Vert (W, Q)\Vert_{\H},
\]
which is closely related to the conserved energy.

 Our first result is  a quadratic bound, which applies to all solutions independently of the size of
the initial data.  This is needed for our local well-posedness result
in Theorem~\ref{t:loc}. Precisely, the large data result is as follows:

\begin{proposition}\label{t:en=large}
  For any $n \geq 1$ there exists an energy functional $\End(\W,R)$ with the
  following properties:

 (i) Norm equivalence:
  \begin{equation*}
    E^{n,(2)}(\W,R) \approx_A  E_{0} (\partial^{n-1} \W, \partial^{n-1} R)  + O_A (\norm_{n-1}^2) .
  \end{equation*}

  (ii) Quadratic energy estimates for solutions to
  \eqref{ww2d-diff-unp}:
  \begin{equation*}
    \frac{d}{dt} \End(\W,R)  \lesssim_A B \norm_n^2 .
  \end{equation*}
\end{proposition}
Here we allow lower order errors in the energy equivalence,
and thus, the bound for $\norm_k$ for instance is obtained 
by reiterating the above estimates for $1 \leq n \leq k$, and 
using the energy conservation as a starting point which corresponds 
to $n = 0$.

 Our  second estimate is a cubic bound which only applies
for small solutions, and is used to prove our cubic lifespan result 
in Theorem~\ref{t:long}. The small data result is as follows:

\begin{proposition}\label{t:en=small}
  For any $n \geq 1$ there exists an energy functional $\Ent$ which
  has the following properties as long as $A \ll 1$:

  (i) Norm equivalence:
  \begin{equation}\label{en-cubic}
    \Ent (\W,R)=  E_{0} (\partial^{n-1} \W, \partial^{n-1} R) + 
O(A) \norm_n^2 .
  \end{equation}

  (ii) Cubic energy estimates:
  \begin{equation}\label{dt-cubic}
    \frac{d}{dt} \Ent (\W,R)  \lesssim_A AB  \norm_n^2 .
  \end{equation}
\end{proposition}

The first step in the analysis will be to isolate the main part of the
systems for $(\W, R)$ and for their derivatives, and derive quadratic
energy estimates for it.  A key part in this will be played by the 
model system studied in Section \eqref{s:model}. This model system plays the same role in this paper as the linearized
system played in the analysis of  the infinite depth water waves, (see\cite{HIT}). However, this correspondence is incomplete, in that 
here we do not have cubic estimates for the linearized system, but we
do have them for the system in $(\W, R)$ and also for its higher derivatives.

We will first differentiate the equations, and 
prove the large data result using the bounds for the model problem
in Proposition~\ref{p:local}. Then we consider the small data 
problem, and combine the prior high frequency analysis with the 
normal form energy derived in the previous section.

\subsection{The case $n=1$.}  We begin by looking at the $(\W, R)$ system \eqref{DiagonalQuasiSystem-re}.
This is a self-contained  diagonal system in these variables, which we rewrite in a form which is similar to the  model problem
in Proposition~\ref{p:local}:
\begin{equation}
  \label{ww2d-diff-unp}
  \left\{
    \begin{aligned}
      &\W_t + b\W_\alpha + \frac{1}{1+\bar \W}R_\alpha - \frac{R_{\alpha}}{1+\bar \W}\Til^2 \W= \mathcal{G}\\
      &R_t + bR_\alpha -\frac{(g+\mathfrak{a})\Til [\W]}{1+\W} =\mathcal{ K},
    \end{aligned}
  \right.
\end{equation}
where
\[
\mathcal{G}:=(1+\W)M  - \frac{R_{\alpha}}{1+\bar \W} (1+\Til^2) [\W] , \quad \mathcal{K}:=-2i\frac{\Im \P\left[
    R\bar{R}_{\alpha} \right]}{1+\W} -\frac{\mathfrak{a}\Til [\W]}{1+\W}.
\]
In order to view this system as an evolution in the space of holomorphic
functions,  we project \eqref{ww2d-diff-unp} onto the space of
holomorphic functions via the projection operator $\P$:
\begin{equation}
  \label{ww2d-diff-proj}
  \left\{
    \begin{aligned}
      &\W_t + b\W_\alpha + \P\left[ \frac{1}{1+\bar \W}R_\alpha \right] - \P \left[ \frac{R_{\alpha}}{1+\bar \W} \Til^2\W\right] = \P \mathcal{G}\\
      &R_t + bR_\alpha -\P \left[ \frac{(g+\mf{a})\Til [\W]}{1+\W}\right]
      =\P \mathcal{ K}.
    \end{aligned}
  \right.
\end{equation}
We also recall here the expressions of $b$, $a$ and $M$
\[
b = 2\Re\left[R - \P[R\bar Y]\right],
\]
and
\[
\mathfrak{a} = g\W + ig\Til[\W] + 2\Im\P[R\bar R_\alpha], \quad M = 2\Re\P[R\bar
Y_\alpha - \bar R_\alpha Y].
\]

To this system we associate the positive definite linear functional
energy  $E^{(2)}_{lin} (\W, R)$ given by \eqref{QuadLinNRG}.
The main result of this subsection establishes energy bounds for the
system \eqref{ww2d-diff-proj}, thus proving the $n=1$ part of 
Proposition~\ref{t:en=large}

\begin{proposition} The above energy applied to solutions of 
projected system \eqref{ww2d-diff-proj} satisfies the  following estimates:
\begin{itemize}
\item[i)] Norm equivalence:
\[
E^{(2)}_{lin} (\W, R) \approx_A \|(\W,R)\|_{\H}^2.
\]
\item[ii)] Cubic energy estimates:
  \[
  \frac{d}{dt}E^{(2)}_{lin} (\W, R)\lesssim_A B \norm_1^2.
  \]
  \end{itemize}
\end{proposition}

Here the energy equivalence follows directly from the positivity 
and boundedness of $\mathfrak{a}$, see Proposition~\ref{TS} and  Lemmma~\ref{l:a}.
The second estimate in the proposition relies on the
estimates obtained in \eqref{est1}. Precisely, in order to obtain
the quadratic energy estimates for the large data, it suffices to
prove a priori bounds for $\Vert (\P \mathcal{G}, \P
\mathcal{K})\Vert_{\mathcal{H}}$. These a priori bounds will be in
terms of the pointwise control norms $A$, $B$ and the energy
$E^{(2)}_{lin} (\W, R)$:

\begin{lemma}
  \label{quadratic, n=1} The following estimates for the lower order
  terms $(\P \mathcal{G}, \P \mathcal{K})$:
  \begin{equation}
    \label{quad-energy}
    \Vert (\P \mathcal{G}, \P \mathcal{K})\Vert_{\mathcal{H}} \lesssim _{A} B \norm_1.
  \end{equation}
\end{lemma}
\begin{proof} We begin with the estimate for $\mathcal{G}$:
  \[
  \Vert \P \mathcal{G}\Vert_{\H}\lesssim \Vert \P M\Vert_{\H}+\Vert \P
  [\W M]\Vert _{L^2}\lesssim_A (B+AB) \Vert R\Vert_{H^{\frac{1}{2}}}.
  \]
  To bound $\mathcal{K}$ we estimate each of the terms separately. Thus,
  for the first term we estimate
  \[
  \begin{aligned}
    \Vert \left\langle D \right\rangle ^{\frac{1}{2}} \left[ \P \left[
        \Im \bar{\P}[\bar{R}R_{\alpha}]\right](1-Y) \right]\Vert _{\H}
    &\lesssim \Vert \left\langle D \right\rangle ^{\frac{1}{2}}
    \bar{\P}[\bar{R}R_{\alpha}] \Vert _{\H} +
    \Vert \left\langle D \right\rangle ^{\frac{1}{2}}  \bar{\P}[\bar{R}R_{\alpha}]Y\Vert _{\H}\\
    &\lesssim \Vert R\Vert_{\bmo^{1}} \Vert R\Vert_{H^\frac{1}{2}} +\Vert Y\Vert_{\bmo^{\frac{1}{2}}}\Vert \left\langle D \right\rangle ^{\frac12}R\Vert_{L^{\infty}}\Vert R\Vert_{H^{\frac12}}\\
    &\lesssim _A (B+AB) \Vert R\Vert_{H^{\frac12}}.
  \end{aligned}
  \]
 This is a direct consequence of the commutator estimate \eqref{PComm}
  together with the $Y$ estimate derived in \eqref{est:Ybmo}. As for the
  second term, we use the estimates derived for $\mathfrak{a}$ and $Y$ in
  \eqref{est:a} respectively \eqref{est:Ybmo}, to arrive at
  \[
  \Vert \mf{a}\Til [\W](1-Y)\Vert_{H^{\frac12}}\lesssim (\Vert
  \mf{a}\Vert_{\bmo^{\frac12}}+\Vert \mf{a}\Vert_{L^{\infty}}\Vert
  Y\Vert_{\bmo^{\frac12}}) \Vert \W\Vert _{L^2}.
  \]
  \end{proof}

  For the small data problem it is of further interest to track the
  solutions on larger time scales in order to prove Proposition~\ref{t:en=small}. This is done at the end of this section.

%%%%%%%%%%%%%%%%%%%%%%%%%%%%%%%%%%%%%
%%%%%%%%%%%%%%%%%%%%%%%%%%%%%%%%%%%%%
%%%%%%%%%%%%%%%%%%%%%%%%%%%%%%%%%%%%%
%%%%%%%%%%%%%%%%%%%%%%%%%%%%%%%%%%%%%

\subsection{The case $n=2$} 
 We recall that the system \eqref{DiagonalQuasiSystem-re}  for
$(\W , R)$ is given by
\begin{equation*}
  \left\{
    \begin{aligned}
      &\W_t + b\W_\alpha + \frac{1}{1+\bar \W}R_\alpha + \frac{R_{\alpha}}{1+\bar \W}\W= (1+\W)M\\
      &R_t + bR_\alpha -\frac{g\Til [\W]}{1+\W} +\frac
      {ia}{1+\W}=0,
    \end{aligned}
  \right.
\end{equation*}
where $a$ is the same as in the infinite depth gravity water waves
\[
 a:=2\Im\P[R\bar{R}_{\alpha}].
\]

We differentiate with respect to $\alpha$ in order to obtain 
a system for $(\W_\alpha,R_\alpha)$,
\begin{equation*}
  \left\{
    \begin{aligned}
      &\W_{\alpha t} + b\W_{\alpha \alpha}  + \frac{\left[ (1+\W)R_{\alpha}\right] _{\alpha}}{1+\bar \W}= -b_{\alpha}\W_{\alpha} -(1+\W)R_{\alpha}\bar{Y}_{\alpha}+\W_{\alpha}M+(1+\W) M_{\alpha}\\
      &R_{\alpha t} + bR_{\alpha \alpha} -\frac{g\Til [\W
        _{\alpha}]}{1+\W} +\frac{g\Til [\W]}{(1+\W)^2}\W_{\alpha}
      +\frac {ia_{\alpha}}{1+\W}
      -\frac{ia}{(1+\W)^2}\W_{\alpha}=-b_{\alpha}R_{\alpha
      },
    \end{aligned}
  \right.
\end{equation*}
and rewrite it as follows
\begin{equation*}
  \left\{
    \begin{aligned}
      &\W_{\alpha t} + b\W_{\alpha \alpha}  + \frac{\left[ (1+\W)R_{\alpha}\right] _{\alpha}}{1+\bar \W}= -b_{\alpha}\W_{\alpha} -(1+\W)R_{\alpha}\bar{Y}_{\alpha}+\W_{\alpha}M+(1+\W) M_{\alpha}\\
      &R_{\alpha t} + bR_{\alpha \alpha} -\left[ \frac{(g+\mathfrak{a})\Til
          [\W_{\alpha}]}{(1+\W)^2}\right]=\frac{ia-g\Til
        [\W]}{(1+\W)^2}(1+i\Til)\W_{\alpha}-\frac
      {ia_{\alpha}}{1+\W}-b_{\alpha}R_{\alpha }.
    \end{aligned}
  \right.
\end{equation*}
We recall that
\begin{equation}
\label{def-M}
M=\frac{R_{\alpha}}{1+\bar{\W}}+\frac{\bar{R}_{\alpha}}{1+\W}-b_{\alpha},
\end{equation}
and use this definition to simplify the system above
\begin{equation*}
  \left\{
    \begin{aligned}
      &\W_{\alpha t} + b\W_{\alpha \alpha}  + \frac{\left[ (1+\W)R_{\alpha}\right] _{\alpha}}{1+\bar \W}=\left(  M-\frac{R_{\alpha}}{1+\bar{\W}}-\frac{\bar{R}_{\alpha}}{1+\W}\right) \W_{\alpha} \\
      &\hspace{7cm}-(1+\W)R_{\alpha}\bar{Y}_{\alpha}+\W_{\alpha}M+(1+\W) M_{\alpha}\\
      &R_{\alpha t} + bR_{\alpha \alpha} -\left[  \frac{(g+\mfa)\Til [\W_{\alpha}]}{(1+\W)^2}\right]=\frac{ia-g\Til [\W]}{(1+\W)^2}(1+i\Til)\W_{\alpha}-\frac {ia_{\alpha}}{1+\W}\\
      &\hspace{7cm}+\left(
        M-\frac{R_{\alpha}}{1+\bar{\W}}-\frac{\bar{R}_{\alpha}}{1+\W}\right)
      R_{\alpha }.
    \end{aligned}
  \right.
\end{equation*}

To align this more closely with the linearized equation and express the system in a manner similar to \cite{HIT}, we introduce the auxiliary holomorphic variable 
\[
\R = (1+\W) R_{\alpha}.
\]
Then it becomes
\begin{equation*}
  \left\{
    \begin{aligned}
      &\W_{\alpha t} + b\W_{\alpha \alpha}  + \frac{\R _{\alpha}}{1+\bar \W} +\frac{R_{\alpha}}{1+\bar{\W}} \W_{\alpha}=-\frac{\bar{R}_{\alpha}}{1+\W} \W_{\alpha}+\R\bar{Y}_{\alpha}+2\W_{\alpha}M+(1+\W) M_{\alpha}\\
      &\\
      &\R_t+b \R_{\alpha}-\left[  \frac{(g+\mfa)\Til [\W_{\alpha}]}{(1+\W)}\right]=\frac{ia-g\Til [\W]}{(1+\W)}(1+i\Til)\W_{\alpha} + (\bar{R}_{\alpha}R_{\alpha}-ia_{\alpha} )+2\R M\\
      &\hspace{10cm} -2\left(
        \frac{R_{\alpha}}{1+\bar{\W}}+\frac{\bar{R}_{\alpha}}{1+\W}\right)
      \R .
    \end{aligned}
  \right.
\end{equation*}
Here, we have isolated on the left the leading part of our
equations. The goal is to interpret the terms on the right hand side
as perturbative (with one exception, which is 
only due to the low regularity setting, see below).  In addition, for the cubic bound we will   also need to pay attention to the quadratic part of the terms in the equations.

In order to simplify our bookkeeping we define two types of error terms for the above system. These are denoted by $\errw $
and $\errr $, which correspond to the two equations. A similar strategy was employed in \cite{HIT}.
However, unlike in \cite{HIT}, here we also include bounded quadratic terms into the error, rather than explicitly keeping track of them. This simplifies the argument somewhat, at the expense of getting a less precise expression for the normal form energy.

 The bounds for these errors are in terms of the control variables $A$, $B$, as well
as the $L^2$ type norm $\norm_2$, where
\[
\norm_2:=\Vert (g^{\frac{1}{2}}\W, R)\Vert_{H^1 \times H^{\frac{3}{2}}}.
\]

The acceptable errors in the $\W_{\alpha}$ equation are denoted, by
$\errw$ and are of two types, $\errw^{[2]}$ and $\errw^{[3]}$. The
first one, $\errw^{[2]}$, consists of  quadratic terms
which satisfy the bounds
\[
\Vert \Til \P\mathcal{G}\Vert_{L^2}\lesssim B\norm_2, \quad  \Vert \mathcal{G}\Vert_{H^{-\frac12}}\lesssim A\norm_2 .
\]
By $\errw^{[3]}$ we denote the cubic and higher counterpart  of $\errw^{[2]}$, which contains terms  $\mathcal{G}$
which satisfy the estimate
\[
\Vert \Til \P \mathcal{G}\Vert_{L^2}\lesssim _A AB\norm_2, \quad \Vert  \mathcal{G}\Vert_{H^{-\frac12}}\lesssim _A A^2\norm_2.
\]

The acceptable errors in the $\R$ equation are denoted by $\errr$ and are of two types, $\errr ^{[2]}$ and $\errr
^{[3]}$. The first one, $\errr^{[2]}$, consists of quadratic terms $\mathcal{K}$
that  satisfy the bounds
\[
\Vert \Til \P \mathcal{K}\Vert_{H^{\frac{1}{2}}}\lesssim B\norm_2, \quad \Vert
\mathcal{K}\Vert_{L^2}\lesssim A\norm_2.
\]
By $\errr^{[3]}$ we denote terms in $K$ which satisfy the estimates
\[
\Vert \Til \P  \mathcal{K}\Vert_{H^{\frac{1}{2}}}\lesssim _A AB\norm_2,
\quad \Vert \mathcal{K}\Vert_{L^2}\lesssim _A A^2\norm_2.
\]

\begin{remark}
  Compared to \cite{HIT}, above we define $\norm_2$ in a more relaxed, inhomogeneous fashion.
This is in part caused by the lack of scaling. It is  reasonable  because here we work with the system for the
  differentiated variables $(\W, R)$ or their higher counterparts, which is 
  used to bound the high frequencies of the solutions.
\end{remark}

A key property of the space of errors is contained in the following lemma:
\begin{lemma}
\label{eroare}
Let $\Phi$ be a function which satisfies
\begin{equation}
\Vert \Phi \Vert _{L^{\infty}} \lesssim A, \quad \Vert  \Phi \Vert_{\bmo^{\frac{1}{2}}}\lesssim B. 
\end{equation}
Then, we have the multiplicative bounds
\begin{equation}
\Phi\cdot \errw= \errw , \quad \Phi \cdot \errr =\errr.
\end{equation}

\end{lemma}
The proof of the lemma is relatively straightforward and is left for
the reader.

We now return to system above and expand some of the terms. We begin with the terms containing $M$. For
this we will make use of the bounds we have established for $M$ in the 
Appendix:
\begin{equation}
  \label{M-needed}
  \begin{aligned}
    & \| M\|_{L^\infty} \lesssim AB, \quad \Vert
    M\Vert_{H^{\frac{1}{2}}}\lesssim A\norm_2, \quad \Vert M_\alpha \Vert_{L^2}\lesssim A\norm_2 .
  \end{aligned}
\end{equation}

Precisely, the $M$ terms in the equations satisfy
\[
M(1+\W_{\alpha}) + M_\alpha \W =\errw, \quad M\R=\errr.
\]
The first claim is a straightforward consequence of the pointwise bound
for $M$ and the $L^2$ bound for $M_\alpha$. For the second, we recall
that $\R =R_{\alpha} (1+\W)$, which  together with Lemma~\ref{eroare} allows us to only estimate
$MR_{\alpha}$. The $H^{\frac{1}{2}}$ bound for $MR_{\alpha}$ follows
after a Littlewood-Paley decomposition of the product: the bounds for the low-high and balanced interactions are a direct consequence of
\eqref{M-needed}, and the bounds for high-low interactions are obtained by combining \eqref{M-needed} and Lemma~\ref{est:TilComm}.

Next we consider the expression
\[
\frac{ia-g\Til [\W]}{(1+\W)}(1+i\Til)\W_{\alpha} = \frac{ia-g\Til [\W]}{(1+\W)}(1+\Til^2) \Re \W_{\alpha},
\]
which we claim belongs to $\errr$. To prove our claim, we split the  above expression  into a quadratic part and a cubic and higher term,
\[
-g\Til [\W] (1+\Til^2) \Re \W_{\alpha} +  ( ia+g\Til [\W]
\W)(1-Y) \left[ (1+\Til ^2)\Re \W_{\alpha}\right].
\]

We may then apply the paraproduct estimates \eqref{Para1} and
\eqref{est:StrongLoHi} with the estimates \eqref{est:a} for \(\mf{a}\)
(which also applies to the component $a$ of $\mfa$) and
\eqref{est:Ybmo}, \eqref{est:Y2} for \(Y\) to obtain:
\[
\begin{aligned}
  \Vert  (ia+g\Til [\W] \W)(1-Y) \left[  (1+\Til ^2)\Re \W_{\alpha}\right]\Vert_{H^{\frac{1}{2}}}\hspace*{4cm}\\
  \lesssim \left\lbrace A \left( \Vert a\Vert
      _{\bmo^{\frac{1}{2}}}+\Vert Y \Vert_{\bmo^{\frac{1}{2}}}\right)
    + \left( g\Vert \left\langle D\right\rangle ^{\frac{1}{2}}\W\Vert _{\bmo}\Vert \W \Vert_{L^{\infty}} +\Vert \W\Vert_{L^{\infty}}\Vert Y\Vert_{\bmo^{\frac{1}{2}}}\right) \right\rbrace  \Vert \W_{\alpha}\Vert_{L^2}.\\
  % & =\left( \Vert a_{bmo^{\frac{1}{1}}}+g\Vert Y
  %   \Vert_{bmo^{\frac{1}{2}}}\right) \norm_1.
\end{aligned}
\]
Similarly, for the quadratic part we obtain
\[
\Vert g\Til \![\W] (1+\Til^2)(\Re \W_{\alpha})\Vert_{H^{\frac{1}{2}}}\lesssim_A A\norm_2 .
\]

Next we consider the difference
\[
\bar{R}_{\alpha}R_{\alpha}-ia_{\alpha} = 2 \bar \P[ R_\alpha \bar R_\alpha] + i \Im \P [ R \bar R_{\alpha\alpha}] .
\]
For this we bound 
\[
\| \bar{R}_{\alpha}R_{\alpha}-ia_{\alpha} \|_{L^2} \lesssim A \norm_2, \qquad
\| \P \Til (\bar{R}_{\alpha}R_{\alpha}-ia_{\alpha}) \|_{H^\frac12} \lesssim B \norm_2 .
\]

Taking into account all the above bounds, it follows that our system can be rewritten in the form
\begin{equation*}
  \left\{
    \begin{aligned}
      &\!\!\W_{\alpha t} + \!b\W_{\alpha \alpha}  + \frac{\R _{\alpha}}{1+\bar \W} +\frac{R_{\alpha}}{1+\bar{\W}} \W_{\alpha}=2\R\bar{Y}_{\alpha} -2\frac{\bar{R}_{\alpha}}{1+\W} \W_{\alpha}+\errw \\
      &\\
      &\!\!\R_t+\!b \R_{\alpha}-\left[ \frac{(g+\mf{a})\Til
          [\W_{\alpha}]}{(1+\W)}\right]\!=-4\!\Re \left(
        \frac{R_{\alpha}}{1+\bar{\W}}\right) \R\! +\!\errr .
    \end{aligned}
  \right.
\end{equation*}

One might wish to compare this system to the model system
\eqref{ModelLinSys}, for which we obtained the nice energy estimates  in 
\eqref{est1}, and use these estimates to prove quadratic energy bounds
provided that the right hand side terms are bounded in $L^2$, and
$H^{\frac{1}{2}}$ respectively.

Unfortunately we still have terms on the right which cannot be  bounded as error terms, i.e., in $L^2 \times H^{\frac{1}{2}}$.
 This matches similar issues appearing in the infinite depth case in \cite{HIT}. To deal with these terms we use the same conjugation with respect to a real exponential weight $e^{2\phi}$, where $\phi=-2\Re \log
(1+\W)$, which was previously used in \cite{HIT}. When implementing
such a transformation, we are not only able to eliminate the unbounded
terms pointed out above, but we also manage to cast our system in a
similar form as the model system in \eqref{ModelLinSys}.

To see this, we compute
\[
\phi_{\alpha}=-2\Re \frac{\W_{\alpha}}{1+\W},\qquad (\partial_t
+b\partial_{\alpha})\phi =2\Re\frac{R_{\alpha}}{1+\bar{\W}}-2M.
\]
We denote the weighted variables by
\[
w:=e^{2\phi}\W_{\alpha}, \quad r:=e^{2\phi}\R.
\]
Before explicitly writing down the resulting equations, we remark that
by Lemma~\ref{eroare} we have
\[
e^{2\phi} \errw=\errw, \quad e^{2\phi}\errr =\errr,
\]
which simplifies the transformed system to
\begin{equation*}
  \left\{
    \begin{aligned}
      & w_t + b w_{ \alpha}  + \frac{r_{\alpha}}{1+\bar \W} - \frac{R_{\alpha}}{1+\bar{\W}} \Til^2 w=\errw \\
      &\\
      &r_t+b r_{\alpha}-\left[ \frac{(g+\mf{a})\Til
          [w]}{(1+\W)}\right]=\errr .
    \end{aligned}
  \right.
\end{equation*}
Here we have also harmlessly replaced $w$ by $- \Til^2 w$ in the last term on the left in the first equation. 
The difference is  easily included in the error as $1+\Til^2$ has a Schwartz symbol. This is done in order to bring 
the above equations more in line with the model linear problem.

Unfortunately our new variables $(w,r)$, are not exactly
holomorphic; the last system contains both holomorphic and
also antiholomorphic components. To remedy this issue we need to
project the system via the projection $\P$, and also work with the
projected variables $(\P w, \P r)$.  At this point one might
legitimately be concerned that restricting to the holomorphic part
would remove a good portion of our variables. However this is not the
case, as one can verify that the a similar argument as the one in
Lemma~3.4 from \cite{HIT} applies to the finite depth case:

\begin{proposition} The energy of $(\P w, \P r)$ above is equivalent
  to the energy of $(\W_{\alpha}, R_{\alpha})$: 
  \begin{equation}
    \Vert (\P w, \P r)\Vert_{\mathcal{H}}\sim _{A} \Vert (w,r)\Vert _{\mathcal{H}}\sim _{A} \Vert (\W_{\alpha},R_{\alpha})\Vert _{\mathcal{H}} \mbox{ modulo } A\norm_2 .
  \end{equation}
\end{proposition}
Unlike in \cite{HIT}, here we allow for lower order errors in order to account for the $L^2$ unboundedness
of $\P$ at low frequencies. Once we do that, it remains to prove only a high frequency bound, for which the 
same argument as in \cite{HIT} applies.

We are now ready to write the system for $(Pw,Pr)$, namely
\begin{equation}
  \label{(Pw,Pr)a}
  \left\{
    \begin{aligned}
      & \P w_t +\P \left[  b \P w_{ \alpha} \right]  +\P \left[  \frac{ \P r_{\alpha}}{1+\bar \W} \right] -\P \left[ \frac{R_{\alpha}}{1+\bar{\W}} \Til^2 \P w \right] = G_2+\P \errw \\
      &\\
      & \P r_t+\P \left[ b \P r_{\alpha}\right] -\P \left[
        \frac{(g+\mf{a})\Til [\P w]}{(1+\W)}\right]= K_2+\P \errr ,
    \end{aligned}
  \right.
\end{equation}
where $(G_2,K_2)$ contain all the additional terms,
\[
\left\{
  \begin{aligned}
    &G_2:= -\P \left[ \frac{R_{\alpha}}{1+\bar{\W}} (1+\Til^2) \P w \right]-\P \left[ b \bar{\P }w_{\alpha}  \right]-\P \left[  \frac{ \bar \P r_{\alpha}}{1+\bar \W} \right] -\P\left[ \frac{R_{\alpha}}{1+\bar{\W}}\bar{\P}w\right] \\
    &K_2:= -\P \left[ b \bar{\P
      }r_{\alpha} \right] + \P \left[ \frac{(g+\mf{a})\Til [\bar{\P} w]}{(1+\W)}\right].
  \end{aligned}
\right.
\]
The goal here is to prove that $G_2=\errw$ and $K_2=\errr$, but this
is straightforward as they all have a nice commutator structure; the proof is left for the reader. We denote the last set of variables 
by $(Pw,Pr):= (\fw,\fr)$; these solve the system 
\begin{equation}
  \label{(Pw,Pr)}
  \left\{
    \begin{aligned}
      &  \mathfrak{w}_t +\P \left[  b \fw_{ \alpha} \right]  +\P \left[  \frac{  \fr_{\alpha}}{1+\bar \W} \right] -\P \left[ \frac{R_{\alpha}}{1+\bar{\W}}  \Til^2\fw \right] = \P\, \errw \\
      &\\
      & \fr_t+\P \left[ b r_{\alpha}\right] -\P \left[ \frac{(g+\mf{a})\Til [
          \fw]}{(1+\W)}\right]= \P \, \errr .
    \end{aligned}
  \right.
\end{equation}

Therefore, we can now apply the energy bounds obtained for the toy
model \eqref{ModelLinSys} to the system \eqref{(Pw,Pr)}.
Now the result of Proposition ~\eqref{t:en=large} follows from the energy
estimates for the model system \eqref{ModelLinSys}, namely \eqref{est1};
further, if $n=2$ then we can take
\[
E^{n,(2)}(\W, R)=E^{(2)}_{lin} (\fw, \fr).
\]
The last goal is to obtain cubic lifespan bounds for the 
small data problem, which would correspond to proving Proposition~\eqref{t:en=small}.  We address this question later in this section.

\subsection{The case $n\geq 3$} We follow the same strategy as in the case $n=2$ and derive the equations for $(\W ^{(n-1)}, R^{(n-1)})$. For
this, we start with the system \eqref{DiagonalQuasiSystem-re} and differentiate
$(n-1)$ times.  For this we will estimate the errors in terms of   $\norm_n$ which measures $(n-1)$ derivatives of $\W$ and $R$, with constants that  depend on the
control norms $A$ and $B$.

The acceptable errors in the $\W^{(n-1)}$ equation are denoted, as
before, by $\errw$ and are of two types, $\errw ^{[2]}$ and $\errw
^{[3]}$. The first one, $\errw ^{[2]}$, consists of holomorphic quadratic
terms in $G$ of the form that satisfy the bound
\[
\Vert \Til \P G\Vert_{L^2}\lesssim B\norm_n \quad \mbox{ and } \Vert G\Vert_{H^{-\frac{1}{2}}}\lesssim A\norm_n.
\]
By $\errw ^{[3]}$ we denote the cubic counterpart of $\errw$ of $G$,
which satisfies the estimate
\[
\Vert \Til \P G\Vert_{L^2}\lesssim_{A}AB\norm_n, \quad \Vert  G\Vert_{H^{-\frac12}}\lesssim_{A}A^2\norm_n.
\]

The acceptable errors in the $R^{(n-1)}$ equation are denoted, as
before, by $\errr$ and are of two types, $\errr ^{[2]}$ and $\errr
^{[3]}$. The first one, $\errr$, consists of holomorphic quadratic
 terms  in $K$  that satisfy the bound
\[
\Vert \Til\P K\Vert_{H^{\frac{1}{2}}}\lesssim B\norm_n, \quad \Vert
K\Vert_{L^2}\lesssim A\norm_n.
\]

By $\errr^{[3]}$ we denote terms in $K$ which satisfy the estimates
\[
\Vert \Til \P K\Vert_{H^{\frac12}}\lesssim_A AB \norm_n, \quad \Vert 
K\Vert_{L^2}\lesssim_A A^2\norm_n.
 \]

 We begin by differentiating the terms in the $\W$ equation. For the
$b$ term, after standard estimates, we have
\[
\begin{aligned}
  \partial^{(n-1)}_{\alpha}(b\W_{\alpha})&=b\W_{\alpha}^{(n-1)}+(n-1)b_{\alpha}\W^{(n-1)}+\errw\\
  &=b\W_{\alpha}^{(n-1)}+(n-1)\left(
    \frac{R_{\alpha}}{1+\bar{\W}}+\frac{\bar{R}}{1+\W}\right)\W^{(n-1)}+\errw.
\end{aligned}
\]
Here we have used the relation \eqref{def-M}, and also  the $L^{\infty}$ bound for $M$.

Continuing, we apply the same analysis for the $(n-1)$ derivative of the next term appearing in the $\W$ equation
\begin{equation}
\label{w1}
\begin{aligned}
  \partial_{\alpha}^{(n-1)} \frac{(1+\W)R_{\alpha}}{1+\bar{\W}}= &\frac{ \left[ (1+\W)R^{(n-1)} \right] _{\alpha}}{1+\bar{\W}}+\frac{R_{\alpha}}{1+\bar{\W}}\W^{(n-1)}+ \errw.
\end{aligned}
\end{equation}
Here we have again isolated the terms which cannot be placed into the error. 

Similarly, using the bounds for $M$ in Lemma ~\ref{l:M}, the last component of the $\W$ equation is
\[
 \partial_{\alpha}^{(n-1)} \left[ (1+\W)M\right]=   (1+\W)\partial^{(n-1)}_{\alpha}M+\errw=(1+\W)\partial^{(n-1)}_{\alpha}2\Re \P \left[ R\bar{Y}_{\alpha}-\bar{R}_{\alpha}Y\right]. 
\]
Because of the differentiation, there are no low frequency issues here. Distributing derivatives inside, the terms with derivatives on the antiholomorphic factors are all errors, so we are left only with the terms where all derivatives apply to the holomorphic factors. Harmlessly discarding the projection we arrive at
\[
\begin{aligned}
\partial_{\alpha}^{(n-1)} \left[ (1+\W)M\right]&=2(1+\W)\Re  \left[ R^{(n-1)}\bar{Y}_{\alpha} -\bar{R}_{\alpha}Y^{(n-1)}\right]+ \errw \\
&= -\frac{\bar{R}_{\alpha}}{1+\W}\W^{(n-1)}+\errw. \\
\end{aligned}
\]

Next, we differentiate the $R$ equation. We begin with
\[
\begin{aligned}
  \partial_{\alpha}^{(n-1)}(bR_{\alpha})=& bR^{(n-1)}_{\alpha}+ (n-1)b_{\alpha}R^{(n-1)}+b^{(n-1)}R_{\alpha}+\errr . \\
  \end{aligned}
  \]
  In the second term we use again the relation \eqref{def-M} and the boundedness of $M$. In the third term, only the holomorphic part of $b$ yields a nontrivial contribution, and that only when all derivatives apply to $Y$. Discarding again the projection,  we obtain
   \[
  \begin{aligned}
 \partial_{\alpha}^{(n-1)}(bR_{\alpha})  =&\ bR^{(n-1)}_{\alpha}+ (n-1)\left( \frac{R_{\alpha}}{1+\bar{\W}}+\frac{\bar{R}_{\alpha}}{1+\W}\right) R^{(n-1)}+\frac{R_{\alpha}}{1+\bar{\W}}R^{(n-1)}+\errr.
\end{aligned}
\]

For the remaining terms in the $R$ equation we write
\begin{equation}
  \label{c1}
  \partial_{\alpha}^{(n-1)} \left( \frac{g\Til [\W]}{1+\W}\right) = \frac{g\Til [\W^{(n-1)}]}{1+\W}-\frac{g\Til[\W]}{(1+\W)^2}\W^{(n-1)}+\errr.
\end{equation}

Lastly, using the bound for $\mathfrak{a}$ in Lemma~\ref{l:a}, we have
\begin{equation}
  \label{c2}
  \begin{aligned}
    \partial_{\alpha}^{(n-1)}\left( \frac{i\mfa}{1+\W}\right)
    =\frac{i\mfa^{(n-1)}}{1+\W}-\frac{i\mfa}{(1+\W)^2}\W^{(n-1)}+\errr.
  \end{aligned}
\end{equation}
In the first term we can discard the $a_1$ component of $\mfa$ into the error. In the contribution of $a=2\Im \P [R\bar{R}_{\alpha}]$, only the holomorphic part has an interesting component, precisely when all the derivatives fall on $R$. Hence we obtain
\begin{equation}
  \label{c3}
  \begin{aligned}
    \frac{ia^{(n-1)}}{1+\W}=&\frac{\bar{R}_{\alpha}}{1+\W} R^{(n-1)}+\errr.
  \end{aligned}
\end{equation}
In the second term in \eqref{c2}, we substitute $i\W ^{(n-1)}$ with $-\Til [\W^{(n-1)}]$ modulo a negligible error. Thus,
together with\eqref{c1}, \eqref{c2} and \eqref{c3} we arrive at
\begin{equation*}
  \begin{aligned}
    &-\partial_{\alpha}^{(n-1)} \left( \frac{g\Til [\W]}{1+\W} -
      \frac{i\mfa}{1+\W}\right) = -\frac{(g+\mf{a})\Til
      [\W^{(n-1)}]}{(1+\W)^2}-
    \frac{\bar{R}_{\alpha}}{1+\W} R^{(n-1)}+\errr.
  \end{aligned}
\end{equation*}
Combining the above equations we obtain the differentiated system
\begin{equation*}
  \left\{
    \begin{aligned}
      &\W^{(n-1)}_t +b\W_{\alpha}^{(n-1)}+\frac{\left( (1+\W)R^{(n-1)}\right)_{\alpha} }{1+\bar{\W}}+\frac{R_{\alpha}}{1+\bar{\W}}\W^{(n-1)}=G\\
      &R_t^{(n-1)}+bR_{\alpha}^{(n-1)}-\frac{(g+\mfa)\Til
        [\W^{(n-1)}]}{(1+\W)^2}=K,
    \end{aligned}
  \right.
\end{equation*}
where
\begin{equation*}
  \left\{
    \begin{aligned}
      &G=-n\frac{\bar{R}_{\alpha}}{1+\W}\W^{(n-1)} -(n-1)\frac{R_{\alpha}}{1+\bar{\W}}\W^{(n-1)}+\errw   \\
      &K= -n\left( \frac{R_{\alpha}}{1+\bar{\W}}+\frac{\bar{R}_{\alpha}}{1+\W}\right)R^{(n-1)}+\errr.
    \end{aligned}
  \right.
\end{equation*}

The following step is to better diagonalize the system, and for this we only
need to modify the $R^{(n-1)}$ equation by using the known
substitution $\R:=(1+\W)R^{(n-1)}$ (see \cite{HIT}). We obtain
\begin{equation*}
  \left\{
    \begin{aligned}
      &\W^{(n-1)}_t +b\W_{\alpha}^{(n-1)}+\frac{\R_{\alpha} }{1+\bar{\W}}+\frac{R_{\alpha}}{1+\W}\W^{(n-1)}=G\\
      &\R_t+b\R_{\alpha}-\frac{(g+\mfa)\Til [\W^{(n-1)}]}{1+\W}=K_1,
    \end{aligned}
  \right.
\end{equation*}
where
\begin{equation*}
  \begin{aligned}
    K_1=-(n+1)
    \frac{R_{\alpha}\R}{1+\bar{\W}}-n\frac{\bar{R}_{\alpha}\R}{1+\W}+\errr.
  \end{aligned}
\end{equation*}
To deal with the mildly unbounded terms on the right we proceed in two steps using the same idea as  in \cite{HIT}. First we implement a new holomorphic substitution
\[
\tilde{\R}:=\R -R_{\alpha}\W^{(n-2)}+(2n-1)\W_{\alpha}R^{(n-2)}.
\]
With the exception of a couple of terms (see also \cite{HIT}), the contribution of the added quadratic correction is cubic and lower order, so we obtain
\begin{equation*}
  \left\{
    \begin{aligned}
      &\W^{(n-1)}_t \! +b\W_{\alpha}^{(n-1)}\!+\! \frac{\tilde{\R}_{\alpha}
      }{1+\bar{\W}}+\! \frac{R_{\alpha}}{1+\W}\W^{(n-1)}=-n\left(
        \frac{\bar{R}_{\alpha}}{1+\W}+\! \frac{R_{\alpha}}{1+\bar{\W}}\right)
      \W^{(n-1)}\! +\errw
      \\
      &\tilde{\R}_t+b\tilde{\R}_{\alpha}-\frac{(g+\mfa)\Til
        [\W^{(n-1)}]}{1+\W}= -n\left(
        \frac{R_{\alpha}}{1+\bar{\W}}+\frac{\bar{R}_{\alpha}}{1+\W}
      \right)\tilde{\R}+\errr .
    \end{aligned}
  \right.
\end{equation*}

At this point we are in a similar situation as we were in the case
$n=2$. Precisely, we still have unbounded terms on the right, and the
goal is to eliminate them. The second step is to use  the same procedure as in the case $n=2$, which is to multiply the equations by $e^{n\phi}$, where
$\phi=-2\Re \log (1+\W)$. After standard estimates, and using also Lemma~\ref{eroare}, we can write a system for
$(w:=e^{n\phi}\W^{(n-1)}, r:=e^{n\phi}\tilde{\R})$:
\begin{equation*}
  \left\{
    \begin{aligned}
      &w_t+ bw_{\alpha}+\frac{r_{\alpha}}{1+\bar{\W}}+\frac{R_{\alpha}}{1+\W}w=\errw\\
      & r_t+br_{\alpha}-\frac{(g+\mfa)\Til [w]}{1+\W}=\errr .
    \end{aligned}
  \right.
\end{equation*}
As $(w,r)$ are no longer holomorphic, we will need to project them via
the projection $\P$. We denote the projected variables $(\P w, \P r)$ by $(\fw, \fr )$, and write the equations for them.  As we have seen in the case $n=2$, we obtain some
additional terms which we can express as commutators. Moreover, these additional terms can be easily bounded using the
commutators estimates obtained in the Appendix to obtain the system
\begin{equation}
  \label{pwpr}
  \left\{
    \begin{aligned}
      &\fw_t+ \P[b \fw]_{\alpha}+\P \left[ \frac{\fr_{\alpha}}{1+\bar{\W}}\right] +\P \left[ \frac{R_{\alpha}}{1+\W}\fw\right] =\P [\errw ]\\
      &\fr_t+P\left[ b \fr_{\alpha}\right] -\P\left[ \frac{(g+\mfa)\Til
          [\fw]}{1+\W}\right] =\P[\errr] .
    \end{aligned}
  \right.
\end{equation}
Now, we modify this system one last time in order to be able to compare it with the model system \eqref{ModelLinSys}, and after one rather 
straightforward estimates we can rewrite it as
\begin{equation}
  \label{pwpr-frank}
  \left\{
    \begin{aligned}
      &\fw_t+ \P[b \fw]_{\alpha}+\P \left[ \frac{\fr_{\alpha}}{1+\bar{\W}}\right] -\P \left[ \frac{R_{\alpha}}{1+\W}\Til^2 [\fw]\right] = \P [\errw ]\\
      &\fr_t+P\left[ b \fr_{\alpha}\right] -\P\left[ \frac{(g+\mfa)\Til
          [\fw]}{1+\W}\right] =\P[\errr] .
    \end{aligned}
  \right.
\end{equation}

In order to be able to apply the estimates obtained for
the model system in \eqref{ModelLinSys}, we need to ensure that the
energy of $(\fw, \fr)$ is equivalent to the one of $(\W ^{(n-1)},
R^{(n-1)})$. This is summarized in the following proposition:
\begin{proposition}
  The energy of of $(\fw, \fr)$ above is equivalent to the one of
  $(\W ^{(n-1)}, R^{(n-1)})$,
  \begin{equation}
    \Vert (\fw, \fr)\Vert_{\mathcal{H}}\approx_A  \Vert ( w,  r)\Vert_{\mathcal{H}}\approx _{A} \Vert ( \W^{(n-1)},  R^{(n-1)})\Vert_{\mathcal{H}}\approx _{A}\mbox{ modulo } A\norm_n.
  \end{equation}
\end{proposition}
 A similar result can be found in \cite{HIT} (see Lemma 3.5).
The proof of the above proposition is quite similar, and we leave it as an exercise for the reader.

Now the result of Proposition~\eqref{t:en=large} follows from the energy estimates for the model system \eqref{ModelLinSys}, namely \eqref{est1} applied to \eqref{pwpr-frank}; to obtain the result we use the energy functional
\[
E^{n, (2)}_{high}(\W, R)=E^{(2)}_{lin}(\fw, \fr).
\]

 The further goal is to obtain cubic lifespan bounds, which would correspond to proving Proposition~ \eqref{t:en=small}. The key to that 
is to produce a suitable modified cubic energy. This is done in the 
next subsection. However, here we will discuss the leading
part of the modified cubic energy, which is given by
\[
E^{n,(3)}_{high}(\W,R)=E^{(3)}_{high}(\fw,\fr):=E^{(2)}_{lin} (w,r)-E^{(2)}_{\omega , lin}(w,r), \quad \mbox{ where }\quad  \omega = \frac{1}{4}(1+\Til ^2)\Re \W .
\]

\begin{remark} Comparing $E^{n,(3)}_{high}(w,r)$ with the corresponding version appearing in the infinite depth case, one
  will notice that the are some differences. On one hand, the second component of the above energy is specific to the finite bottom case, and does not appear at all in the infinite bottom problem. On the other hand, the last  three terms in the quasilinear cubic energy from \cite{HIT} are no longer showing up in the above leading energy. Mainly, this is because here we use better bookkeeping of the errors, and those terms are now reclassified as admissible error terms.  In other words, here they  are incorporated into the lower order component
of the quasilinear modified energy we seek to construct.
\end{remark}

We claim that we have favourable bounds for the time evolution of this energy. Precisely, we have
\begin{proposition}
  \label{high}
  Let $(\fw,\fr)$ be defined as above.  Then
  \begin{itemize}
  \item[a)] Assuming that $A\ll 1$, we have
    \begin{equation}
      \label{high-equi}
      E^{(3)}_{high}(\fw,\fr)=  E_{0}(\W^{(n-1)}, R^{(n-1)}) +O(A)\norm_n^2.
    \end{equation}
  \item[b)] The solution of $(\fw, \fr)$ of \eqref{pwpr} satisfies
    the following energy estimate
    \begin{equation}
      \frac{d}{dt}  E^{(3)}_{high}(\fw, \fr)\lesssim_A B\norm_{n}^2, \quad \Poly{\geq 4}\frac{d}{dt}  E^{(3)}_{high}(\fw, \fr)\lesssim_A AB\norm_{n}^2.
    \end{equation}
  \end{itemize}
\end{proposition}
 The proof is a straightforward application of Proposition~\eqref{p:local}. To see that, one needs to verify that the real weight $\omega= \dfrac{1}{4}(1+\Til ^2)\Re \W$ satisfies the required bounds \eqref{en-omega}. But this is true in view of Lemma~\ref{l:a}, as $\omega$ is a multiple of the $a_1$ component of $\mfa$.

\subsection{The quasilinear modified energy for $n\geq 2$, 
small  data.} In this section we construct an $n$-th order energy with
cubic estimates, $E^{n, (3)}$, which satisfies the bounds in Proposition~\ref{t:en=small}.
This energy is obtained following the method introduced in \cite{HIT}(\emph{the quailinear modified energy method}), which we now describe by
splitting it into several steps:

\bigskip

{\bf{ 1. Construct the normal form energy}}. 
This has been accomplished in the previous Section~\ref{s:nf-en},
but for convenience we outline the process here. Formally, it begins 
with the construction of a normal form transformation whose aim is to eliminate
the quadratic terms in the equation \eqref{FullSystem-re} for $(W,Q)$.  The normal form variables $(\tW, \tQ)$ are given by
\[
\begin{cases}
  \tW =W +W^{[2]}= W + B^h[W,W] + \frac1gC^h[Q,Q] + B^a[W,\bar W] + \frac 1gC^a[Q,\bar Q] \vspace{0.1cm}\\
  \tQ =Q+Q^{[2]} = Q + A^h[W,Q] + A^a[W,\bar Q] + D^a[Q,\bar W],
\end{cases}
\]
where the bilinear multipliers arising here are defined in
Section~\ref{s:nf}.  A full description of these
symbols is given later in the same section. What matters is that
the normal form variables $(\tW, \tQ)$ solve an equation of the form
\begin{equation*}
  \left\{
    \begin{aligned}
      &\Poly{\leq2}(\tW_t +\tQ_{\alpha})= 0\\
      &\Poly{\leq2}(\tQ _t - g \Til \tW)= 0.
    \end{aligned}
  \right.
\end{equation*}
 Following Section~\ref{s:nf-en}, its associated cubic normal form
energy functional is
\[
\begin{split}
  \tEnnf(W,Q) = & \ \Poly{\leq3}E_0 (\partial^n \tilde
  W, \partial^n \tilde Q) \\ = & \ E_0(\partial^n W, \partial^n Q) +
  2g \la \partial^n W, \partial^n W^{[2]} \ra - 2 \la
  \Til^{-1} \partial^{n+1} Q, \partial^n Q^{[2]} \ra .
\end{split}
\]
This is chosen so that the following relation holds
\begin{equation}\label{no-cubic}
\Poly{\leq 3} \frac{d}{dt} \tEnnf(W,Q) = 0.
\end{equation}
Here we discard the quartic terms in $E_0 (\partial^n \tilde W, \partial^n \tilde Q)$
as on one hand they are both highly unbounded, and on the other hand they
 do not affect the last relation above. Further, unlike the cubic terms, the quartic terms carry no intrinsic meaning as the normal form transformation is only uniquely determined up to cubic terms.

 As we show in the proof Proposition~\ref{p:ennnf}, the normal form energy  $\tEnnf(W,Q)$  can be expressed up to quartic terms as a function of diagonal variables
 $(\W,R)$ in the form
\[
E^n_{NF}(\W, R)=E_0 (\partial^{n-1} \W, \partial^{n-1} R)+gB(\W, \W, \W)+A(\W, R,R).
\]
with trilinear forms $A$ and $B$ whose symbols we have computed. 

 We further remark that while the normal form expression has
 singularities at frequency zero, no such singularities are present in
 the normal form energy. This is due to symmetrization cancellations
 akin to some form of null condition. Even better,
 neither $W$ nor $Q$ can appear undifferentiated in the above 
cubic terms.

 The chief  disadvantage of the normal form energy, which 
due to the fact that our problem is quasilinear, is
that the quartic and higher terms in its time derivative $d/dtE^{n}_{NF}$
 are highly unbounded. Thus there is no hope to prove the bound 
\eqref{dt-cubic} for it, neither does \eqref{en-cubic} hold, for that matter. 

To better isolate the above difficulty, we have decomposed the normal form energy
into two parts,
\[
E^n_{NF}=E^{n}_{NF, high}+E^{n}_{NF, low},
\]
where
\[
\begin{aligned}
&E^n_{NF,high}(\W, R)=  \ E_0 (\partial^{n-1} \W, \partial^{n-1} R)+g B_{high}(\W, \W, \W)
+\tilde A_{high}(\W, R,R),
\\
&E^n_{NF,low}(W, Q)=  \ g B_{low}(\W, \W, \W)+ A_{low}(\W, R,R).
\end{aligned}
\]
Here the lower order part is quite complicated algebraically, but has the virtue 
that it does not cause difficulties  neither in  \eqref{en-cubic} nor in \eqref{dt-cubic}.
The high frequency part, on the other hand, has the advantage that we can compute
 it explicitly. Precisely, by Proposition~\ref{p:ennnf} we have
  \begin{equation}
    \label{m1+}
    \begin{aligned}
  &B_{high}(\W,\W,\W)     := \langle \partial^{n-1} \W,\partial^{n-1} \W\rangle_{- 4n\Re \W
        + \frac12 (1+\Til ^2)\Re \W},\smallskip
\\
&A_{high}(\W,R,R) :=  - \langle \partial^{n-1} R, \Til^{-1} \partial^{n-1} R_\alpha \rangle_{- 4n\Re \W
        - \frac12 (1+\Til^2)\Re \W}\\
        & - 2 \langle \W  \partial^{(n-1)}R, \Til^{-1} \partial^{(n-1)}R_\alpha \rangle + 2 \la \partial^{(n-2)}\W R_\alpha,\Til^{-1}\partial^{(n-1)}R_\alpha\ra.
    \end{aligned}
  \end{equation}

This is the part we need to further modify and adapt to the quasilinear 
structure of our problem.

\bigskip

\bigskip

{\bf{2. Construct the quasilinear modified energy.} }Here we construct
\emph{ the quasilinear modified energy $E^{n, (3)}$}, starting from the normal
form energy $E^n_{NF}(\W,R)$. Inspired by the expression for the high frequency 
part $ E^n_{NF,high}(\W,R)$ of the normal form energy, one is naturally led to consider
the high frequency quasilinear  modified energy $E^{(3)}_{high} (\fw,\fr)$ where
\begin{equation}
\begin{aligned}
E^{(3)}_{high}(w,r):=&E_{ lin}^{(2)}(w,r) -\frac{1}{4}E^{(2)}_{\omega, lin} (w,r). \\
\end{aligned}
\end{equation}
Comparing the two, we would like them to agree to cubic order. This is not 
exactly the case, however the next best thing happens, namely that 
the cubic part of the difference is lower order:

\begin{lemma} \label{l:eq-high}
The trilinear form $\Poly{\leq 3}(  E^n_{NF}(\W,R) -  E^{(3)}_{high} (\fw,\fr))$ 
is a lower order form in $(\W,R)$, where $n\geq 1$.
\end{lemma}
The lemma is proved later in this section.

Based on this, we define 
\[
E^{n, (3)} = E^{n, (3)}_{high} + E^{n, (3)}_{low},
\]
where
\[
E^{n, (3)}_{low} =  E^n_{NF,low} +  \Poly{\leq 3}(  E^n_{NF} -  E^{n, (3)}_{high} (w,r)).
\]
This guarantees that we have the relation
\begin{equation}\label{all-cubic}
\Poly{\leq 3} E^{n, (3)} = \Poly{\leq 3} E^n_{NF} .
\end{equation}

\bigskip

{\bf{3. $E^{n,3}$ is a good quasilinear cubic energy.}} In other words
we want to prove that the estimate in Proposition~\eqref{high} holds.
In view of \eqref{no-cubic} and \eqref{all-cubic} it follows that for
solutions to \eqref{FullSystem-re} we have
\begin{equation}
\Poly{\leq 3}\frac{d}{dt} E^{n, (3)} = 0 .
\end{equation}
Thus, we obtain
\[
\frac{d}{dt}E^{n, (3)}=\Poly{\geq 4}  \frac{d}{dt}E^{n, (3)}_{low}
 +\Poly{\geq 4}  \frac{d}{dt} E^{n, (3)}_{high}.
\]
This relation allows us to split the task of proving bounds
for $E^{n, (3)}$ into separate bounds for the high, respectively the 
low frequency part. Precisely, it remains to establish the 
following:

\begin{lemma}\label{l:high}
The high frequency part $E^{n, (3)}_{high}$ satisfies the bounds
\begin{equation}
  \label{en-high}
   E^{n, (3)}_{low} =  \| (\partial^{n-1} \W, \partial^{n-1} R)\|_{\H}^2+ O(A) \norm_n^2,
\end{equation}
respectively
\begin{equation}
  \label{lambda4-high}
  \left| \Poly{\geq 4} \left( \frac{d}{dt} E^{n, (3)}_{high}\right)\right| \lesssim_A AB\norm_n ^2.
\end{equation}
\end{lemma}

\begin{lemma}\label{l:low}
The (cubic)  low frequency part $E^{n, (3)}_{low}$ satisfies the bounds
\begin{equation}
  \label{en-low}
   E^{n, (3)}_{low} = O(A) \norm_n^2,
\end{equation}
respectively
\begin{equation}
  \label{lambda4-low}
  \left| \Poly{\geq 4} \left( \frac{d}{dt} E^{n, (3)}_{low}\right)\right| \lesssim_A AB\norm_n ^2.
\end{equation}
\end{lemma}

To conclude the proof of Proposition~\ref{t:en=small} it remains to prove the three 
lemmas above. This is the same argument as in \cite{HIT}, but here it is slightly
more complicated, at least at the computational level.

\begin{proof}[Proof of Lemma~\ref{l:eq-high}]
We first expand the expression $\Poly{\leq 3}(  E^n_{NF, high}(\W, R) -  \dfrac{1}{2}E^{(3)}_{high} (\fw,\fr))$ for the case $n\geq 3$
and express the result in terms of $(\W,R)$. Up to cubic terms the expansion of $(\fw,\fr)$ is
\begin{equation*}
\left\{
\begin{aligned}
&\Poly{\leq3}\,\fw= \W^{(n-1)}-2n \P[\Re \W \cdot \W^{(n-1)}] \\
&\Poly{\leq 3}\, \fr=R^{(n-1)}-2n\P [\Re \W\cdot R^{(n-1)}]+\W R^{(n-1)}-R_{\alpha}\W^{(n-2)}+(2n-1)\W_{\alpha}R^{(n-2)}. 
\end{aligned}
\right.
\end{equation*}
Before substituting the expansion of $(\fw, \fr)$ into the energy formulas, we observe that the projection $\P$ can be dropped off; moreover the last term in the quadratic expansion of $\fr$ only contributes to lower order terms based on the definition provided in the earlier section. Thus, we can also omit this term. The explicit quadratic and cubic terms showing up in the expression of $E^{(3)}_{high} (\fw,\fr))$ are
\begin{equation}
\label{miha}
\begin{aligned}
\Poly{\leq 3} E^{(3)}_{high} (\fw,\fr))&=\Poly{\leq 3} \left( E^{(2)}_{lin}(\fw, \fr)-\frac{1}{2}E^{(2)}_{\omega, lin}(\fw,\fr)\right) \\
&=\left\langle \W^{(n-1)}, \W^{(n-1)} \right\rangle _{g}+\left\langle LR^{(n-1)}, LR^{(n-1)} \right\rangle \\
&\quad  + \left\langle \W^{(n-1)}, \Re \W \cdot \W^{(n-1)} \right\rangle _{-4ng}+\left\langle \W^{(n-1)}, \W^{(n-1)}\right\rangle _{g\omega}\\
&\quad -\left\langle \Til^{-1}R^{(n)}, -4n\Re \W\cdot R^{(n-1)}+2R^{(n-1)}\W-2R_{\alpha}\W^{(n-2)}\right\rangle \\
&\quad -\frac{1}{2}\left\langle \W^{(n-1)}, \W^{(n-1)}\right\rangle _{g\omega}-\frac{1}{2}\left\langle LR^{(n-1)}, LR^{(n-1)}\right\rangle_{\omega}, 
\end{aligned}
\end{equation}
where $\omega =(1+\Til^2)\Re\W$. 

It remains to compare the result with the expression of $ \Poly{\leq 3}E^n_{high, NF} (\W, R)$, which we recall below: 
\[
\Poly{\leq 3}E^n_{NF, high}(\W, R)=  \ E_0(\partial^{n-1} \W, \partial^{n-1} R) + g B_{high}(\W,\W,\W) + A_{high}(\W,R,R), 
\]
where $B_{high}(\W,\W,\W)$,  $A_{high}(\W,R,R)$ are given in \eqref{m1}.

First we observe that the first line of the expansion in \eqref{miha} is in fact  $\ E_0(\partial^{n-1} \W, \partial^{n-1} R)$. The terms on the second line in \eqref{miha} together with the first term on la last line are the terms appearing in $gB_{high}$ modulo a commutator, which yields a lower order term; the commutator is
\[
\left[ \Til, \Re\W \right]\Re \W^{(n-1)}. 
\]
We return to the remaining terms in \eqref{miha} and observe  that the first term in the expansion of the inner product on the third line together with the last term on la last line match (after integrating by parts) the first term in the expansion of $A_{high}$, \eqref{m1}, up to the commutators 
\[
\left[ L, \omega \right]\Im (LR^{(n-1)}), \quad  \left[ L, \omega \right]\Til \Re (LR^{(n-1)}),
\]
which are again lower order terms.

Lastly, the last two terms, $ -\left\langle \Til^{-1}R^{(n)}, 2R^{(n-1)}\W\right\rangle $ and $ \left\langle \Til^{-1}R^{(n)}, 2R_{\alpha}\W^{(n-2)}\right\rangle $, are a perfect match to the remaining terms in $A_{high}$.
 
 For the case $n=2$ the computation is similar but simpler.
The last three terms in $\Poly{\leq 3}$ no longer appear, whereas in 
the expression for $A_{high}(\W,R,R)$ in \eqref{m1} the last two terms 
also cancel.

\end{proof}

\begin{proof}[Proof of Lemma~\ref{l:high}]
 This is a direct consequence of Lemma~\eqref{high}.

\end{proof}

\begin{proof}[Proof of Lemma~\ref{l:low}]
We recall that  $E^{n,3}_{low}$ is a trilinear expression
of the form
\[
 E^{n,3}_{low} = g B_{low}(\W,\W,\W) + A_{low}(\W,R,R),
\]
where $ B_{low}$ and $A_{low}$ are translation invariant
trilinear forms. To begin with, we note that the exact form of the
terms in $E^{n,3}_{low}$ is irrelevant here.  All that matters is
their symbol class, which we now recall. In the case of $B_{low}$, the
symmetric symbol $B_{low}(\xi,\eta,\zeta)$ satisfies
\[
B_{low} \in ES(d \rho^{2n-3}), 
\]
while in the case of $A$, the symbol $A_{low} (\xi,\eta,\zeta)$ is only
symmetric in the last two variables and satisfies
\[
A_{low} \in ES(\rho^{2n-2}) + ES( d d_1 \rho^{2n-3}).
\]
Here $d$, $d_1$ measure the distance to the axes as follows:
\[
d = 1+\min\{|\xi|,|\eta|,|\zeta|\}, \qquad d_1 = 1+\min\{|\eta|,|\zeta|\}.
\]
We recall that we can eliminate the exponentials in the symbols at the expense 
of  replacing some of the arguments $(\W,R)$ by their complex conjugates.

We begin with the estimate \eqref{en-low}. By applying a standard
trilinear Littlewood-Paley decomposition combined with a 
standard separation  of variables argument we can thus write
$B_{low}$ as a sum of a rapidly convergent series
\[
B_{low}(\W,\W,\W) =  \sum_{1 \leq j \leq k}  2^j 2^{(2n-3)k} \sum_{m}
\int \chi^{m,1}_{j,k}(D) \W_j \cdot \chi^{m,2}_{j,k}(D)  \W_k \cdot \chi^{m,3}_{j,k}(D) \W_k \ d\alpha .
\]
Here complex conjugates are also allowed, and the symbols $\chi^{m,i}_{j,k}(\xi)$
have the following properties:

\medskip
(i) They are  smooth on the respective dyadic scales $2^j$, respectively $2^k$
uniformly with respect to $j,k$.
\medskip

(ii) They are  rapidly decaying in $m$, also uniformly with respect to $j,k$.

\medskip
In particular the multipliers $\chi^{m,i}_{j,k}(D)$ are uniformly
bounded in all $L^p$ spaces and rapidly decaying with respect to $m$.
Hence, we immediately obtain the following bound for $\tilde A_{low}$:
\[
\begin{split}
| B_{low}(\W,\W,\W)| \lesssim & \ \sum_{1 \leq j \leq k} 2^j 2^{(2n-3)k} 
\| \W_j\|_{L^\infty} \|\W_k\|_{L^2}^2 
\\ \lesssim & \ \sup_j \| \W_j\|_{L^\infty} \sum_k2^{(2n-2)k} \|\W_k\|_{L^2}^2 
\\ \leq & \ g^{-1} A \norm_n^2.
\end{split}
\]
The computation is only slightly more involved for  $A_{low}$. We only discuss 
the $ ES( d d_1 \rho^{2n-3})$ part, as the analysis for the lower homogeneity part
$ES(\rho^{2n-2})$ is similar but simpler. We need to 
consider two cases depending on whether the $\W$ factor or an $R$ factor is low frequency.
We obtain
\[
\begin{split}
| A_{low}(\W,\! R,\! R)| \lesssim &  \sum_{1 \leq j \leq k} 2^j 2^{(2n-2)k} 
\| \W_j\|_{L^\infty} \|R_k\|_{L^2}^2 + 2^{2j} 2^{(2n-3)k} \|R_j\|_{L^\infty} \|\W_k\|_{L^2}
\|R_k\|_{L^2} 
\\ \lesssim &  \sup_j \| \W_j\|_{L^\infty}\! \sum_k2^{(2n-1)k} \|R_k\|_{L^2}^2 
\!+\! \sup_j 2^{\frac{j}2} \| R_j\|_{L^\infty} \sum_k2^{(2n-\frac32)k} \|\W_k\|_{L^2} \|R_k\|_{L^2}
\\ \leq & \ A \norm_n^2.
\end{split}
\]

Now consider the bound \eqref{lambda4-low}, where we write
\[
\Poly{\geq 4}\frac{d}{dt}  B_{low}(\W,\W,\W) = 
3  B_{low}(\Poly{\geq 2} \partial_t \W ,\W,\W), 
\]
respectively 
\[
\Poly{\geq 4}\frac{d}{dt}  A_{low}(\W,R,R) = 
  A_{low}(\Poly{\geq 2} \partial_t \W ,R,R) + 
2   A_{low}( \partial_t \W ,\Poly{\geq 2}R,R).
\]
For the time derivatives of $\W$ and $R$ we separate the leading order 
transport term, precisely its paraproduct part, writing
\[
\Poly{\geq 2} \partial_t \W = (\partial_t + T_b \partial_\alpha)\W 
-  T_b \partial_\alpha \W ,
\]
and similarly for $R$. Here by a slight abuse of notation we include the contribution of 
the low frequencies in $b$ in $T_b$. This is because we do not have good control over
the low frequencies of $b$, so these cannot be bounded perturbatively, and instead
must be treated only in a commutator type fashion.

 The first term has better regularity, and its
contribution is treated perturbatively.  Precisely, a computation
similar to the one above applies provided we can establish the
pointwise bounds
\begin{equation}\label{dt-para-inf}
\| \Poly{\geq 2}  (\partial_t + T_b \partial_\alpha) \W\|_{B^{0,\infty}_\infty} 
+g^{-\frac12} \| \Poly{\geq 2}  (\partial_t + T_b \partial_\alpha) R\|_{B^{\frac12,\infty}_\infty} 
\lesssim_A AB ,
\end{equation}
respectively the $L^2$ bounds 
\begin{equation}\label{dt-para-2}
\| \Poly{\geq 2}  (\partial_t + T_b \partial_\alpha) \W\|_{H^{n-\frac32}} 
+g^{-\frac12} \| \Poly{\geq 2}  (\partial_t + T_b \partial_\alpha) R\|_{H^{n-1}} 
\lesssim_A A \norm_n .
\end{equation}
Both of  these are proved in Lemma~\ref{l:dt-rw} in the Appendix. 

For the contribution of the transport term, on the other hand, we need to capture some 
cancellation. We discuss the case of the form $B_{low}$, as $A_{low}$ is similar.
In the product case, this cancellation is a simple integration by parts,
based on the formula
\[
\int b \partial_\alpha \W_1 \W_2 \W_3 + \W_1 b \partial_\alpha \W_2 \W_3 
+ \W_1 \W_2 b \partial_\alpha \W_3  \ d\alpha = - \int b_\alpha   \W_1 \W_2 \W_3 \ d\alpha ,
\]
where the derivative is moved onto $b$. In our case, however, 
we need to contend instead with factors which at frequency $2^j$
have the form $\chi_j(D)(b_{<j}\partial_\alpha \W_j)$. 

As a preliminary observation, 
we remark that we can commute out the coefficient $b_{<j}$, by writing
\[
\chi_j(D)(b_{<j} \partial_\alpha \W_j) = b_{<j} \chi_j(D)\partial_\alpha  \W_j +
[\chi_j(D), b_{<j}] \partial_\alpha \W_j.
 \]
Here the commutator term can be expressed in the form
\[
[\chi_j(D), b_{<j}] \partial_\alpha \W_j = L(\nabla b_{<j}, \W_j), 
\]
where $L$ stands for a translation invariant bilinear form with integrable kernel.
Then one can directly use the bounds in Lemma~\ref{l:b} for $b$ to show that 
this term satisfy the same bounds as in \eqref{dt-para-inf}, \eqref{dt-para-2}, and thus 
can be treated perturbatively.

Once we have discarded the commutator term, we can include $\chi_j(D)$ into $\W_j$
for brevity, and then  we are left with having to estimate an expression of the form
\[
I = \int b_{<j}    \partial_\alpha \W_j   \cdot   \W_k \cdot  \W_k 
+ \W_j  \cdot b_{<k}  \partial_\alpha  \W_k \cdot  \W_k + 
\W_j  \cdot   \W_k b_{<k}\partial_\alpha  \W_k \, d\alpha . 
\]
Separating the expression $b_{<k}$ in all factors we can integrate by parts
and obtain
\[
I = - \int  \partial_\alpha  b_{<k}  \W_j   \cdot   \W_k \cdot  \W_k\, d\alpha
-  \int b_{[j,k]}    \partial_\alpha \W_j   \cdot   \W_k \cdot  \W_k \, d\alpha .
\]
Now in the first integral we group the product $ \partial_\alpha  b_{<k}  \W_k$,
which again satisfies the same bounds as in \eqref{dt-para-inf}, \eqref{dt-para-2}.
In  the second integral the derivative yields a $2^j$ factor, and now the 
expression $2^j  b_{[j,k]} $ is even better than $\partial_\alpha  b_{<k} $.
\end{proof}

\subsection{The quasilinear modified energy for $n = 1$, 
small  data.} In this section we construct a first order energy with
cubic estimates, $E^{1, (3)}$, which satisfies the bounds in Proposition~\ref{t:en=small}.
This energy is obtained following the same procedure as in the case $n \geq 2$ presented before,
but with some minor computational differences, which we now describe.

One main source of differences is  the expression for $A_{high}$ which is slightly different here. 
Also in this case it is no longer meaningful to do the exponential conjugation. Because 
of this, it is now convenient to set up the quasilinear correction to the normal form energy 
in a more direct fashion,
\[
E^{1,(3)}_{high} = E^{(3)}_{high}(\W,R) + E^{(3),a}(\W,R),
\]
where the extra component 
\[
E^{(3),a}(\W,R) = - 2 \langle \W, \W^2\rangle + 2 \langle R, \W \Til^{-1} R_\alpha \rangle  
\]
mirrors the similar correction in the infinite bottom case \cite{HIT}.

An advantage of doing this is that the remaining lower order cubic part 
\[
E^{1,(3)}_{low} = B_{low}(\W,\W,\W) + A_{low}(R,R,\W) 
\]
contains only terms whose symbol is not only lower order on the diagonals
but also away from them, namely their symbols satisfy 
\begin{equation}\label{low-1}
B_{low} \in S(\rho^{-1}), \qquad A_{low} \in S(1) .
\end{equation}
This is due to the similar gain in Proposition~\ref{p:ennnf}. 

With these definitions we  remark that  Lemma~\ref{l:eq-high} is still valid. For that we need to match the terms in  $\Poly{\leq 3}(  E^n_{NF, high}(\W, R) )$ to the terms in  $\Poly{\leq 3}( \frac{1}{2}E^{1,(3)}_{high} (\W,R))$. The computations are similar to the ones we did for the case $n\geq 2$ but simpler.

Further, the  statements of Lemmas~\ref{l:high}, \ref{l:low} remain unchanged. It remains to prove Lemmas~\ref{l:high},\ref{l:low} in this context.

\begin{proof}[Proof of Lemma~\ref{l:high}, $n=1$]
The bound \eqref{en-high} is straightforward. The $E^{(3)}_{high}$ part of \eqref{lambda4-high} is also exactly as before in view of Lemma~\ref{quadratic, n=1}. It remains to prove  the extra correction  $E^{(3),a}(\W,R) $
also satisfies \eqref{lambda4-high}.
For convenience we state this in a separate lemma:

\begin{lemma}\label{l:high-a}
The cubic correction $E^{ (3),a}$ satisfies the bounds
\begin{equation}
  \label{en-high-a}
   E^{(3),a} \lesssim_A A \norm_1^2,
\end{equation}
respectively
\begin{equation}
  \label{lambda4-high-a}
  \left| \Poly{\geq 4} \left( \frac{d}{dt} E^{(3),a}\right)\right| \lesssim_A AB\norm_1 ^2.
\end{equation}
\end{lemma}

\begin{proof}
The first bound is straightforward, but the second does require some computations.  We consider both correction terms
\[
I_1=\left\langle \W, \W^2\right\rangle , \quad
 I_2=\left\langle R, \W \Til^{-1}R_{\alpha}\right\rangle ,
\]
and discuss each of them separately.

To estimate their derivatives it is easiest to use the unprojected form 
\eqref{DiagonalQuasiSystem-re} of the equations for $\W$ and $R$, which for our purposes here
we write in the form
\begin{equation}
\label{quadratic+}
\left\{
\begin{aligned}
&(\partial_t +b\partial_{\alpha})\W=-b_{\alpha}(1+\W)+\bar{R}_{\alpha}:=G\\
&(\partial_t +b\partial_{\alpha})R=i\frac{g\W -\mathfrak{a}}{1+\W}:=K.
\end{aligned}
\right.
\end{equation}
For $G$ and $K$ we only need their quadratic parts and higher,
\[
G^{2+}=-b_{\alpha}\W +\P [R\bar{Y}]_{\alpha}, \quad K^{2+}=-\frac{(ig\W-a_1)\W +a}{1+\W}.
\]
Then we have
\[
\begin{aligned}
\Poly{\geq 4} \left(  \frac{d}{dt}I_1\right)&=-\left\langle  b\W_{\alpha},\W^2\right\rangle 
-\left\langle \W_{\alpha},2b\W \W_{\alpha} \right\rangle +\left\langle  G^{2+}, \W ^2\right\rangle+2\left\langle \W, \W G^{2+}\right\rangle .
\end{aligned}
\]

Distributing derivatives and using Corollary~\ref{c:adj}, we separate the terms with undifferentiated $b$ as
\[
\begin{aligned}
-\left\langle  b\W_{\alpha},\W^2\right\rangle -\left\langle \W_{\alpha},b\partial_{\alpha} (\W ^2) \right\rangle =&-\left\langle  b\W_{\alpha},\W^2\right\rangle +
+\left\langle  \Til ^{-1}\partial_{\alpha} [b\Til \W],\W^2\right\rangle \\
&=\left\langle  (-b\partial_{\alpha}+ \Til ^{-1}\partial_{\alpha} b\Til )\W,\W^2\right\rangle . \\
\end{aligned}
\]
Note that we can express this as the sum of two terms, as shown below
\[
\left\langle   \Til ^{-1}\partial_{\alpha} [b\, ,\, \Til ]\W,\W^2\right\rangle +\left\langle b_{\alpha}\W, \W ^2\right\rangle,
\]
where both can be easily controlled by $AB\norm_1$ using Lemma~\ref{est:TilComm} followed by Lemma~\ref{est:bBMO}. The contribution of $G^{2+}$ is harmless since all the terms in $G^{2+}$ are bounded in $L^2$,
\[
\Vert G^{2+}\Vert_{L^2}\lesssim_A B\norm_1.
\]

We now return to the last correction term, $I_2$:
\[
\begin{aligned}
\Poly{\geq 4} \left(  \frac{d}{dt} I_2\right)  &=
\langle R_t, \W \Til^{-1} R_\alpha \rangle + \langle R_t, \W \Til^{-1} R_\alpha \rangle + 
\langle R_t, \W \Til^{-1} R_\alpha \rangle. 
\end{aligned}
\]
The argument for this expression  is slightly more involved. 
We proceed as in the proof of Lemma~\ref{l:low}, but with some extra
care. We begin with a Littlewood-Paley decomposition
\[
\langle R, \W \Til R_\alpha \rangle = \sum_{k,k_1,k_2 \geq 0} 
\langle P_{k_1} R, P_{k_2} \W \Til^{-1} P_{k_3} R_\alpha \rangle,
\]
and similarly for the time derivative. For the above summand to be nonzero, 
we need the two highest frequencies to be comparable. We first distinguish two easier cases:

\begin{enumerate}
\item[(i)] If $\min\{k,k_1,k_2\}  \lesssim 1$, and the time derivative applies to the low frequency.
Then the time differentiated factor is bounded in $L^2 \cap L^\infty$, and the two remaining factors 
are estimated in $L^2$ or $L^\infty$ as needed.

\item[(ii)] If $k < k_1 = k_2$, then we take advantage of the fact that our factors are holomorphic,
and thus have exponential decay at positive frequencies.  Thus we obtain an $e^{-Nk_1}$ gain
which is more than enough for all our estimates.
\end{enumerate}

This leaves us with two principal cases, namely the sums:
\[
J_1 = \partial_t \sum_{k > 4} \int \bar R_{k} \W_k R_{\leq k,\alpha} \, d\alpha, \qquad  
J_2 = \partial_t \sum_{k > 4} \int \bar R_{k} \W_{\leq k} R_{k,\alpha} \, d \alpha .
\]
To estimate their time derivatives we use again the decomposition
\[
\partial_t \W = (\partial_t + T_b \partial_\alpha) \W - T_b \partial_\alpha \W,
\qquad
\partial_t R = (\partial_t + T_b \partial_\alpha) R - T_b \partial_\alpha R.
\]
For the first term in each decomposition we have the estimates in Lemma~\ref{l:dt-rw}.
Using them, the bounds for the corresponding contributions to $J_1$ and $J_2$ are somewhat tedious 
but routine. It remains to consider the $T_b$ contributions, which are
\[ 
\begin{split}
J_1^b =  \sum_{k > 4} \int    b_{<k} \bar R_{k,\alpha} \W_k R_{\leq k,\alpha} 
+  \bar R_{k} b_{<k}  \W_{k,\alpha} R_{\leq k,\alpha} +
\bar R_{k} \W_k  \partial_\alpha (T_b  \bar R_{\leq k,\alpha})  
\, d\alpha, \\ 
J_2^b =  \sum_{k > 4} \int  b_{<k}  \bar R_{k,\alpha} \W_{\leq k} R_{k,\alpha} +  \bar R_{k} T_b  W_{\leq k,\alpha} 
  R_{k,\alpha} + 
 \bar R_{k} \W_{\leq k}  \partial_\alpha (b_{<k} R_{k,\alpha}) 
\, d \alpha. 
\end{split}
\]
Integrating by parts we rewrite these integrals as
\[
\begin{aligned}
&J_1^b =  \sum_{k > 4} \int   
\bar R_{k} \W_k  \partial_\alpha ((T_b - b_{<k})  \bar R_{\leq k,\alpha}) 
\, d\alpha, \\ 
&J_2^b =  \sum_{k > 4} \int   \bar R_{k} (T_b - b_{<k})  \W_{\leq k,\alpha} 
  R_{k,\alpha}
\, d \alpha .
\end{aligned}
\]
Here the expressions $(T_b - b_{<k})  \bar R_{\leq k,\alpha}$, respectively 
$(T_b - b_{<k})  W_{\leq k,\alpha} $ are of the same type as the expressions
considered in Lemma~\ref{l:dt-rw} as part of $(\partial_t + T_b \partial_\alpha) R$,
respectively $(\partial_t + T_b \partial_\alpha) \W$. Thus they also satisfy the bounds 
in Lemma~\ref{l:dt-rw}, and the desired conclusion follows.

\end{proof}

\begin{proof}[Proof of Lemma~\ref{l:low}, $n=1$ ]
Because of the better bounds for the lower order terms in \eqref{low-1}, this proof 
is straightforward and is omitted.
\end{proof}
\end{proof}

\section{Proof of the main results}
Given the estimates obtained in the previous sections both for the main evolution ~\ref{FullSystem-re}
and for the linearized equation, the proof of the main results in Theorem~\ref{t:loc} and Theorem~\ref{t:long}
are fairly routine. Thus, in this section we provide an outline of the proofs only. For a more in-depth exposition
of arguments of this type we refer the reader to the earlier article \cite{HIT} devoted to the infinite depth problem.
We will however emphasize the differences between the finite and infinite depth case. 

\begin{proof}[Proof of Theorem~\ref{t:loc}, outline]
Due to scaling considerations we can work with $h=1$ and $g \lesssim 1$. 
The main steps in the proof are as follows:

\medskip

{\em 1. Existence of regular solutions.} Here we start with initial
data $(W,Q)(0) \in L^2 \times H^\frac12$ and $(\W,R)(0) \in H^n \times
H^{n+\frac12}$ with $n \geq 2$, which has extra regularity both at low
frequency and at high frequency. For such data, local in time
solutions are constructed as weak limits of solutions for a frequency
localized system. In doing this it is convenient to work with the
differentiated equation \eqref{DiagonalQuasiSystem-re}, in order to
have the equations in diagonalized form. For this the argument in
\cite{HIT} applies almost identically.

We note one advantage of working with the holomorphic coordinates,
namely that the free water surface is not required to be a graph. If it were not for this, 
we could simply use the local well-posedness result in \cite{abz} or \cite{L-book}.

\bigskip

{\em 2. Uniqueness of regular solutions.}  
Here we consider two solutions $(W_1,Q_2)$ and $(W_2,Q_2)$ with regularity 
$(W_j,Q_j) \in C([0,T]; \H)$ and $(\W_j,R_j)  \in C([0,T]; \H^n)$ with $n \geq 2$, and show that if their
initial data agree then the two solutions must be equal. Note that while more regularity 
is assumed at high frequency, that is no longer the case at low frequency.

For the proof one subtracts the two sets of equations, estimating the difference of the 
two solutions for the differentiated equation \eqref{DiagonalQuasiSystem-re}. 
The key point is that up to perturbative terms, the difference $(w,r) = (\W_1-\W_2,R_1-R_2)$ solves a linear 
system similar to our model evolution for the linearized equation  \eqref{PLinSys}.
Then one can conclude the proof of uniqueness in a standard manner using Gronwall's inequality.

\bigskip
{\em 3. Lifespan bounds in terms of the $\H^1$ size of the data.} The lifespan of solutions 
constructed above depends both on the $\H^n$ size of the data $(\W,R)(0)$ and on $g$.
Here we show that we can in effect obtain lifespan bounds which depend only on the 
$\H^1$ size of the data and which are independent of $g$.  To be precise, we take initial data 
which satisfy the bounds
\begin{equation}\label{large-bd}
\| (W,Q)(0)\|_{\H} \leq g {\mathcal M}_0, \qquad \| (\W,R)(0)\|_{\H} \leq g {\mathcal M}_0, \qquad \| (\W_\alpha,R_\alpha)(0)\|_{\H} \leq  {\mathcal M}_0,
\end{equation}
as well as the pointwise bounds
\begin{equation}\label{large-bd+}
\| Y(0)\|_{L^\infty} \leq \K_0, \qquad \Im W+ 1 \geq c_0 > 0.
\end{equation}
Then we will show that there exists $T = T({\mathcal M}_0,Y_0,c_0)$ so that the solutions exist on $[-T,T]$ with similar bounds.

For the proof we use a bootstrap argument, assuming that the following bounds 
hold in $[0,T]$:
\begin{equation}\label{boot-bd}
\| (W,Q)\|_{\H} \leq g {\mathcal M}, \qquad \| (\W,R) \|_{\H} \leq g {\mathcal M}, \qquad \| (\W_\alpha,R_\alpha)\|_{\H} \leq  {\mathcal M},
\end{equation}
as well as the pointwise bounds
\begin{equation}\label{boot-bd+}
\| Y\|_{L^\infty} \leq \K, \qquad W+ 1 \geq c > 0.
\end{equation}
Then we need to show that for a suitable choice of ${\mathcal M}$, $\K$, $c$ depending on 
${\mathcal M}_0$, $\K_0$ and $c_0$ but not on $g$ we can improve all these bounds. 
Through the following computations we denote by $C_0$ various constants
which only depend on ${\mathcal M}_0$ and $K_0$.

We begin by observing that by Sobolev embedding our control parameters satisfy
\[
A,B \leq C({\mathcal M},\K), \qquad a \geq c g.
\]
Hence by the energy estimates for the differentiated equation in Proposition~\ref{t:en=large} we obtain 
\[
\| (\W,R)(0)\|_{\H} \leq g c^{-1}  C_0 (1+t C({\mathcal M},\K)),
\qquad \| (\W_\alpha,R_\alpha)\|_{\H} \leq  c^{-1}  C_0 C(\K) (1+ t C({\mathcal M},\K)),
\]
where the $\K$ dependence in the second bound is caused by the need to invert a $1+\W$ factor, see \cite{HIT} for a full argument.

To bound $(W,Q)$ in time we use the equations directly to obtain
\[
\| (W,Q)\|_{\H} \leq g( C_0 + t C({\mathcal M},\K)).
\]
To bound $Y$ in $L^\infty$ we reuse the argument in \cite{HIT},
which yields 
\[
\|Y(t)\|^2_{L^\infty} \leq  C_0(1+t C({\mathcal M},\K)).
\]
Finally, to bound $\Im W$ from below we use directly  the $W$ equation 
to obtain
\[
(\partial_t + \Re F \partial_\alpha) \Im W = (1+\Re W_\alpha) \Im \left(\frac{R}{1+\bar\W}\right),
\]
which yields
\[
\inf_{\alpha \in \R} 1+\Im W(t,\alpha) \geq c_0 - tC({\mathcal M},\K)  .
\]

Summarizing, in order to close the bootstrap we need to have the bounds
\[
{\mathcal M} > C_0 c^{-1} C(\K) (1+t C({\mathcal M},\K)), \qquad \K^2 > C_0 c^{-1} (1+t C({\mathcal M},\K)),
\qquad c < c_0 - tC({\mathcal M},\K)  .
\]
This is achieved by first choosing $c = c_0 /2$, then $\K$ large enough  $\K^2 = 2 C_0 c^{-1}$  next ${\mathcal M}$ large enough ${\mathcal M} = 2 C_0 C(\K) c^{-1}$,
and finally a small enough $T < T({\mathcal M},\K,c)$.

\bigskip

{\em 4. $\H^n$ solutions for $n \geq 2$.} Here we relax our low
frequency regularity assumption for the data to $(W,Q)(0) \in \H$,
while keeping the high frequency regularity $(\W,R)(0) \in \H^n$, $n
\geq 2$, and prove that solutions still exist. By Step 2, such solutions are also unique. To obtain such solutions we consider a sequence of data $(W_n,Q_n)(0)$
with regularity $(W_n,Q_n)(0) \in L^2 \times H^{\frac12}$ so that 
\[
(W_n,Q_n)(0) \to (W,Q)(0) \text{ in } \H, \qquad (\W_n,R_n)(0) \to (\W,R)(0)  \text{ in } \H^2.
\]
This is easily achieved by cutting off the low frequencies
\[
(W_n,Q_n)(0) = P_{>-n} (W,Q)(0).
 \]
 For $n$ large enough this family of data is uniformly bounded in the sense of 
 \eqref{large-bd}, so by the previous step they generate solutions $(W_n,Q_n)$ with uniform bounds life-span. But then the estimates on the linearized equation
in Section~\ref{s:linearization} show that the sequence $(W_n,Q_n)$ converges
to some $(W,Q)$ uniformly in the $\mfH$ topology. Due to the uniform  bounds 
on $(W_n,Q_n)$ this linearly yields $(W_n,Q_n) \to (W,Q)$ in $ \mfH^{2-}$.
Thus $R$ is well defined and we also have $(\W_n,R_n) \to (\W,R)$ in $\H^1$. Using now the uniform bounds on $(\W_n,R_n)$ we obtain weak convergence $(\W_n,R_n) \to (\W,R)$ in $\H^2$, and strong convergence in all weaker topologies. Thus we have 
obtained the desired solutions $(W,Q)$.

\bigskip

{\em 5. Rough  solutions.} Here we show that the solution operator  constructed above for data $(W,Q)(0) \in \H$ with $(\W,R)(0) \in \H^2$ extends continuously to data
with only $(W,Q)(0) \in \H$ and $(\W,R)(0) \in \H^1$. 

Indeed, consider some data which only satisfies the latter requirement. 
Then we regularize the data $(W,Q)(0)$ to $(W_n,Q_n)(0) = P_{<n}(W,Q)(0)$.
This linearly guarantees convergence
\[
 |(W_n,Q_n)(0) -(W,Q)(0)| \to 0 \qquad \text{in } H^{2},
\]
which also shows that $\W_n(0) \to \W(0)$ uniformly,
and also 
\[
(\W_n,R_n) \to (\W,R)   \qquad \text{in } H^{1}.
\]
Now we turn our attention to the key point, which is 
to improve this last convergence to $\H^1$. We will in effect do slightly better
than that, and for this we need to work with slowly varying 
frequency envelopes. Precisely, we have the following:

\begin{lemma}\label{data-cut}
Let $\{c_n\}_{n \geq 0}$ be a slowly varying frequency envelope 
for $(\W,R)(0)$ in $\H^1$. Then we have the estimate
\[
2^{\frac32 n} \| (\W_n,R_n)(0) - P_{<n}(\W,R)(0)\|_{\H} + 2^{-n}\|(\W_n,R_n)(0)\|_{\H^2}
\lesssim_A c_n .
\]
\end{lemma}
We note that this lemma not only shows that  $(\W_n,R_n)(0) \to (\W,R)(0)$ in $\H^{1}$, but also that they share the common $c_n$ frequency envelope.

\begin{proof}
We drop the ``$(0)$" notation for this proof.
Only the $R$ part of the bounds is nontrivial. Expressing all in terms of $R$ 
and $W$, for the first expression above we need to bound in $H^\frac12$ the 
difference
\[
\frac{1}{1+P_{<n} \W} P_{<n} [R(1+\W)] -P_{<n} R =
\frac{1}{1+P_{<n} \W} ( P_{<n} [R\W] -  P_{<n} R P_{<n} \W).
\]

We will bound the last difference in $L^2$ using the usual paradifferential 
decomposition. We can express it as
\[
 P_{<N} [R\W] -  P_{<n} R P_{<n} \W = \Pi(P_n R,P_n \W)
 + [P_{<n}, R_{<n-4}]  P_n \W + [P_{<n}, \W_{<n-4} ]P_n R .
\]
Estimating the high frequency factors in $L^2$ and the low frequency factors in $L^\infty$ we obtain
\[
\| P_{<n} [R\W] -  P_{<n} R P_{<n} \W\|_{L^2} \lesssim A 2^{-\frac32 n} c_n .
\]

The $R_n$ bound in the second expression above is easier and is left for the reader.
\end{proof}

Once we have uniform bounds for $(\W_n,R_n)(0)$ in $\H^1$, by the previous step it follows that the corresponding solutions $(\W_n,R_n)$ have a uniform life-span, with uniform bounds. Our next goal is to show that the frequency envelope bounds 
are inherited also by the solutions.

\begin{lemma}
Let $(W_n,Q_n)$ be the solutions associated to the initial data as above.
Then we have the estimates
\begin{equation}\label{wrn-low}
\| (W_{n+1},Q_{n+1})-(W_n,Q_n) \|_{\mfH} \lesssim_{A,B} 2^{-2n} 
\end{equation}
\begin{equation}\label{wrn-high}
\| (\W_n,R_n) \|_{\H^2} \lesssim_{A,B} 2^{n} c_n,
\end{equation}
respectively
\begin{equation}\label{wrn-diff}
\| (\W_{n+1},R_{n+1})-(\W_n,R_n) \|_{\H} \lesssim_{A,B} 2^{-n} c_n.
\end{equation}
\end{lemma}

\begin{proof}
Given the $\H^2$ bound for the initial data $(W_n,Q_n)(0)$ in the previous lemma,
the bound \eqref{wrn-high} is a direct consequence of our higher order 
energy bounds.

For \eqref{wrn-low} we will use instead the linearized equation. Precisely,
we now interpret $n$ as a continuous parameter. Then the functions
\[
(w,q)= \frac{d}{dn} (W_n,Q_n)
\]
solve the linearized equation, and have initial data $(w,q)(0)= P_n(W,Q)(0)$
localized at frequency $2^n$. Considering now the diagonalized variables
$(w,r)= (w,q+R\Til^2 w)$, an argument similar to the proof of Lemma~\ref{data-cut}
shows  that their data satisfies 
\begin{equation}\label{lin-est-n}
\| (w,r)(0)\|_{\H} \lesssim 2^{-2n} c_n.
\end{equation}

Applying the bounds for the linearized equation Theorem~\ref{LinQuadWP}
we extend the estimate \eqref{lin-est-n} along the flow,
\begin{equation}\label{lin-est-nt}
\| (w,r)\|_{\H} \lesssim 2^{-2n} c_n
\end{equation}

For the estimate \eqref{wrn-diff} we first bound the high frequency part using the \(\H^2\) bound \eqref{wrn-high}. Precisely,
\[
\|P_{>n+1}(\W_{n+1},R_{n+1})\|_{\H} + \|P_{>n}(\W_n,R_n)\|_{\H} \lesssim 2^{-n}c_n,
\]
where the constant is independent of \(n\).

To bound the low frequency part we define
\[
(w_1,r_1) = \frac{d}{dn}(\W_n,R_n).
\]
We observe that in terms of \((w,r)\) we have
\[
(w_1,r_1) = (w_\alpha,r_\alpha + R_n(1+\Til^2)w_\alpha + (R_n)_\alpha\Til^2 w).
\]
A simple application of the usual Littlewood-Paley trichotomy then yields the estimate
\begin{equation}\label{dlin-est-nt}
\|P_{< n}(w_1,r_1)\|_{\H}\lesssim_A2^n\|(w,r)\|_{\H}.
\end{equation}

From the estimate \eqref{lin-est-nt} we gain an $\mfH$ bound for $(w,q)$,
which integrated between $[n,n+1]$ yields \eqref{wrn-low}. On the other 
hand, we have
\[
\frac{d}{dn}P_{<n}(\W_n,R_n) = P_{<n}(w_1,r_1) + P_n(\W_n,R_n),
\]
where the second term may again be bounded using \eqref{wrn-high}. Integrating \eqref{dlin-est-nt} we obtain \eqref{wrn-diff}.

\end{proof}

The bounds in the last lemma insure not only that the sequence $(W_n,Q_n)$
converges strongly to a solution $(W,Q)$ in the sense that
\[
(W_n,Q_n) \to (W,Q) \qquad \text{uniformly in } H^2
\]
\[
(\W_n,R_n) \to (\W,R) \qquad \text{uniformly in } \H^1
\]
but also that $(\W,R)$ inherits the same frequency envelope $\{c_n\}$ in $\H^1$. 

Once we have constructed the rough solutions $(W,Q)$ as the unique limit of the 
regularized problems, the frequency envelope bounds easily lead to continuous dependence with respect to data. This is a standard argument; for which we refer
the reader to \cite{HIT}.

\end{proof}

\begin{proof}[Proof of Theorem~\ref{t:long}, outline]
 Using the spatial scaling , it suffices to assume that
$h = 1$. Given the  initial data $(W, Q)(0)$ for \eqref{FullSystem-re} satisfying 
 \[
  g^{-1} \| (W,Q)(0)\|_{\H} + g^{-1} \| (\W,R)(0)\|_{\H} +
   \| (\W_\alpha,R_\alpha)(0)\|_{\H}  \leq \epsilon,
  \]
we consider the solutions on a time interval $[0,\, T]$ and seek to prove the estimate
\begin{equation}
\label{prove}
  g^{-1} \| (W,Q)(t)\|_{\H} + g^{-1} \| (\W,R)(t)\|_{\H} +
   \| (\W_\alpha,R_\alpha)(t)\|_{\H}  \leq C\epsilon, \quad t\in [0, \ T],
  \end{equation}
provided that $T$ is much smaller than $\epsilon^{-2}$. In view of our local well-posedness result this shows that the
solutions can be extended up to time $T_{\epsilon} =C\epsilon^{-2}$, concluding the proof of the theorem.

In order to prove \eqref{prove} we use a bootstrap argument; we make the bootstrap assumption 
\begin{equation}
\label{bootstrap}
 g^{-1} \| (W,Q)(t)\|_{\H} + g^{-1} \| (\W,R)(t)\|_{\H} +
   \| (\W_\alpha,R_\alpha)(t)\|_{\H}  \leq 2C\epsilon, \quad t\in [0, \ T].
\end{equation}
From \eqref{bootstrap}, and by Sobolev embedding theorem, \eqref{sobolev-A} and \eqref{sobolev-B}, our control norms $A$ and $B$ satisfy
\[
A, B \lesssim C\epsilon.
\]

To bound $(W, Q)$ in time we  directly use the conserved energy $\mathcal{E}$. Using the expression \eqref{reprezentare} for $\mathcal{E}$ we see that 
\[
\mathcal{E}=(1+O(A))E_{0}(W, Q).
\]
Hence, using the bootstrap assumption \eqref{bootstrap} we obtain
\[
\Vert (W, Q)\Vert_{\mathcal{H}}\lesssim g( \epsilon +C\epsilon^2).
\]

The bound for $(\W, R)$ can be obtained from the cubic energy estimates already established 
for the differentiated equation in Proposition~\ref{t:en=small}. To obtain such a bound we first need to recall that the cubic energy estimate in there is in terms of the control norm $\norm_1$, which is now taken uniformly in time.
Explicitly, we integrate \eqref{dt-cubic} in time 
\begin{equation}
\label{oly1}
E^{1, (3)}(\W, R)(t)\lesssim E^{1, (3)}(\W, R)(0)+TAB\norm_1 ^2,
\end{equation}
and use \eqref{en-cubic} to obtain
\begin{equation}
\label{oly2}
E_0(\W, R)(t)\lesssim E_{0}(\W, R)(0)+TAB\norm_1 ^2 +A \norm_1 ^2.
\end{equation}
We further need control of $\norm_1$ norm, and this follows from
\begin{equation}
\label{oly3}
\norm_1^2\lesssim_A E_0(W,Q)+E_0(\W, R),
\end{equation}
where the first term on the right is needed in order to account for the low frequencies in $(\W, R)$.
Thus, we arrive at
\[
\begin{aligned}
\Vert (\W, R)\Vert ^2_{L^{\infty}(0, T), \mathcal{H}}\lesssim &E_0(\W, R)(0)+TAB \sup_{t\in [0,T]} (E_0 (W,Q)(t) +E_0 (\W, R)(t))\\
&+A\sup_{t\in [0,T]} ( (E_0 (W, Q)(t)+E_0(\W, R)(t)),
\end{aligned}
\]
and using the bootstrap assumptions \eqref{bootstrap} we get
\[
\Vert (\W, R)\Vert_{\mathcal{H}}\lesssim g(\epsilon +TC^2\epsilon^3+C\epsilon^2).
\]

The bound $(\W_{\alpha}, R_{\alpha})$ is obtained in the same way as above 
\[
\begin{aligned}
\Vert (\W_{\alpha}, R_{\alpha})\Vert^2 _{L^{\infty}(0, T), \mathcal{H}}\lesssim &\ E_0(\W, R)(0)\\
&\ +TAB \sup_{t\in [0,T]} (E_0 (W,Q)(t) +E_0 (\W, R)(t) +E_{0}(\W_{\alpha}, R_{\alpha})(t))\\
&\ +A \sup_{t\in [0,T]} ((E_0 (W,Q)(t) +E_0 (\W, R)(t) +E_{0}(\W_{\alpha}, R_{\alpha})(t)),
\end{aligned}
\]
and using the bootstrap assumptions \eqref{bootstrap} we get
\[
\Vert (\W_{\alpha}, R_{\alpha})\Vert_{\mathcal{H}}\lesssim \epsilon +TC^2\epsilon^3+C\epsilon^2.
\]
Hence, the estimate in \eqref{prove} follows provided that $C\gg 1$ and $T \ll C^{-1}\epsilon^{-2}$.
Similar bootstrap  argument applies for higher derivatives.
\end{proof}

% %%%%%%%%%%%%%%%%
%%% APPENDIX %%%
%%%%%%%%%%%%%%%%

\begin{appendix}

\section{Multilinear estimates}

\subsection{Some harmonic analysis results}
In this section we collect a number of elementary estimates that will allow us to adapt the estimates established in infinite depth case \cite{HIT} to the finite depth setting.

We take an inhomogeneous Littlewood-Paley decomposition \(I = S_0 + \sum_{j\geq1}P_j\) and denote
\[
f_0 = S_0f,\qquad f_j = P_jf,\quad j\geq1.
\]
We define the inhomogeneous Besov space \(B^{s,p}_q\) with norm
\[
\|f\|_{B^{s,p}_q}^q = \sum\limits_{j\geq0}\|\langle D\rangle^sf_j\|_{L^p}^q,
\]
with the usual modification when \(q=\infty\).
We also define the inhomogeneous space \(\bmo\) of functions of bounded mean oscillation with norm
\[
\|f\|_{\bmo} = \|f\|_{\BMO} + \|f_0\|_{L^\infty},
\]
where
\[
\|f\|_{\BMO} = \sup\limits_Q \frac{1}{|Q|}\int_Q|f - f_Q|\,d\alpha,\qquad f_Q = \int_Q f\,d\alpha,
\]
and the supremum is taken over all intervals \(Q\subset\Rl\). We recall that \(B_2^{0,\infty}\subset\bmo\subset B^{0,\infty}_\infty\). We define the corresponding \(\bmo\)-Sobolev spaces by
\[
\|u\|_{\bmo^s} = \|\langle D\rangle^su\|_{\bmo}.
\]

We define the paraproduct operators
\[
T_fg = \sum\limits_{j>4}f_{<j-4}g_j,\qquad \Pi[f,g] = \sum\limits_{\substack{|j-k|\leq 4\\j,k\geq0}}f_jg_k,
\]
and the associated product decomposition
\[
fg = T_fg + T_gf + \Pi[f,g].
\]
We then have the following estimates (see for example \cite[Propositions 2.2, 2.6]{HIT}):

\medskip
\begin{lemma}[Paraproduct bounds]\label{lem:Paraproducts}~

a) Coifman-Meyer paraproduct estimates. For \(1<p<\infty\) and \(s,\sigma\geq0\),
\begin{equation}\label{Para1}
\begin{aligned}
\|\langle D\rangle^sT_{\langle D\rangle^\sigma u}f\|_{L^p} &\lesssim \|f\|_{\bmo^{s+\sigma}}\|u\|_{L^p},\\
\|\langle D\rangle^s\Pi[f,\langle D\rangle^\sigma u]\|_{L^p} &\lesssim\|f\|_{\bmo^{s+\sigma}}\|u\|_{L^p}.
\end{aligned}
\end{equation}

b) Besov endpoint estimates. For \(s\geq0\),
\begin{equation}\label{Para2}
\begin{aligned}
\|\langle D\rangle^sT_{\langle D\rangle^\sigma u}f\|_{L^\infty} &\lesssim \|f\|_{B^{s+\sigma,\infty}_2}\|u\|_{B^{0,\infty}_2},\qquad \sigma>0\\
\|\langle D\rangle^s\Pi[f,\langle D\rangle^\sigma u]\|_{L^\infty} &\lesssim \|f\|_{B^{s+\sigma,\infty}_2}\|u\|_{B^{0,\infty}_2},\qquad \sigma\geq0.
\end{aligned}
\end{equation}

c) \(\BMO\) endpoint estimates. For \(s\geq0\),
\begin{equation}\label{Para3}
\begin{aligned}
\|\langle D\rangle^sT_{\langle D\rangle^\sigma u}f\|_{\bmo} &\lesssim \|f\|_{\bmo^{s+\sigma}}\|u\|_{\bmo},\qquad \sigma>0\\
\|\langle D\rangle^s\Pi[f,\langle D\rangle^\sigma u]\|_{\bmo} &\lesssim \|f\|_{\bmo^{s+\sigma}}\|u\|_{\bmo},\qquad \sigma\geq0.
\end{aligned}
\end{equation}
\end{lemma}
\medskip

The following bounds, which are direct consequences
of the classical Coifman-Meyer estimates, are closely related:

\medskip
\begin{lemma}[Commutator bounds]\label{l:CMComm}~

a) Let \(\mc M \in S^1\) be a smooth Fourier multiplier with principal symbol homogeneous of order \(1\). Then for \(1<p<\infty\) we have the estimate
\begin{equation}\label{est:CMComm}
\|[\mc M,f]u\|_{L^p}\lesssim \|f_\alpha\|_{L^\infty}\|u\|_{L^p}.
\end{equation}

b) Let \(\mc M \in S^s\) be a smooth Fourier multiplier with principal symbol 
homogeneous of order \( s\) with $0\leq s < 1$. Then for \(1<p<\infty\) we have the  estimate
\begin{equation}\label{est:CMComm0}
\|[\mc M,f]u\|_{L^p}\lesssim \|\Til f\|_{\bmo^s}\|u\|_{L^p}.
\end{equation}
\end{lemma}

We also need the following more involved estimate:

\begin{lemma}
The following double commutator bound holds:
\begin{equation}\label{double}
\|[[\Til \D, b],\D]\|_{L^2 \to L^2} \lesssim \| b_\alpha\|_{\bmo}.
\end{equation}
\end{lemma}

\begin{proof}
We consider the paradifferential decomposition of the multiplication by $b$. For the map
\[
u \to T_u b,
\]
we have the bounds
\[
\| T_u b\|_{H^1} \lesssim \|u\|_{L^2}\|b_\alpha\|_{\bmo},
\qquad \| T_u b\|_{L^2} \lesssim \|u\|_{H^{-1}}\|b_\alpha\|_{\bmo},
\]
which follow from the first estimate in \eqref{Para1}.
Thus we can neglect the commutator structure.

Similarly, for  the map
\[
u \to \Pi[u ,b],
\]
we have the bounds
\[
\| \Pi[u, b]\|_{H^1} \lesssim \|u\|_{L^2}\|b_\alpha\|_{\bmo},
\qquad \| \Pi[u,b]\|_{L^2} \lesssim \|u\|_{H^{-1}}\|b_\alpha\|_{\bmo},
\]
from the second estimate in \eqref{Para1}, and again we can neglect the commutator structure.

It remains to consider the contribution of $T_b$. For this we write
\[
[T_b, \Til \D] u = \sum_{k} [b_{< k-4},\Til \D] u_k = 
\sum_k 2^{-\frac{k}2} B_k(\partial_\alpha b_{< k-4}, u_k ),
\]
where $B_k$ are translation invariant bilinear operators with 
uniformly integrable kernels. Commuting again we have
\[
[[T_b, \Til \D],\D] u = \sum_{k}  
\sum_k 2^{-\frac{k}2} B_k([\partial_\alpha b_{< k-4},\D], u_k )
=\sum_k 2^{-k} C_k(\partial_\alpha^2 b_{< k-4}, u_k ),
\]
where again $C_k$ are translation invariant bilinear operators with 
uniformly integrable kernels. Then we can bound
\[
\| [[T_b, \Til \D],\D] u\|_{L^2} \lesssim \sum_{k} 
2^{-k}  \| \partial_\alpha^2 b_{< k-4}\|_{L^\infty} \|u_k\|_{L^2}
\lesssim \|b_\alpha\|_{B^{0,\infty}_\infty} \|u\|_{L^2},
\]
which suffices.
\end{proof}

\medskip

We will make use of the
following estimates for rapidly decaying Fourier multipliers:

\medskip
\begin{lemma}\label{Smoothing}
Let \(S\) be a Fourier multiplier with Schwartz symbol. Then
for  all real \(s,\sigma\), \(1\leq p\leq\infty\) and \(N\geq0\) we have the estimate
\begin{equation}\label{est:StrongLoHi}
\|\langle D\rangle^sST_f\langle D\rangle^{\sigma}u\|_{L^p} + \|\langle D\rangle^sT_fS\langle D\rangle^\sigma u\|_{L^p}\lesssim_N \|f\|_{B^{-N,\infty}_\infty}\|u\|_{L^p}.
\end{equation}

Further, we have the commutator estimate
\begin{equation}\label{est:WeakLoHi}
\|\langle D\rangle^s [S,T_f]\langle D\rangle^\sigma u\|_{L^p}\lesssim_N \|\Til f\|_{B^{-N,\infty}_\infty}\|u\|_{L^p}.
\end{equation}
\end{lemma}
\bpf
This is a standard argument based on the classical Littlewood-Paley trichotomy. 
Due to the frequency localization of the paraproduct operator $T_f$, the 
rapid decay in the symbol of $S$ transfers to both the input $u$, the factor $f$
and to the output. This directly leads to the derivative gains in the Lemma.
\epf
\medskip

The next result serves to bound commutators with the Tilbert transform $\Til$:
\medskip
\begin{lemma}
  Let \(\mc M\) be a Fourier multiplier whose symbol $m(x)$ is bounded
  with $m'(\xi)$ in the Schwartz class.  Then for \(1<p<\infty\) and
  \(s\geq0\) we have the commutator estimates
\begin{equation}\label{est:TilComm}\begin{alignedat}{3}
\|\langle D\rangle^s[\mc M,f]\langle D\rangle^{\sigma}u\|_{L^2} &\lesssim \|\Til f\|_{\bmo^{s+\sigma}}\|u\|_{L^2},\qquad &\sigma\geq0&\\
\|\langle D\rangle^s[\mc M,f]\langle D\rangle^{\sigma}u\|_{L^\infty} &\lesssim \|\Til f\|_{B^{s+\sigma,\infty}_2}\|u\|_{B^{0,\infty}_2},\qquad &\sigma>0&\\
\|\langle D\rangle^s[\mc M,f]\langle D\rangle^{\sigma}u\|_{\bmo} &\lesssim \|\Til f\|_{\bmo^{s+\sigma}}\|u\|_{\bmo},\qquad &\sigma>0&.
\end{alignedat}
\end{equation}
\end{lemma}
\bpf
By hypothesis we can split the multiplier $\mc M$ as 
\[
\mc M = {m(\infty) }P_{> 10} + m(-\infty) P_{<-10} + S,
\]
where $P_{>10}$ and $P_{<-10}$ are multipliers whose symbols are
smooth cutoff functions selecting the indicated frequency regions,
and $S$ has Schwartz kernel. 
For the commutator with $P_{> 10}$ (and similarly with $P_{<-10}$) we have
\[
\langle D\rangle^s[P_{>10},T_f]\langle D\rangle^{\sigma}u_{>20}\equiv0,\qquad \langle D\rangle^s[P_{>10},f_0]\langle D\rangle^{\sigma}u_0 \equiv0.
\]
The estimates then follow from Lemma \ref{lem:Paraproducts} as in the infinite 
depth case \cite{HIT}.

For the second term we write
\begin{align*}
\langle D\rangle^s[S,f]\langle D\rangle^{\sigma}u &= \langle D\rangle^s[S,T_f]\langle D\rangle^{\sigma}u + \langle D\rangle^s[S,f_0]\langle D\rangle^{\sigma}u_{\leq 4}  + \langle D\rangle^sS T_{\langle D\rangle^{\sigma}u}f\\
&\quad  - \langle D\rangle^sT_{S \langle D\rangle^{\sigma}u}f + \langle D\rangle^sS \Pi[f_{\geq1},\langle D\rangle^{\sigma}u] - \langle D\rangle^s\Pi[f_{\geq1},S \langle D\rangle^{\sigma}u].
\end{align*}
The first and second terms term may be estimated using \eqref{est:WeakLoHi}. 
The remaining terms may be estimated using Lemma \ref{lem:Paraproducts}.
\epf
\medskip

Finally we recall two Moser estimates, the first of which is classical, and the second from \cite{HIT}.
\medskip
\begin{lemma}
Let \(F\) be a smooth function such that \(F(0) = 0\) then for \(s\geq0\) and \(u\in L^\infty\cap H^s\) we have the Moser estimate
\begin{equation}\label{MoserL2}
\|F(u)\|_{H^s}\lesssim_{\|u\|_{L^\infty}}\|u\|_{H^s}.
\end{equation}
Similarly, for $u \in \bmo^s$ we have
\begin{equation}\label{Moserbmo}
\|F(u)\|_{\bmo^s}\lesssim_{\|u\|_{L^\infty}}\|u\|_{\bmo^s}.
\end{equation}
\end{lemma}
\medskip

\subsection{Holomorphic functions on the strip}
We recall the projection to holomorphic functions is given by
\[
\P u = \frac12\left[(1 - i\Til)\Re u + i(1+i\Til^{-1})\Im u\right]=\frac14\left[(2 - i\Til + i\Til^{-1})u - i(\Til+\Til^{-1})\bar u\right].
\]
As a consequence,
\begin{align*}
\Re\P u &= \frac12\left[\Re u - \Til^{-1}\Im u\right]=\frac14\left[(1+i\Til^{-1})u +(1 - i\Til^{-1})\bar u\right],\\
\Im\P u &= -\frac12\left[\Til\Re u - \Im u\right]=\frac{1}{4i}\left[(1 - i\Til)u - (1+i\Til) \bar u\right].
\end{align*}
We also recall the definition of the inner product, which is given by
\begin{align*}
\langle u,v\rangle &= \int \Til\Re u\cdot\Til\Re v + \Im u\cdot\Im v\,d\alpha\\
&= \frac12\Re\int \left(\Til u\cdot\Til\bar v + u\cdot\bar v\right) + \left(\Til u\cdot\Til v - u\cdot v\right)\,d\alpha.
\end{align*}
It is useful to understand 
the adjoints of multiplication operators with respect to this inner product:

\medskip
\begin{lemma}\label{lem:Adjoint}
  Let \(f\) be a complex-valued function. With respect the inner
  product \(\langle\cdot,\cdot\rangle\) the adjoint of the operator
  \(\M_f\) is
\begin{equation}\label{Adjoint}
\begin{aligned}
\M_f^*u &= \Til^{-1}(\P - \bar\P)\left[\bar f\Til\P[u]\right].
\end{aligned}
\end{equation}
\end{lemma}
\bpf
Using that \(\Til\Re\P[v] = -\Im\P[v]\) and that \(\Til\) is skew-symmetric we may write the inner product as
\begin{align*}
\la\M_fu,v\ra &= \int \Re[fu]\cdot\Til\Im\P[v] - \Im[fu]\cdot\Til\Re\P[v]\\
&= - \int \Til\Re u\cdot \Til^{-1}\Im(\bar f\Til\P[v]) + \Im u\cdot \Re( \bar f\Til \P[v]).
\end{align*}
As a consequence we have
\[
\M_f^*v = - \Til^{-2}\Im(\bar f\Til\P[v]) - i\Re(\bar f\Til\P[v]).
\]
Comparing this to the expression for \(\P\) we obtain the formula \eqref{Adjoint}.
\epf
\medskip

The following immediate consequence of the above Lemma is very handy to use:

\medskip
\begin{corollary}\label{c:adj}
If $u$ and $v$ are holomorphic functions in $\mfH$ then we have
\begin{equation}\label{ceq:adj}
\langle f \Til u, v \rangle = - \langle u,  \bar f \Til v \rangle.
\end{equation}
\end{corollary}
\medskip

As our function spaces \(\mfH\), \(\H\) lose a derivative at low
frequency in the real component, for a space \(X\) of complex-valued
functions we define the norm
\[
\|f\|_{\mfH X}^2 = \|\Til\Re f\|_X^2 + \|\Im f\|_X^2,
\]
with the shorthand \(\mfH = \mfH L^2\).
We will frequently use the following estimate for the commutator 
with the projection to holomorphic functions:

\medskip
\begin{lemma}\label{lem:MainCommutatorBound}
For \(s\geq0\) we have the estimates
\begin{equation}\label{PComm}
\begin{alignedat}{3}
\|\langle D\rangle^s [\P,f] \langle D\rangle^\sigma \Til g\|_{\mfH} &\lesssim \|f\|_{\mfH\bmo^{s+\sigma}}\|g\|_{\mfH},\qquad &\sigma\geq0&\\
\|\langle D\rangle^s[\P,f] \langle D\rangle^\sigma \Til  g \|_{\mfH L^\infty} &\lesssim \|f\|_{\mfH B^{s+\sigma,\infty}_2}\|g\|_{\mfH B^{0,\infty}_2},\qquad &\sigma>0.&
\end{alignedat}
\end{equation}
\end{lemma}
\bpf
We may write the real and imaginary parts of the commutator as
\begin{align*}
\Til\Re[\P,f]\langle D\rangle^\sigma \Til g &= \frac12[\Til,\Re f]\langle D\rangle^\sigma\Im g - \frac12[\Til,\Im f]\langle D\rangle^\sigma\Til^2 \Re g\\
&\quad -\frac12\Im f (1 + \Til^2)\langle D\rangle^\sigma \Til \Re g,\\
\Im[\P,f]\langle D\rangle^\sigma \Til g &= - \frac12[\Til,\Re f]\langle D\rangle^\sigma\Til \Re g + \frac12[\Til,\Im f]\langle D\rangle^\sigma\Til\Im g\\
&\quad + \frac12\Im f\langle D\rangle^\sigma(1+\Til^2)\Im g.
\end{align*}
The estimates then follow from the commutator estimate \ref{est:TilComm} and the paraproduct estimates \eqref{Para1}, \eqref{Para2} and \eqref{est:StrongLoHi}, using that the operator \(1+\Til^2\) has Schwartz symbol.
\epf
\medskip

Finally we prove the following lemma that allows us to estimate the product of two holomorphic functions in negative Sobolev spaces:
\medskip
\begin{lemma}
If \(f,g\) are holomorphic then for \(s>0\) and \(2\leq p,q\leq\infty\) satisfying \(\frac1p + \frac1q = \frac12\) we have the estimate
\begin{equation}\label{e:HolomProductBound}
\|fg\|_{H^{-s}}\lesssim \|f\|_{L^p}\|g\|_{B^{-s,q}_2}.
\end{equation}
\end{lemma}
\bpf
For each \(j\geq0\) we decompose
\[
\|P_j[f g]\|_{L^2} = \|P_j[f g_{\leq j+10}]\|_{L^2} + \|P_j[f g_{>j+10}]\|_{L^2}.
\]
The first term may now be estimated using dyadic decomposition. For the second term both \(f,g\) must be localized at comparable dyadic frequencies \(\gg 2^j\). In particular, one term must be localized at negative wavenumbers and the other at positive wavenumbers. However, as both terms are holomorphic we may harmlessly apply the projection \(\P\) to each term, which is rapidly decaying on positive wavenumbers.
\epf
\medskip

\subsection{Water wave related bounds}

We begin with the following result for the function \(Y\) which follows directly  from \cite[Lemma 2.5]{HIT} and Moser type estimates:

\medskip
\begin{lemma} The function $Y = \dfrac{\W}{1+\W}$  satisfies the bounds
\begin{equation}\label{est:Ybmo}
\|Y\|_{\bmo^{\frac12}}\lesssim_A g^{-\frac12} B,
\end{equation}
respectively 
\begin{equation}\label{est:Y2}
\|Y\|_{H^{n-1}}\lesssim_A g^{-\frac12}\norm_n, \qquad n \geq 1.
\end{equation}
\end{lemma}
\medskip

Next we consider the advection velocity \(b\):
\medskip
\begin{lemma}\label{l:b}
The the advection velocity \(b\) satisfies  the estimates
\begin{align}
\|\Til b\|_{\bmo^{\frac12}} \lesssim_A g^\frac12 A,\qquad \|\Til b\|_{\bmo^1} \lesssim_A B.\label{est:bBMO}
\end{align}
respectively 
\begin{equation}\label{est:bL2}
\|\Til b\|_{H^{n-\frac12}} \lesssim_A \norm_n, \qquad n \geq 1.
\end{equation}
\end{lemma}
\bpf
We write \(b = b_1 + b_2\) where
\[
b_1 = 2\Re R,\qquad b_2 = -2\Re\P[R\bar Y].
\]

For \(b_1\) we have the estimate
\[
\|\Til b_1\|_{\bmo^s} \leq \|R\|_{\bmo^s}.
\]
For \(b_2\) we may write \(\P[R\bar Y] = [\P,R]\bar Y\) so
\[
\|\Til b_2\|_{\bmo^s} = \|\Til\Re[\P,R]\bar Y\|_{\bmo^s}.
\]
As \(Y\) is antiholomorphic we have
\[
\Til\Re[\P,R]\bar Y = \frac12[\Til,\Re R]\Re\bar Y - \frac12[\Til,\Im R]\Im\bar Y -\frac12\Im R (1 + \Til^2)\Re\bar Y.
\]
For the first two terms we may use the commutator estimate \eqref{est:TilComm} to obtain
\[
\|[\Til,\Re R]\Re\bar Y\|_{\bmo^s} + \|[\Til,\Im R]\Im\bar Y\|_{\bmo^s}\lesssim A\|R\|_{\bmo^s}.
\]
For the final term we simply use that \(1 + \Til^2\) has Schwartz symbol to estimate
\[
\|\Im R (1 + \Til^2)\Re\bar Y\|_{\bmo^s}\lesssim A \|R\|_{\bmo^s}.
\]

The proof of the $L^2$-type bound follows in a similar manner.
\epf

Next we prove a number of estimates for the real frequency shift \(\mfa\). Our estimates are similar to 
\cite[Proposition 2.6]{HIT} although the present case is slightly more involved due to the different projector \(\P\),
as well as the extra term in $\mfa$.

\medskip
\begin{lemma}\label{l:a}
The following bounds hold for the frequency shift \(\mfa\):
\begin{equation}
\|\mfa\|_{L^\infty} \lesssim_A gA,\qquad \|\mfa\|_{\bmo^{\frac12}}\lesssim_A g^\frac12 B,\label{est:a}
\end{equation}
\begin{equation}
\| \mfa\|_{H^{n-1}} \lesssim_A g^\frac12 \norm_n\label{est:a2}
\end{equation}
\begin{equation}
\|\mfa_t + b\mfa_\alpha + g(1 + \Til^2)\Re R_\alpha\|_{L^\infty} \lesssim  gAB\label{est:aTransport}.
\end{equation}
\end{lemma}

\bpf
We recall that \(\mfa = a + \ao\) where
\[
a = 2\Im\P[R\bar R_\alpha],\qquad \ao = g(1+\Til^2)\Re\W.
\]
We will prove the bounds in the Lemma separately for $a$ and for $a_1$.
\medskip

{\em 1. \(L^\infty\), \(\bmo^{\frac12}\) and \(H^{n-1}\) bounds.}
For \(\ao\) we use that \(1+\Til^2\)  has Schwartz symbol to obtain
\[
\|\ao\|_{L^\infty}\lesssim g\|\W\|_{L^\infty},\qquad \|\ao\|_{\bmo^{\frac12}}\lesssim g\|\W\|_{\bmo^{\frac12}},\qquad \|\ao\|_{H^{n-1}}\lesssim g\|\W\|_{H^{n-1}}.
\]

For \(a\) we use that \(\P\bar R_\alpha = 0\) to write \(\Im\P[R\bar R_\alpha ] = \Im[\P,R]\bar R_\alpha\). We then apply the commutator estimate \eqref{PComm} to obtain
\[
\|a\|_{L^\infty}\lesssim \|R\|_{B^{\frac12,\infty}_2}^2,\qquad \|a\|_{\bmo^{\frac12}}\lesssim \|\langle D\rangle^{\frac12}a\|_{L^\infty}\lesssim \|R\|_{B^{\frac34,\infty}_2}^2,
\]
and for the second of these we apply the interpolation estimate
\[
\|R\|_{B^{\frac34,\infty}_2}^2\lesssim \|\langle D\rangle^{\frac12}R\|_{L^\infty}\|R\|_{\bmo^1}.
\]

For the \(H^{n-1}\) estimate we first differentiate
\[
\partial^{n-1}a = 2\Im\sum\limits_{k=0}^{n-1}\P[R^{(k)}\bar R^{(n-k)}].
\]
If \(k\geq1\) then we estimate by interpolation and if \(k=0\) then we apply the commutator bound \eqref{PComm}.
\medskip

{\em 2. Transport equation bounds.}
For \(\ao\) we calculate
\[
(\partial_t + b\partial_\alpha)\ao + g(1+\Til^2)\Re R_\alpha = g(1 + \Til^2)\Re\left[\W_t + b\W_\alpha + \R_\alpha\right] - ig[\Til,b]\W_\alpha.
\]
The first term may be bounded using Lemmas \ref{lem:Paraproducts}, \ref{Smoothing}, the estimate \eqref{M-infty} for \(M\) and that \(1+\Til^2\) has Schwartz symbol. For the second term we apply the commutator estimate \eqref{est:TilComm} to obtain
\[
\|g[\Til,b]\W_\alpha\|_{L^\infty}\lesssim g\|\Til b\|_{B^{\frac34,\infty}_2}\|\W\|_{B^{\frac14,\infty}_2}.
\]
By interpolation,
\[
\|\Til b\|_{B^{\frac34,\infty}_2}\lesssim \|\Til b\|_{\bmo^{\frac12}}^{\frac12}\|\Til b\|_{\bmo^1}^{\frac12},\qquad \|\W\|_{B^{\frac14,\infty}_2}\lesssim \|\W\|_{L^\infty}^{\frac12}\|\W\|_{\bmo^{\frac12}}^{\frac12}.
\]
and we may then apply the estimate \eqref{est:bBMO} for \(b\).

For \(a\) we have
\begin{align*}
(\partial_t + b\partial_\alpha)a
&= 2\Im [\P,\P\left[R_t + bR_\alpha\right]]\bar R_\alpha + 2\Im [\P,R]\partial_\alpha\bar\P\left[\bar R_t + b\bar R_\alpha\right]\\
&\quad + 2\Im\left(b\partial_\alpha\P[R\bar R_\alpha] - \P\left[bR_\alpha\bar R_\alpha\right] - \P\left[R\partial_\alpha\bar \P(b\bar R_\alpha)\right]\right).
\end{align*}
For the first two terms we apply the commutator estimate \eqref{PComm} to obtain
\begin{align*}
&\left\|2\Im [\P,\P\left[R_t + bR_\alpha\right]]\bar R_\alpha\right\|_{L^\infty} + \left\|2\Im [\P,R]\partial_\alpha\bar\P\left[\bar R_t + b\bar R_\alpha\right]\right\|_{L^\infty}\\
&\qquad \lesssim \|\Im\P\left[R_t + bR_\alpha\right]\|_{B^{\frac14,\infty}_2}\|R\|_{B^{\frac34,\infty}_2}.
\end{align*}
We observe that
\[
\Im\P\left[R_t + bR_\alpha\right] = \frac12(g+\mfa)\Re Y - \frac12\mfa + \frac12\Til\left[(g+\mfa)\Im Y\right],
\]
and hence
\[
\|\Im\P\left[R_t + bR_\alpha\right]\|_{B^{\frac14,\infty}_2}\lesssim \|\mfa\|_{B^{\frac14,\infty}_2}(1 + \|Y\|_{L^\infty}) + (g + \|\mfa\|_{L^\infty}) \|Y\|_{B^{\frac14,\infty}_2}.
\]
and the estimate follows from interpolation and the estimates \eqref{est:a} for \(\mfa\) and \eqref{est:Ybmo} for \(Y\).

For the final term appearing in \(a_t + ba_\alpha\) we must ensure that \(b\) does not appear undifferentiated at low frequency. We start by dividing up dyadically according to the frequency of the holomorphic term \(R\):
\[
b\partial_\alpha\P[R\bar R_\alpha] - \P\left[bR_\alpha\bar R_\alpha\right] - \P\left[R\partial_\alpha\bar \P(b\bar R_\alpha)\right] = \sum\limits_{j\geq0}f_j,
\]
where
\[
f_j = b\partial_\alpha\P[R_j\bar R_\alpha] - \P\left[bR_{\alpha,j}\bar R_\alpha\right] - \P\left[R_j\partial_\alpha\bar \P(b\bar R_\alpha)\right].
\]
We then decompose each \(f_j = f_j^\high + f_j^\low\) according to the frequency balance of \(b\) and \(R\),
\begin{align*}
f_j^\high &= b_{>j}\partial_\alpha\P[R_j\bar R_\alpha] - \P\left[b_{> j}R_{\alpha,j}\bar R_\alpha\right] - \P\left[R_j\partial_\alpha\bar \P(b_{> j}\bar R_\alpha)\right],\\
f_j^\low &= b_{\leq j}\partial_\alpha\P[R_j\bar R_\alpha] - \P\left[b_{\leq j}R_{\alpha,j}\bar R_\alpha\right] - \P\left[R_j\partial_\alpha\bar \P(b_{\leq j}\bar R_\alpha)\right].
\end{align*}

When \(b\) is at high frequency we write
\[
f_j^\high = b_{>j}\partial_\alpha[\P,R_j]\bar R_\alpha - [\P,b_{>j}R_{\alpha,j}]\bar R_{\alpha} - [\P,R_j]\partial_\alpha \bar\P(b_{>j}\bar R_\alpha).
\]
Taking the imaginary part and applying the commutator estimate \eqref{PComm} we obtain
\[
\|\Im f_j^\high\|_{L^\infty} \lesssim 2^j\|b_{>j}\|_{B^{\frac14,\infty}_2}\|R_j\|_{L^\infty}\|R\|_{B^{\frac34,\infty}_2}.
\]
Summing over \(j\geq0\) we obtain
\[
\sum\limits_{j\geq0}\|\Im f_j^\high\|_{L^\infty}\lesssim \|b_{>0}\|_{B^{\frac34,\infty}_2}\|R\|_{B^{\frac12,\infty}_2}\|R\|_{B^{\frac34,\infty}_2},
\]
and the estimate follows from interpolation and the estimate \eqref{est:bBMO} for \(b\).

When \(b\) is at low frequency we write
\[
f_j^\low = \partial_\alpha[b_{\leq j},\P](R_j\bar R_\alpha) - b_{\leq j,\alpha}[\P,R_j]\bar R_\alpha + \P[R_j\partial_\alpha[\P,b_{\leq j}]\bar R_\alpha].
\]
Again we apply the commutator estimate \eqref{PComm}, using that \(b_{\leq j,\alpha}\) is real-valued, to obtain
\[
\|\Im f_j^\low\|_{L^\infty}\lesssim 2^{\frac 38j}\|\Til b_{\leq j}\|_{B^{\frac 78,\infty}_2}\|R_j\|_{L^\infty}\|R\|_{B^{\frac34,\infty}_2}.
\]
Summing over \(j\geq0\) we obtain
\[
\sum\limits_{j\geq0}\|\Im f_j^\low\|_{L^\infty}\lesssim \|\Til b\|_{B^{\frac34,\infty}_2}\|R\|_{B^{\frac12,\infty}_2}\|R\|_{B^{\frac34,\infty}_2}\lesssim_A AB,
\]
which completes the proof of \eqref{est:aTransport}.
\epf
\medskip

We now estimate some of the secondary auxiliary functions $d$ and $M$:

\begin{lemma}\label{l:d}
We have the estimate
\begin{equation}\label{est:d}
\|d\|_{\bmo}\lesssim_A B.
\end{equation}
\end{lemma}
\bpf
We recall that
\[
d = R_\alpha(1 - \bar Y).
\]
As a consequence it suffices to show that
\[
\|R_\alpha\bar Y\|_{\bmo}\lesssim AB.
\]

Decomposing using paraproducts we have
\[
R_\alpha\bar Y = T_{R_\alpha}\bar Y + T_{\bar Y}R_\alpha + \Pi[R_\alpha,\bar Y].
\]
We then use \eqref{Para3} to estimate
\[
\|T_{R_\alpha}\bar Y\|_{\bmo}\lesssim \|\langle D\rangle^{\frac12}R\|_{L^\infty}\|Y\|_{\bmo^{\frac12}},\qquad \|\Pi[R_\alpha,\bar Y]\|_{\bmo}\lesssim \|R_\alpha\|_{\bmo}\|Y\|_{\bmo}.
\]

For the remaining term we are unable to use \eqref{Para3}, but we can obtain a similar estimate by relaxing \(\bmo\) to \(L^\infty\) for the low frequency term (see \cite[Proposition 2.2]{HIT}),
\[
\|T_{\bar Y}R_\alpha\|_{\bmo}\lesssim \|Y\|_{L^\infty}\|R_\alpha\|_{\bmo}.
\]
The estimate \eqref{est:d} then follows.

\epf

\begin{lemma}\label{l:M}
The function $M$ satisfies the  pointwise bounds
\begin{equation}\label{M-infty}
\| M\|_{L^\infty} \lesssim AB,
\end{equation}
as well as the Sobolev bounds for \(n\geq1\)
\begin{equation}\label{M-L2}
\| M\|_{H^{k-\frac32}} \lesssim_A A\norm_k,\qquad \Vert M \Vert_{H^{k-1}}\lesssim_A g^{-\frac12}B\norm_k.
\end{equation}
\end{lemma}
\begin{proof}
We start with the proof of \eqref{M-infty}. We first decompose \(M = M_0 + M_{\geq1}\) into a low and high frequency part.

For the high frequency part we first write \(M\) in the form
\[
M = 2\Re [\P,R]\bar Y_\alpha - 2\Re[\P,Y]\bar R_\alpha,
\]
and then use \eqref{PComm} to obtain
\[
\|M_{\geq1}\|_{L^\infty}\lesssim \|\Til M\|_{L^\infty}\lesssim \|R\|_{B^{\frac34,\infty}_2}\|Y\|_{B^{\frac14,\infty}_2}\lesssim AB.
\]
For the low frequency part we face an additional difficulty compared 
to the infinite depth case, which is due to the low frequenct unboundedness of the projector $\P$. To address this 
we observe that \(M\) has a certain null structure, by writing
\[
M = 2\Re\P[R\bar Y_\alpha - Y\bar R_\alpha] = \Re[R\bar Y_\alpha - Y\bar R_\alpha] - \Til^{-1}\partial_\alpha\Im(R\bar Y).
\]
Applying the projection \(S_0\) we obtain
\[
\|M_0\|_{L^\infty}\lesssim \|\Pi[R,\bar Y_\alpha]\|_{L^\infty} + \|\Pi[Y,\bar R_\alpha]\|_{L^\infty} + \|\Pi[R,\bar Y]\|_{L^\infty}.
\]
We may then estimate these terms using \eqref{Para2} to complete the proof of \eqref{M-infty}. The proof of \eqref{M-L2} is similar.
\end{proof}

\begin{lemma}\label{l:dt-rw}
The following estimates hold:
\begin{equation}
\| \Poly{\geq 2}  (\partial_t + T_b \partial_\alpha) \W\|_{B^{0,\infty}_\infty} 
+g^{-\frac12} \| \Poly{\geq 2}  (\partial_t + T_b \partial_\alpha) R\|_{B^{\frac12,\infty}_\infty} 
\lesssim_A AB ,
\end{equation}
respectively the $L^2$ bounds 
\begin{equation}
g^{-\frac12} \| \Poly{\geq 2}  (\partial_t + T_b \partial_\alpha) R\|_{H^{n-1}} 
\lesssim_A A \norm_n, \qquad n \geq 1, 
\end{equation}
and
\begin{equation}
\| \Poly{\geq 2}  (\partial_t + T_b \partial_\alpha) \W\|_{H^{n-\frac32}} 
\lesssim_A A \norm_n, \qquad n \geq 2. 
\end{equation}
If instead $n = 1$ then for each $k$ there is a decomposition
\[
P_{<k} \Poly{\geq 2}  (\partial_t + T_b \partial_\alpha) \W = F^1_k + F^2_k,
\]
so that 
\begin{equation}
\| F^1_k\|_{L^2} \lesssim_A B \norm_1, \qquad \| F^2_k\|_{L^2} \lesssim_A 2^{\frac{k}2} A \norm_1.
\end{equation}
\end{lemma}
\begin{proof}
We recall the equations for $(\W,R)$:
\[
 \begin{cases}
    \W_t + b\W_\alpha + \dfrac{1+\W}{1+\bar \W}R_\alpha = (1+\W)M\vspace{0.1cm}\\
    R_t + bR_\alpha = i\dfrac{g\W - \mfa}{1+\W}.
  \end{cases}
\]
We begin with the pointwise bounds. For the  $M$ term we use \eqref{M-infty}. Next we 
estimate
\[
\left\| \dfrac{\W -\bar \W}{1+\bar \W}R_\alpha \right\|_{\bmo} \lesssim \| R_\alpha\|_{\bmo} \left\| \dfrac{\W-\bar \W}{1+\bar \W} \right\|_{L^\infty}
+ \| R\|_{B^{\frac34,\infty}_2} \left\| \dfrac{\W-\bar \W}{1+\bar \W} \right\|_{B^{\frac14,\infty}_2} \lesssim AB,
\]
which is akin to the $\bmo$ bound for $d$.
For the $Y$ term in the second equation we use \eqref{est:Ybmo} as well as the algebra property for $\bmo^\frac12$.
The same applies for the $a$ term in combination with \eqref{est:a}.

It remains to bound the $b$ terms, where we carefully note that no low frequencies of $b$ are included here.
Then using   \eqref{est:bBMO} we have 
\[
\| (b-T_b)\W_\alpha\|_{L^\infty} \lesssim \| \Til b \|_{B^{\frac34,\infty}_2} \| \W \|_{B^{\frac14,\infty}_2} \lesssim AB,
\]
respectively 
\[
\| (b-T_b)R_\alpha\|_{\bmo^\frac12} \lesssim \| \Til b \|_{B^{\frac34,\infty}_2} \| R \|_{B^{\frac34,\infty}_2} \lesssim AB.
\]

Next we consider the $L^2$ bounds. For the $M$ term we use a standard Littlewood-Paley decomposition 
together with \eqref{M-infty} and \eqref{M-L2}. For the $a$ term we similarly use \eqref{est:a} and \eqref{est:a2}.
For the $b$ paradifferential remainder we use \eqref{est:bBMO} and \eqref{est:bL2}. The other terms follow 
in standard  bilinear fashion. 

In the case $n = 1$ the same method applies once we have produced 
a convenient decomposition of $(\partial_t + T_b \partial_\alpha) \W$. 
Precisely, all contributions go to $F^1_k$ except for those 
arising from the terms $(b_{< k+4} - T_b) \W_\alpha$, respectively
$\dfrac{1+\W}{1+\bar\W} R_{<k+4,\alpha}$.

\end{proof}

\end{appendix}

\section*{Acknowledgments}

This material is based in part upon work supported by the NSF under grant DMS-1440140 while the authors were in residence at the MSRI in Berkeley, California during the Fall 2015 semester. The last author was
also supported by the NSF grant DMS-1266182. In addition, all three authors were supported in part by the Simons Foundation.

\bibliographystyle{abbrv}
\bibliography{refs}

\end{document}